\documentclass[a4paper]{amsart}
\usepackage{amsmath}
\usepackage{amssymb}
\usepackage{amsfonts}
\usepackage{pxfonts}
\usepackage[OT2,T1]{fontenc}

\usepackage{pict2e}

\usepackage[
textwidth=14.5cm,
textheight=21cm,
hmarginratio=1:1,
vmarginratio=1:1]{geometry}
\usepackage{pifont}
\usepackage{graphicx}

\newtheorem{thm}{Theorem}[subsection]
\newtheorem{cor}[thm]{Corollary}
\newtheorem{lem}[thm]{Lemma}
\newtheorem{prop}[thm]{Proposition}
\theoremstyle{definition}
\newtheorem{defn}[thm]{Definition}

\theoremstyle{remark}
\newtheorem{rem}[thm]{Remark}

\numberwithin{equation}{subsection}
\numberwithin{figure}{section}

\newcommand{\diff}{\mathrm{d}}
\newcommand{\C}{{\mathbb C}}
\newcommand{\R}{{\mathbb R}}
\newcommand{\D}{{\mathbb D}}
\newcommand{\Te}{{\mathbb T}}

\newcommand{\imag}{\mathrm{i}}
\newcommand{\e}{\mathrm{e}}

\newcommand{\Z}{{\mathbb Z}}

\newcommand{\supp}{\operatorname{supp}}
\newcommand{\Ordo}{\mathrm{O}}

\newcommand{\Qop}{{\mathbf Q}}

\newcommand{\Lop}{{\mathbf L}}
\newcommand{\Tope}{{\mathbf T}}
\newcommand{\Topep}{\pmb{\mathcal{T}}}
\newcommand{\Hop}{{\mathbf H}}
\newcommand{\Jop}{{\mathbf J}}

\newcommand{\Perop}{\boldsymbol\Pi}
\newcommand{\Superop}{\boldsymbol\Sigma}
\newcommand{\Superoprem}{\Superop_2^{\text{\ding{172}}}}
\newcommand{\SuperopremN}{\Superop_{2,N}^{\text{\ding{172}}}}
\newcommand{\bTheta}{\boldsymbol\Theta}
\newcommand{\Rop}{\mathbf R}

\newcommand{\Vop}{\mathbf V}

\newcommand{\Lspaceo}{{\mathfrak{L}}(\R)}
\newcommand{\LspaceI}{{\mathfrak{L}}(I)}
\newcommand{\LspaceIg}{{\mathfrak{L}}(I_\gamma)}
\newcommand{\LspaceIbg}{{\mathfrak{L}}(I_{\beta/\gamma})}
\newcommand{\LspaceIone}{{\mathfrak{L}}(I_{1})}
\newcommand{\LspaceIoneodd}{{\mathfrak{L}}_{\mathrm{odd}}(I_{1})}
\newcommand{\LspaceIgcompl}{{\mathfrak{L}}(\R\setminus\bar I_\gamma)}
\newcommand{\LspaceIbgcompl}{{\mathfrak{L}}(\R\setminus\bar I_{\beta/\gamma})}
\newcommand{\LspaceIonecompl}{{\mathfrak{L}}(\R\setminus\bar I_{1})}

\newcommand{\Lspaces}{{\mathfrak{L}}_0(\R)}
\newcommand{\Lspaceoper}{{\mathfrak{L}}(\R/2\Z)}
\newcommand{\Lspacesper}{{\mathfrak{L}}_0(\R/2\Z)}

\newcommand{\stars}{\circledast}
\newcommand{\kfun}{q}
\newcommand{\Kfun}{Q}
\newcommand{\Powerseriesclass}{\mathfrak{P}(\gamma)}
\newcommand{\PowerseriesclassR}{\mathfrak{P}_\R(\gamma)}
\newcommand{\PowerseriesclassRdown}{\mathfrak{P}_\R^\downarrow(\gamma)}
\newcommand{\PowerseriesclassII}{\mathfrak{P}(2-t)}

\newcommand{\PowerseriesclassIIRdown}{\mathfrak{P}_\R^\downarrow(2-t)}
\newcommand{\PowerseriesclassIII}{\mathfrak{P}(2-\tfrac{1}{\gamma})}
\newcommand{\PowerseriesclassIIIRdown}
{\mathfrak{P}_\R^\downarrow(2-\tfrac{1}{\gamma})}
\newcommand{\PowerseriesclassIVRdown}
{\mathfrak{P}_\R^\downarrow(2-\tfrac{1}{2-t})}
\newcommand{\PowerseriesclassVRdownj}
{\mathfrak{P}_\R^\downarrow(\gamma_j(t))}

\newcommand{\proj}{{\mathbf P}}
\newcommand{\pev}{\mathrm{vp}}
\newcommand{\almostev}{\mathrm{a.e.}}

\newcommand{\pv}{\operatorname{pv}}

\newcommand{\id}{\operatorname{\text{\bf I}}}

\newcommand{\calE}{{\mathcal E}}

\DeclareMathOperator{\re}{Re}
\DeclareMathOperator{\im}{Im}

\begin{document}

\title{An extension of ergodic theory for Gauss-type maps}

\author{Haakan Hedenmalm}
\address{Hedenmalm: Department of Mathematics\\
KTH Royal Institute of Technology\\
S--10044 Stockholm\\
Sweden}

\email{haakanh@kth.se}

\author{Alfonso Montes-Rodr\'\i{}guez}
\address{Montes-Rodr\'\i{}guez: Department of Mathematical Analysis\\
University of Sevilla\\
Sevilla\\
Spain}

\email{amontes@us.es}


\thanks{The research of Hedenmalm was supported by Vetenskapsr\aa{}det (VR).
The research of both authors was supported
by Plan Nacional ref. MTM2012-35107 and by Junta Andaluc\'\i{}a P12-FQM-633,
FQM-260.}

%\maketitle
$\quad$

\begin{abstract}
The impetus to this work is the need to show that for positive reals $\alpha$ 
and $\beta$, the functions
\[
\e^{\imag\pi\alpha m t},\quad \e^{-\imag\pi\beta n/t},\qquad m,n\in\Z_+\cup\{0\},
\]
span a weak-star dense subspace of $H^\infty_+(\R)$ if and only if
$0<\alpha\beta\le1$. Here, $H^\infty_+(\R)$ is the subspace of $L^\infty(\R)$
which consists of those functions whose Poisson extensions to the upper
half-plane are holomorphic. In earlier work in the $L^\infty(\R)$ context,
we showed the relevance of the analysis of the dynamics of the Gauss-type 
mapping $x\mapsto -\beta/x$ mod $2\Z$ for this problem (if $\alpha=1$, 
which can be assumed by a scaling argument). For $\beta=1$, the ergodic 
properties of the absolutely continuous invariant measure $(1-x^2)^{-1}\diff x$ 
on the interval $I_1=]\!-\!1,1[$ turned out to be crucial. In the present
setting, although the norm in $H^\infty_+(\R)$ is the same as in $L^\infty(\R)$,
in the real sense, it is much finer. The corresponding real space is   
$H^\infty_\stars(\R)$, which consists of all the functions in $L^\infty(\R)$ whose 
modified Hilbert transform is also in $L^\infty(\R)$. From the real perspective,
our task is clearer: We need to show that the functions
\[
\e^{\imag\pi m t},\quad \e^{-\imag\pi\beta n/t},\qquad m,n\in\Z,
\]
span a weak-star dense subspace in $H^\infty_\stars(\R)$ precisely when 
$0<\beta\le1$. The predual of $H^\infty_\stars(\R)$ is identified with a space
$\Lspaceo$ of distributions on $\R$, obtained as the sum of $L^1(\R)$ and 
$\Hop L^1_0(\R)$, where $L^1_0(\R)$ is the codimension $1$ subspace of 
$L^1(\R)$ of functions with integral $0$. While the space $\Lspaceo$ consists
of distributions, it also can be said to consist of weak-$L^1$ functions, and
a theorem of Kolmogorov guarantees that the viewpoints are equivalent.
It is in a sense the extension of $L^1(\R)$ which is analogous to having 
BMO$(\R)$ as the extension of $L^\infty(\R)$.  
Whereas transfer (and subtransfer) operators usually act on $L^1$ of an 
interval (or, more generally, on the finite Borel measures on that interval), 
here we consider the corresponding operators acting on the restriction of 
$\Lspaceo$ to the interval $I$ in question, denoted $\LspaceI$.
In the convex body of invariant absolutely continuous probability measures
an element is ergodic if it is an extreme point. In the setting of infinite 
ergodic theory, which is more relevant here, ergodicity means that no 
element of $L^1$ on the interval can be invariant (under the transformation, 
or, which is the same, under the transfer operator). We study mainly a 
particular instance of infinite ergodic theory, and extend the concept of
ergodicity by showing that for the transformation $x\mapsto-\beta/x$ mod 
$2\Z$ on $I_1$, (i) for $0<\beta<1$, there is no nontrivial subtransfer 
operator invariant distribution in $\LspaceIone$, whereas (ii) for $\beta=1$, 
there is no nontrivial transfer operator invariant \emph{odd} distribution 
in $\LspaceIone$. The oddness helps in the proof, but we expect it to be
superfluous. The conclusion is nevertheless strong enough to supply an 
affirmative answer to our original density problem. 
To obtain the results (i)-(ii), we develop new tools,
%In the critical regime $\alpha\beta=1$, the proof
%relies on the nonexistence of a certain invariant distribution 
%in the  predual of real $H^\infty$ 
%for the above-mentioned Gauss-type map on
%the interval $]\!-\!1,1[$, which is a new result of dynamical flavor. To
%obtain it, we develop new tools,
% for the analysis of Gauss-type maps, related to the appearance of 
%the Hilbert transform. 
which offer a novel amalgam of 
%involving transfer operators for Gauss-type maps in a novel 
%setting, where 
ideas from Ergodic Theory with ideas from Harmonic Analysis. 
We need to handle in a subtle way series of 
powers of transfer operators, a rather intractable problem where
even the recent advances by Melbourne and Terhesiu do not apply. More
specifically, our approach involves a splitting of
the Hilbert kernel induced by the transfer operator. The careful
analysis of this splitting involves detours to the Hurwitz zeta function
as well as to the theory of totally positive matrices.

\end{abstract}

%%% ----------------------------------------------------------------------
\maketitle
%%% ----------------------------------------------------------------------

%\addtolength{\oddsidemargin}{-1.3cm}
%\addtolength{\evensidemargin}{-1.3cm}
%\addtolength{\textwidth}{1.5cm}
\addtolength{\textheight}{2.2cm}
%\addtolength{\topmargin}{-1.0cm}
%\addtolength{\footskip}{1.5cm}

%\begin{document}

%%% Research support

%\subjclass{}

%\keywords{}

%%% Section numbering
%\setcounter{section}{-1}

%\setcounter{equation}{0}
%\setcounter{thm}{0}
%\setcounter{prop}{0}
%\setcounter{lemma}{0}
%\setcounter{cor}{0}
%\setcounter{remark}{0}

\section{Introduction}

\subsection{An elementary example: the doubling map of
an interval}
\label{subsec-doubling}

Let us consider the doubling map of the unit interval $\bar I_1^+:=[0,1]$, 
given by $\theta(t):=2t$ mod $\Z$; to be more precise, we put $\theta(t)=2t$
on $[0,\frac12]$, and $\theta(t)=2t-1$ on $]\frac12,1]$.  
For $h\in L^\infty(I_1^+)$ and $g\in L^1(I_1^+)$, we have the identity
\[
\int_0^1 h\circ\theta(t)\,g(t)\,\diff t=\int_0^1h(t)\,\bTheta g(t)\,\diff t,
\] 
where $\bTheta$ is the associated \emph{transfer operator}
\[
\bTheta g(t):=\frac12\bigg(g\bigg(\frac{t}{2}\bigg)+g\bigg(\frac{t+1}{2}
\bigg)\bigg),
\qquad t\in I_1^+.
\]
The function $g\in L^1(I_1^+)$ (and the corresponding absolutely continuous 
measure $g(t)\diff t$) is said to be \emph{invariant} with respect to the 
doubling map $\theta$ if $\bTheta g=g$. We quickly check that the constant 
function $g_0(t)\equiv1$ is invariant, and wonder if there are any other 
invariant functions beyond the scalar multiples of $g_0=1$. To analyze
the situation, Fourier analysis comes very handy. We expand the function 
$g\in L^1(I_1^+)$ in a Fourier series
\[
g(t)\sim\sum_{j=-\infty}^{+\infty}\hat g(j)\,\e^{\imag 2\pi jt},\qquad t\in I_1^+,
\]
which actually need not converge pointwise, but this does not bother us. 
The Fourier series associated with $\bTheta g$ is then
\[
\bTheta g(t)\sim\frac12\sum_{j=-\infty}^{+\infty}\hat g(j)\,
(\e^{\imag \pi jt}+\e^{\imag\pi j(t+1)})=\frac12\sum_{j=-\infty}^{+\infty}\hat g(j)\,
(1+(-1)^j)\e^{\imag \pi jt}=\sum_{k=-\infty}^{+\infty}\hat g(2k)\,
\e^{\imag 2\pi kt},
\]
and, by iteration,
\[
\bTheta^n g(t)\sim\sum_{k=-\infty}^{+\infty}\hat g(2^n k)\,
\e^{\imag 2\pi kt},\qquad n=0,1,2,\ldots.
\]
If $g$ solves the more general eigenvalue problem $\bTheta g=\lambda g$
for some complex nonzero scalar $\lambda\in\C^\times:=\C\setminus\{0\}$, 
then we see by equating Fourier coefficients that we must have
\begin{equation}
\hat g(k)=\lambda^{-n}\hat g(2^nk),\qquad k\in\Z,
\label{eq-invariant1.011}
\end{equation}
for $n=0,1,2,\ldots$. By plugging in $k=0$, we derive from the above equation
\eqref{eq-invariant1.011} that $\lambda=1$ is the only possibility, provided
that $\hat g(0)\ne0$. 
Moreover, for $k\in \Z^\times:=\Z\setminus\{0\}$, we know from the 
Riemann-Lebesgue lemma that
\[
\hat g(2^nk)\to0\quad\text{as}\quad n\to+\infty,
\]
which lets us to conclude from \eqref{eq-invariant1.011} that
\[
\hat g(k)=0,\qquad k\in\Z^\times,
\]
provided that $|\lambda|\ge1$. In this case, $g$ is of course constant, 
and if the constant is nonzero, then we also know that $\lambda=1$.
In particular, the only invariant functions in $L^1(I_1^+)$ are the constants.
This observation is an equivalent reformulation of the well-known ergodicity
of the doubling map with respect to the uniform measure on the interval $I_1^+$
(see below). 
%implies that the doubling map is ergodic, because
%the constant function $1$ cannot be split as a convex combination of two 
%\emph{other} invariant positive functions with integral $1$. 

{\sc Observation:} 
As we look back at the argument just presented, we realize that we did not
use all that much about the function $g$, just that the conclusion of the
Riemann-Lebesgue lemma holds. So in principle, we could replace $g$ by a
finite Borel measure, and obtain the same conclusion, if the Fourier 
coefficients of the measure tend to $0$ at infinity. Such measures are known
as \emph{Rajchman measures}, and have been studied in depth in harmonic 
analysis. But the point of view we want to present here goes beyond that
setting. \emph{We are in fact at liberty to replace $g$ by a distribution with
a periodic extension}, so that it has a Fourier series expansion, and so long
as its Fourier coefficients $\hat g(j)$ tend to $0$ as $|j|\to+\infty$, 
the argument works, and tells us that the constants are the unique 
$\bTheta$-invariant elements of this much wider space of distributions. 
Such periodic distributions $g$ which Fourier coefficients $\hat g(j)$ 
which tend to $0$ as $|j|\to+\infty$ deserve to be called 
\emph{Rajchman pseudomeasures} (cf. \cite{Katzbook}).  
This uniqueness within the Rajchman pseudomeasures can be understood as 
\emph{an extension of standard ergodic theory for the doubling map with 
respect to the constant density $1$}. Indeed, 
%since ergodicity for an 
%invariant probability measure is equivalent to being an extreme point in 
%the convex body of all invariant probability measures, 
an easy argument shows that the following are equivalent, for an invariant 
probability measure $\mu$:

(i) $\mu$ is ergodic, and

(ii) whenever $\nu$ is a finite (signed) invariant measure, absolutely 
continuous with respect to $\mu$, then $\nu$ is a scalar multiple of $\mu$.

This is probably well-known. For completeness, we supply the relevant
argument. Note first that we may restrict to real measures and real scalars 
in (ii). The implication (i)$\implies$(ii) is pretty standard and runs 
as follows. 
By replacing $\nu$ by the sum of $\nu$ and an appropriate scalar multiple 
of $\mu$, we reduce to the case when $\nu$ has signed mass $0$. Then, unless 
$\nu=0$, we split $\nu$ into positive and negative parts, which are seen to 
be left invariant by the transfer operator, as otherwise the transfer 
operator applied to $\nu$ would have smaller total variation than $\nu$ 
itself. But then the support (or rather, carrier) sets for the positive and 
negative parts are necessarily invariant under the 
transformation, in violation of the ergodic assumption (i), and the only
remaining alternative is that $\nu=0$, i.e., the assertion (ii) holds.
The remaining implication (ii)$\implies$(i) is even simpler. 
We prove the contrapositive implication, and assume that (i) fails, so
that $\mu$ is not ergodic. Then $\mu$ is not an extreme point in the 
convex body of all invariant probability measures, and hence it
splits as a nontrivial convex combination of two invariant probability 
measures. Both measures are assumed different than $\mu$ itself, and each 
one is obviously absolutely continuous with respect to $\mu$, which shows 
that (ii) fails as well.

\subsection{The Gauss-type maps on the symmetric unit 
interval} 
\label{subsec-gausstype}

It was the fact that the doubling map is piecewise affine that made it amenable
to methods from Fourier analysis. This is not the case for the Gauss-type map
$\tau_\beta$ acting on the symmetric interval $I_1:=]\!-\!1,1[$, defined in 
the following fashion. 
First, we let $\{x\}_2$ denote the \emph{even-fractional part of $x$},
by which we mean the unique number in the half-open interval
$\tilde I_1:=]\!-\!1,1]$ with $x-\{x\}_2\in2\Z$. The Gauss-type map 
$\tau_\beta:\tilde I_1\to \tilde I_1$ is given by the expression
\[
\tau_\beta(x):=\bigg\{-\frac{\beta}{x}\bigg\}_2.
\]
Here and in the sequel, \emph{$\beta$ is assumed real with $0<\beta\le1$}. 
The basic properties of $\tau_\beta$ are well-known, see, e.g. \cite{HMcpam}.
We outline the basic aspects below, which are mainly based on the work of 
Thaler \cite{Thaler} and Lin \cite{Lin}. 
For $0<\beta<1$, the 
set $I_1\setminus\bar I_\beta$ acts as an \emph{attractor} for the iterates under
$\tau_\beta$, and inside the attractor $I_1\setminus\bar I_\beta$, the orbits form
$2$-cycles. Here, $\bar I_\beta$ denotes the symmetric interval 
$\bar I_\beta:=[-\beta,\beta]$. For $\beta=1$, on the other hand, we are in the
setting of infinite ergodic theory, where $(1-x^2)^{-1}\diff x$ is the ergodic
invariant measure. The reason is that the endpoint $1$ (which for all
essential purposes may be identified with the left-hand endpoint $-1$ for
the dynamics) is only weakly repelling. The \emph{tranfer operator} 
$\Topep_\beta$ linked with the map $\tau_\beta$ is the operator which can be 
understood as taking the unit point mass $\delta_x$ at a point 
$x\in\tilde I_1$ to the unit point mass $\delta_{\tau_\beta(x)}$ at the point 
$\tau_\beta(x)$. To be more definitive, for a function $f\in L^1(I_1)$, we 
write $f$ as an integral of point masses,
\begin{equation}
f(x)=\int_{I_1}f(t)\,\delta_x(t)\,\diff t=\int_{I_1}f(t)\,\delta_t(x)\,\diff t,
\label{eq-deltadecomp1.001}
\end{equation}
understood in the sense of distribution theory, and say that 
\begin{equation}
\Topep_\beta f(x):=\int_{I_1}f(t)\,\Topep_\beta\delta_t(x)\,\diff t=
\int_{I_1}f(t)\,\delta_{\tau_\beta(t)}(x)\,\diff t,\qquad x\in I_1,
\label{eq-Topepdef1}
\end{equation}
which is seen to be the same as the more explicit formula
%For $0<\beta\le1$ and a function $f\in L^1(I_1)$, extended
%to vanish on $\R\setminus I_1$, we let $\Topep_\beta f$ denote the function
%defined by
\begin{equation}
\Topep_\beta f(x)=
\begin{cases}
\displaystyle\sum_{j\in\Z}\frac{\beta}{(x+2j)^2}\,f\bigg(-
\frac{\beta}{x+2j}\bigg),\qquad x\in I_1,
\\
\,\,0,\qquad\qquad\qquad\qquad\qquad\quad\,\, x\in\R\setminus I_1,
\end{cases}
\label{eq-Sop1.002'}
\end{equation}
which has the added advantage that the values off the interval $I_1$ are 
declared to vanish. The behavior of $\tau_\beta$ is rather uninteresting on the
attractor $I_1\setminus\bar I_\beta$, and for this reason, we introduce the
\emph{subtransfer operator} $\Tope_\beta$ which discardsthe  point masses from 
the attractor. In other words, we put 
\begin{equation}
\Tope_\beta f(x):=\Topep(1_{\bar I_\beta}f)(x)=\int_{\bar I_\beta}
f(t)\,\Topep_\beta\delta_t(x)\,\diff t=\int_{\bar I_\beta}f(t)\,
\delta_{\tau_\beta(t)}(x)\,\diff t,\qquad x\in I_1.
\label{eq-Topedef1}
\end{equation}
In more direct terms, this is the same as
\begin{equation}
\Tope_\beta f(x):=\sum_{j\in \Z^\times}\frac{\beta}{(2j+x)^2}
f\bigg(-\frac{\beta}{2j+x}\bigg),\qquad x\in I_1,
\label{eq-Tope.1}
\end{equation}
which we see from \eqref{eq-Sop1.002'}. Here, $\Z^\times=\Z\setminus\{0\}$, as
before. For $0<\beta<1$, the $\tau_\beta$-orbit of a point $x\in I_1$ falls
into the attractor $I_1\setminus \bar I_\beta$ almost surely. In terms of
the subtransfer operator $\Tope_\beta$, this means that 
\begin{equation}
\forall f\in L^1(I_1):\quad \Tope_\beta^n f\to0\quad\text{in}\quad L^1(I_1),
\quad\text{if}\quad 0<\beta<1.
\label{eq-exact001}
\end{equation}
For $\beta=1$, things are a little more subtle. Nevertheless, it can be 
shown that 
\begin{equation}
\forall f\in L^1(I_1):\quad 1_{I_\eta}\Tope_1^n f\to0\quad\text{in}\quad L^1(I_1),
\label{eq-exact002}
\end{equation}
for every fixed real $\eta$ with $0<\eta<1$. Here, as expected, $I_\eta$ is
the symmetric interval $I_\eta:=]\!-\!\eta,\eta[$. In particular, there is
no nontrivial function $f\in L^1(I_1)$ with $\Tope_\beta f=\lambda f$ for any 
$\lambda\in\C$ with $|\lambda|\ge1$ and any $\beta$ with $0<\beta\le1$.

In \cite{HMcpam}, the subtransfer operator $\Tope_\beta$ was shown to extend to
a bounded operator on the space $\LspaceIone$, whose elements are 
distributions on $I_1$. The space $\LspaceIone$ consists of the restrictions 
to the open interval $I_1$ of the distributions in the space 
\[
\Lspaceo:= L^1(\R)+\Hop L^1_0(\R),
\]
supplied with the induced quotient norm, as we mod out with respect to all 
the distributions whose support is contained in $\R\setminus I_1$. 
The quotient norm comes from the norm on the space $\Lspaceo$, which is 
given by  
\begin{equation}
\|u\|_{\Lspaceo}:=\inf\big\{\|f\|_{L^1(\R)}+\|g\|_{L^1(\R)}:\,\,
u=f+\Hop g,\,\, f\in L^1(\R),\,\,g\in L^1_0(\R)\big\},
\label{eq-normLspaceo}
\end{equation}
and we should mention that the $\Lspaceo$ is in the natural sense 
\emph{the predual of the real $H^\infty$-space on the line, 
denoted by $H^\infty_\stars(\R)$}, which consists of all the functions in 
$L^\infty(\R)$ whose modified Hilbert transform also is in $L^\infty(\R)$. 
In the definition of the space $\Lspaceo$, the letter $\Hop$ stands for the 
\emph{Hilbert transform}, given by the principal value integral
\begin{equation*}
\Hop g(x):=\frac{1}{\pi}\mathrm{pv}\int_\R g(t)\frac{\diff t}{x-t}=
\lim_{\epsilon\to0^+}\frac{1}{\pi}
\int_{\R\setminus[x-\epsilon,x+\epsilon]} g(t)\frac{\diff t}{x-t},
%\label{eq-Hilbert02}
\end{equation*}
and $L^1_0(\R)$ is the codimension $1$ subspace
\[
L^1_0(\R):=\Bigg\{g\in L^1(\R):\,\,\int_\R g(t)\diff t=0\Bigg\}.
\]
By a theorem of Kolmogorov, the Hilbert transform of an $L^1(\R)$ is 
well-defined pointwise almost everywhere as a function in the quasi-Banach 
space $L^{1,\infty}(\R)$ of weak-$L^1$ functions. More generally, if 
$E\subset\R$ is Lebesgue measurable with positive length, the 
\emph{weak-$L^1$ space} $L^{1,\infty}(E)$ consists of all measurable functions 
$f:E\to\C$ with finite quasinorm
\begin{equation}
\|f\|_{L^{1,\infty}(E)}:=\sup\big\{\lambda|N_f(\lambda)|:\,\lambda>0\big\},
\label{eq-quasinorm1.0099}
\end{equation}
where $N_f(\lambda)$ denotes the set 
\[
N_f(\lambda):= \{t\in E:\,|f(t)|>\lambda\},
\]  
and the absolute value sign assigns the linear length to given set. 
 Kolmogorov's theorem allows us to
think of the distributions (or pseudomeasures) in $\Lspaceo$ as elements of
$L^{1,\infty}(\R)$, so that in particular, $\LspaceIone$ can be identified with
a subspace of $L^{1,\infty}(I_1)$, the corresponding weak-$L^1$ space on the 
interval $I_1$. For the pointwise interpretation, the formula \eqref{eq-Tope.1}
for the operator $\Tope_\beta$ remains valid. We will work mainly in the setting
of distribution theory. When we need to speak of the pointwise function 
rather than the distribution $u$, we write $\pev(u)$ in place of $u$, and 
call it the \emph{valeur au point}. So ``$\pev$'' maps from distributions 
to functions. 
 
On the space $L^1(I_1)$, the subtransfer operators $\Tope_\beta$ all act
contractively. This is not the case with the extension to $\LspaceIone$.

\begin{thm}
Fix $0<\beta\le1$. Then the operator $\Tope_\beta:\,\LspaceIone\to\LspaceIone$ 
is bounded, but its norm exceeds $1$.
\label{thm-nocontract}
\end{thm} 

The proof of Theorem \ref{thm-nocontract} is supplied in Subsection 
\ref{subsec-normexp1}.

A decomposition analogous to \eqref{eq-deltadecomp1.001} holds for 
distributions $u\in\LspaceIone$ as well, only we would need two integrals, one 
with $\delta_t(x)$ and the other with $\Hop\delta_t(x)$ (and the latter 
integral should be taken over a bigger interval, e.g. $I_2=]\!-\!2,2[$ to
allow for tails). Thinking physically, we allow for two kinds of 
``particles'', focused particles $\delta_t$ as well as spread-out particles 
$\Hop\delta_t$. Then $\LspaceIone$ is a kind of state space, and $\Tope_\beta$
acts on this state space. It is then natural to ask whether there is a 
nontrivial invariant state under $\Tope_\beta$. More generally, we would ask
whether there exists a $u\in \LspaceIone$ with $\Tope_\beta u=\lambda u$
for any scalar $\lambda\in\C$ with $|\lambda|\ge1$. To appreciate the subtlety 
of this question, we note that in the slightly larger space $L^{1,\infty}(I_1)$, 
\emph{there are plenty of invariant states $u\in L^{1,\infty}(I_1)$ with 
$\Tope_\beta u=u$}, see the example provided in Remark \ref{example}.  
That example is constructed as the Hilbert transform of the difference of two 
Dirac point masses, with one point inside $I_1$ and the other point outside
$\bar I_1$. The example in fact suggests that within the space of 
Hilbert transforms of finite Borel measures, the invariant states
might possess an intricate and interesting structure. In the space $\Lspaceo$,
which contains the Hilbert transforms of the absolutely continuous measures, 
this is however not the case.

\begin{thm}
Fix $0<\beta<1$. 
For $u_0\in\LspaceIone$, we have the asymptotic decay 
$\pev(\Tope_\beta^N u_0)\to0$ in $L^{1,\infty}(I_1)$ as $N\to+\infty$.
\label{thm-basic1.001}
\end{thm}

So, although $\Tope_\beta$ has norm that exceeds $1$ on $\LspaceIone$, 
the orbit of a given $u\in\LspaceIone$ converges to $0$ in the weaker sense 
of the quasinorm in $L^{1,\infty}(I_1)$. In other words, the 
$L^{1,\infty}$-quasinorm serves as a \emph{Lyapunov energy} for the asymptotic 
stability of the $\Tope_\beta$-orbits. 
In the setting of the smaller space $L^1(I_1)$, this convergence amounts to the
statement that the basin of attraction of the attractor 
$I_1\setminus\bar I_\beta$ contains almost every point of the interval $I_1$.
Apparently, this property extends to the larger space $\LspaceIone$, but not
to e.g. $L^{1,\infty}(I_1)$ (see Remark \ref{example}).
The proof of Theorem \ref{thm-basic1.001} is supplied in Subsection 
\ref{subsec-thm-basic1.001}.

\begin{cor}
Fix $0<\beta<1$. 
If $\Tope_\beta u=\lambda u$ for some $u\in\LspaceIone$ and
some scalar $\lambda\in\C$ with $|\lambda|\ge1$, then $u=0$.
\label{eq-ergodic0.00001}
\end{cor} 

In other words, for $0<\beta$, the point spectrum of the operator 
$\Tope_\beta:\LspaceIone\to\LspaceIone$ is contained in the open unit disk 
$\D$. It is clear that Corollary \ref{eq-ergodic0.00001} follows from
Theorem \ref{thm-basic1.001}. 

Our understanding is slightly less complete for $\beta=1$.
We recall that a distribution, defined on a symmetric interval about $0$,
is \emph{odd} if its action on the even test functions equals $0$.

\begin{thm}
$(\beta=1)$
For odd $u_0\in\LspaceIone$, we have the asymptotic decay
$1_{I_\eta}\pev(\Tope_1^N u_0)\to0$ in $L^{1,\infty}(I_1)$ as $N\to+\infty$
for each $\eta$ with $0<\eta<1$.
\label{thm-basic1.002}
\end{thm}  

The proof, which is supplied in Subsection \ref{subsec-setup}, is 
\emph{much much more sophisticated} than that of Theorem 
\ref{thm-basic1.001}. It uses the full strength of the machinery developed 
around a subtle \emph{dynamical decomposition} of the odd part of the 
Hilbert kernel. We believe that a similar dynamical decomposition 
is available for the even part of the Hilbert kernel as well, which would 
remove the need for the oddness assumption. Again, the 
$L^{1,\infty}(I_\eta)$-quasinorms serve as Lyapunov energy functionals, for 
each $\eta$ with $0<\eta<1$.
In the setting of the smaller space $L^1(I_1)$, the corresponding statement 
is based on the fact that the dynamics of $\tau_1$ has $\pm1$ as a weakly 
repelling fixed point, so that the ergodic invariant measure for $L^1(I_1)$ 
get infinite mass and cannot be in $L^1(I_1)$.  It follows immediately from 
Theorem \ref{thm-basic1.002} that the point spectrum of the operator 
$\Tope_1:\LspaceIoneodd\to\LspaceIoneodd$ is contained in the open unit disk 
$\D$. In particular, there is no $\Tope_1$-invariant element of 
$\LspaceIoneodd$, the subspace of the odd distributions in $\LspaceIone$.   
% are no odd eigenfunctions 
%$u\in\Lspaceo$ to $\Tope_1$ with eigenvalue outside the open unit disk, 

\begin{cor}
$(\beta=1)$ 
If $\Tope_1 u=\lambda u$ for some odd $u\in\LspaceIone$ and
some scalar $\lambda\in\C$ with $|\lambda|\ge1$, then $u=0$.
\label{eq-ergodic0.00002}
\end{cor} 

As already mentioned, this corollary is an immediate consequence of 
Theorem \ref{thm-basic1.002}.

From a dynamical perspective, it is quite natural to introduce the odd-even 
symmetry, as the transformation $\tau_\beta$ itself is odd: 
$\tau_\beta(-x)=-\tau_\beta(x)$ (except possibly at the endpoints $\pm1$). 
E.g., in connection with the partial fraction expansions with even partial 
quotients, it is standard to keep track of only the orbit of the absolute
values on the interval $I_1^+$. 
Note that clearly, the subtransfer operators $\Tope_\beta$ preserve odd-even 
symmetry.
As for the remaining even symmetry case, we observe that 
$\Hop(\delta_{-1}-\delta_1)=\frac{2}{\pi}\pv(1-x^2)^{-1}$ which is even and
equals (a constant multiple of) the density of the ergodic invariant measure.
%, so we would need distinguish endpoint 
%delta masses from functions in $L^1(I_1)$. 

\begin{rem}
In view of the Observation in Subsection \ref{subsec-doubling}, 
Corollaries \ref{eq-ergodic0.00001} and \ref{eq-ergodic0.00002} go beyond
the standard notion of ergodicity. The main point is that we insert
\emph{distribution theory in place of measure theory}. We have not been able to
find any appropriate references for this in the literature, but suggest some
relevance of the works \cite{ADS1}, \cite{ADS2} for the discrete setting, and
\cite{Buf} for flows. 
\end{rem}

\subsection{Applications to the problem of completeness of a
system of unimodular functions}

As an application to Corollaries \ref{eq-ergodic0.00001} and 
\ref{eq-ergodic0.00002}, we have the following result on the completeness
of the nonnegative integer powers of two singular inner functions in the
weak-star topology of the space $H^\infty_+(\R)$ of functions which extend 
boundedly and holomorphically to the upper half-plane.    

\begin{thm}
Fix two positive reals $\alpha,\beta$.
Then the functions
\[
\e^{\imag \pi\alpha m t},\quad \e^{-\imag\pi \beta n/t},\qquad m,n=0,1,2,\ldots,
\]
which are elements of $H^\infty_+(\R)$, span together a weak-star dense
subspace of $H^\infty_+(\R)$ if and only if $\alpha\beta\le1$.
\label{thm-2.0}
\end{thm}

Note that the ``only if'' part of Theorem \ref{thm-2.0}, is quite simple, 
as for instance the work in \cite{CHM} shows that in case $\alpha\beta>1$, the
weak-star closure of the linear span in question has infinite codimension in
$H^\infty_+(\R)$. Hence the main thrust of the theorem is the ``if'' part.
The proof of Theorem \ref{thm-2.0} is supplied in two installments: for
$\alpha\beta<1$ in Subsection \ref{subsec-thm-2.0-beta<1}, and 
for $\alpha\beta=1$ in Subsection \ref{subsec-applasdecay=1}. 

A standard M\"obius mapping brings the upper half-plane to the the unit disk 
$\D$, and identifies the space $H^\infty_+(\R)$ with $H^\infty(\D)$, the space
of all bounded holomorphic functions on $\D$. For this reason, Theorem
\ref{thm-2.0} is equivalent to the following assertion, which we state as
a corollary.

\begin{cor}
Fix two positive reals $\lambda_1,\lambda_2$.
Then the linear span of the functions  
\[
\phi_1(z)^m=\exp\bigg(m\lambda_1\frac{z+1}{z-1}\bigg) \quad \text { and } 
\quad \phi_2(z)^n=\exp\bigg(n\lambda_2\frac{z-1}{z+1}\bigg),
\qquad
m,n=0,1,2,\ldots,
\]
is weak-star dense in $H^\infty(\D)$ if and only if 
$\lambda_1\lambda_2\le\pi^2$.
\label{cor-2appl}
\end{cor}

We suppress the trivial proof of the corollary.

\begin{rem}
Clearly, Theorem \ref{thm-2.0} supplies a complete and affirmative
answer to Problems 1 and 2 in \cite{MS}. We recall the question from \cite{MS}:
the issue was raised whether the algebra generated by the two inner functions
\[
\phi_1(z)=\exp\bigg(\lambda_1\frac{z+1}{z-1}\bigg) \quad \text { and } 
\quad \psi_2(z)=\exp\bigg(\lambda_2\frac{z-1}{z+1}\bigg)
\]
for $0<\lambda_1,\lambda_2<+\infty$, is weak-star dense in $H^\infty(\D)$ 
if and only if $\lambda_1\lambda_2\le\pi^2$. The ``only if'' was understood
already in \cite{MS}, while it is a consequence of Theorem  \ref{thm-2.0} 
that if $\lambda_1\lambda_2\le\pi^2$, then the linear span of the functions
\[
\phi_1(z)^m=\exp\left(m\lambda_1\frac{z+1}{z-1}\right) \quad 
\text { and } \quad 
\psi_2(z)^n=\exp\left(n\lambda_2\frac{z-1}{z+1}\right), 
\qquad m,n=0,1,2,\ldots, 
\]
is weak-star dense set in $H^\infty(\D)$, without the need to resort to 
the whole algebra.
\end{rem}

The $L^\infty(\R)$ analogue of Theorem \ref{thm-2.0} was obtained in \cite{HM}.
In the context of Theorem \ref{thm-2.0}, the $L^\infty(\R)$ result leads to 
completeness in the weak-star topology of $\mathrm{BMOA}_+(\R)$, the
BMOA space of the upper half-plane. The latter assertion is 
substantially weaker than Theorem \ref{thm-2.0}, as it is not difficult to 
exhibit a sequence of functions in $H^\infty_+(\R)$ which fails to be weak-star 
complete in $H^\infty_+(\R)$, but is weak-star complete in 
$\mathrm{BMOA}_+(\R)$.

\section{Basic properties of the dynamics of 
Gauss-type maps on intervals}
\label{sec-background.dynamics}

\subsection{Notation for intervals}
\label{subsec-notation-intervals}
For a positive real $\gamma$, let $I_\gamma:=]\!\!-\!\gamma,\gamma[$ denote
the corresponding symmetric open interval, and let $I_\gamma^+:=]0,\gamma[$
be the positive side of the interval $I_\gamma$. At times, we will need the
half-open intervals $\tilde I_\gamma:=]\!-\!\gamma,\gamma]$ and
$\tilde I_\gamma^+:=[0,\gamma[$, as well as the closed
intervals $\bar I_\gamma:=[-\gamma,\gamma]$ and $\bar I_\gamma^+:=[0,\gamma]$.

\subsection{Dual action notation}
\label{subsec-dualaction}
For a Lebesgue measurable subset $E$ of the real line $\R$,
we write
\[
\langle f,g\rangle_E:=\int_{E}f(t)g(t)\diff t,
\]
whenever $fg\in L^1(E)$. This will be of interest mainly when $E$ is an
open interval, and in this case, we use the same notation to describe the dual
action of a distribution on a test function.
For a set $E\subset\R$, $1_E$ stands for the characteristic function of $E$,
which equals $1$ on $E$ and vanishes elsewhere. So, in particular, we see that
\[
\langle f,g\rangle_E=\langle 1_E f, g\rangle_\R=\langle 1_E f, 1_E g\rangle_\R.
\]

\subsection{Gauss-type maps on intervals}
\label{subsec-Gausstype}
For background material in Ergodic Theory, we refer to e.g. the book \cite{CFS}.

For $N=2,3,4,\ldots$, the \emph{$N$-step wandering subset} is given by
\begin{equation}
%\begin{array}{r@{}l}
\calE_{\beta,N}:=\big\{x\in\bar I_\beta:\,\,\tau_\beta^{n}(x)\in \bar I_\beta
\,\,\,\text{for}\,\,\,n=1,\ldots,N-1\big\},
%\end{array}
\label{eq-EsetN}
\end{equation}
where $\tau_\beta^{n}:=\tau_\beta\circ\cdots\circ\tau_\beta$
($n$-fold composition). We also agree that $\calE_{\beta,1}:=\bar I_\beta$.
The sets $\calE_{\beta,N}$ get smaller as $N$ increases, and we form 
their intersections
\begin{equation}
\calE_{\beta,\infty}:=\bigcap_{N=1}^{+\infty}\calE_{\beta,N},
\label{eq-Esetinfty}
\end{equation}
The \emph{cone of positive functions} consists of all integrable functions
$f$ with $f\ge0$ a.e. on the respective interval. Similarly, we say 
that $f$ is \emph{positive} if $f\ge0$ a.e. on the given interval.

\begin{prop} Fix $0<\beta\le1$. Then we have the following assertions:

\textrm{\rm (i)} 
The operators $\Tope_\beta:L^1(I_1)\to L^1(I_1)$ and $\Topep_\beta:L^1(I_1)
\to L^1(I_1)$ are both norm contractions, which preserve the respective
cones of positive functions. 

\textrm{\rm (ii)} On the positive functions, $\Topep_\beta$ acts isometrically
with respect to the $L^1(I_1)$ norm. 

\textrm{\rm (iii)} If $\calE_{\beta,N}$ denotes the $N$-step 
wandering subset
given by \eqref{eq-EsetN} above, then 
$\Tope_\beta^Nf=\Topep_\beta^N(1_{\calE_{\beta,N}}f)$ for $f\in L^1(I_1)$ and 
$N=1,2,3,\ldots$.

\textrm{\rm (iv)} For $0<\beta<1$, and $f\in L^1(I_1)$, we have that 
$\|\Tope_\beta^N f\|_{L^1(I_1)}\to0$ as $N\to+\infty$. In particular, 
$|\calE_{\beta,N}|\to0$ as $N\to+\infty$. 
 
\textrm{\rm (v)} For $\beta=1$ and $f\in L^1(I_1)$ with mean
$\langle f,1\rangle_{I_1}=0$, we have that 
$\|\Tope_1^N f\|_{L^1(I_1)}\to0$ as $N\to+\infty$.

\textrm{\rm (vi)} For $\beta=1$ and $f\in L^1(I_1)$, we have that 
$\|1_{I_\eta}\Tope_1^N f\|_{L^1(I_1)}\to0$ as $N\to+\infty$ for each real $\eta$
with $0<\eta<1$.
\label{prop-ergodicity1.01}
\end{prop}
 
This is a conglomerate of ingredients from Propositions 3.4.1, 3.10.1, 
3.11.3, 3.13.1, 3.13.2, and 3.13.3 in \cite{HMcpam}.  

\subsection{An elementary observation extending the domain of definition for $\Tope_\beta$}
\label{subsec-3.5}
We begin with the following elementary observation.
\medskip

\noindent{\sc Observation.}
The subtransfer and transfer operators $\Tope_\beta$ and $\Topep_\beta$, 
initially defined on $L^1$ functions, make sense for wider classes of 
functions. Indeed, if $f\ge0$, then the formulae \eqref{eq-Sop1.002'} and 
\eqref{eq-Tope.1} make sense pointwise, with values in the extended 
nonnegative reals $[0,+\infty]$. More generally, if
$f$ is complex-valued, we may use the triangle inequality to dominate the
convergence of $\Tope_\beta f$ by that of $\Tope_\beta|f|$. This entails that
$\Tope_\beta f$ is well-defined a.e. if $\Tope_\beta|f|<+\infty$ holds a.e.
The same goes for $\Topep_\beta$ of course.
\medskip

This means that $\Tope_\beta f$ will be well-defined for many functions $f$,
not necessarily in $L^1(I_1)$.

\subsection{Symmetry preservation of the subtransfer operator
$\Tope_\beta$}

%The fact that the action of $\Tope_\beta$ commutes with the reflection in the
%origin will be needed. The precise formulation reads as follows. Let
%$\check\id$ be the antipodal operator $ \check\id f (x):=f(-x)$, which is its
%own inverse: $\check\id^2=\id$.

The property that $\Tope_\beta$ preserves symmetry on $L^1(I_1)$ holds much more
generally.

\begin{prop}
Fix $0<\beta\le1$.
To the extent that $\Tope_\beta f$ is well-defined pointwise, we have
the following:

\textrm{\rm (i)} If $f$ is odd, then $\Tope_\beta f$ is odd as well.

\textrm{\rm (ii)} If $f$ is even, then $\Tope_\beta f$ is even as well.
\label{lem-symmetry1}
\end{prop}

This follows from Proposition 3.6.1 in \cite{HMcpam}.

Along with the symmetry, we can add constraints like monotonicity and
convexity. Under such restraints on $f$, the pointwise values of 
$\Tope_\beta f$ are guaranteed to exist, and the constraint is preserved 
under $\Tope_\beta$. 

\begin{prop}
Fix $0<\beta\le1$.
We have the following:

\textrm{\rm (i)} If $f:I_1\to\R$ is odd and (strictly) increasing, then 
so is $\Tope_\beta f$.

\textrm{\rm (ii)} If $f:I_1\to\R$ is even and convex, and if $f\ge0$,
then so is $\Tope_\beta f$.
\label{lem-symmetry2}
\end{prop}

This follows from Propositions 3.7.1 and 3.7.2 in \cite{HMcpam}. 

\subsection{Preservation of point values of continuous 
functions under $\Tope_\beta$}

For $\gamma$ with $0<\gamma<+\infty$, let $C(\bar I_\gamma)$ denote the space
of continuous functions on the compact symmetric interval
$\bar I_\gamma=[-\gamma,\gamma]$. 
%The following observation is immediate 
%and hence its proof suppressed.

\begin{prop}
Fix $0<\beta\le1$.
If $f\in C(\bar I_\beta)$, then $\Tope_\beta f\in C(\bar I_1)$. Moreover, if
in addition, $f$ is odd, then $\Tope_\beta f(1)=\beta f(\beta)$.
\label{prop-3.8.2}
%endpointpres1.1}
\end{prop}

This result combines Propositions 3.8.1 and 3.8.2 in \cite{HMcpam}. 

%\begin{prop}
%Fix $0<\beta\le1$.
%If $f\in C(\bar I_\beta)$ is odd, then $\Tope_\beta f(1)=\beta f(\beta)$.
%\label{prop-3.8.2}
%\end{prop}

%\begin{proof}
%By \eqref{eq-Uop.Wop} and the assumption that $f$ is odd, cancellation of all
%terms except for the one corresponding to index $j=-1$ gives that
%\[
%\Tope_\beta f(1)=\sum_{j\in\Z^\times}\frac{\beta}{(2j+1)^2}
%f\bigg(-\frac{\beta}{2j+1}\bigg)=\beta f(\beta).
%\]
%The proof is complete.
%\end{proof}

%\end{document}

\subsection{Subinvariance of certain key functions}
\label{subsec-subinvariance}

Next, we consider the $\Tope_\beta$-iterates of the function 
%(for $0<\alpha\le1$)
\begin{equation}
\kappa_\alpha(x):=\frac{\alpha}{\alpha^2-x^2},\qquad x\in I_1,
\label {eq-kappadef1}
\end{equation}
where $\alpha$ is assumed confined to the interval $0<\alpha\le1$.
This function is not in $L^1(I_1)$, although it is in $L^{1,\infty}(I_1)$.
However, by the observation made in Subsection
\ref{subsec-3.5}, we may still calculate the expression
$\Tope_\beta \kappa_\alpha$ pointwise wherever
$\Tope_\beta|\kappa_\alpha|(x)<+\infty$.
Note that $\kappa_1(x)\diff x$ is the invariant measure for the transformation
$\tau_1(x)=\{-1/x\}_2$, which in terms of the transfer operator $\Tope_1$
means that $\Tope_1\kappa_1=\kappa_1$. 
%It is of fundamental importance in
%most of our considerations that this invariant measure has
%\emph{infinite mass}, i.e.,  that $\kappa_1\not\in L^1(I_1)$. The reason 
%for this is as mentioned previously that the transformation $\tau_1$ has 
%indifferent fixed points. 
%We should add that the control of the orbits is much more difficult and not 
%so well understood in the case of indifferent fixed points, in contrast 
%with the case of repelling fixed points when the theory is well developed.

\begin{lem} Fix $0<\beta\le1$.
For the function $\kappa_\beta(x)=\beta/(\beta^2-x^2)$, we have that
\[
\Tope_\beta\kappa_\beta(x)=\Tope_\beta|\kappa_\beta|(x)=
\kappa_1(x)=\frac{1}{1-x^2},\qquad \text{a.e.}\,\,\, x\in I_1,
\]
As for the function $\kappa_1(x)=(1-x^2)^{-1}$, we have the estimate
\[
0\le\Tope_\beta^n\kappa_1(x)\le\beta^n\,\kappa_1(x)=\frac{\beta^n}{1-x^2},
\qquad  x\in I_1,\,\,\,n=1,2,3,\ldots,
\]
which for $0<\beta<1$, may be replaced by the uniform estimate
\[
\Tope_\beta^n\kappa_1(x)\le\frac{2\beta^n}{1-\beta},\qquad x\in I_1
,\,\,\,n=1,2,3,\ldots.
\]
\label{prop-kappa1}
\end{lem}

\begin{rem}
As noted earlier, for $\beta=1$, we have the equality 
$\Tope_1\kappa_1=\kappa_1$.
\end{rem}

%We also obtain a uniform estimate of $\Tope_\beta^n\kappa_1$ for $0<\beta<1$
%and $n=1,2,3,\ldots$.

%\begin{prop}
% Fix $0<\beta\le1$.
%For $n=1,2,3,\ldots$, we have
%that
%\[
%\Tope_\beta^n\kappa_1(x)\le\frac{2\beta^n}{1-\beta},\qquad x\in I_1.
%\]
%\label{prop-Uop.iter}
%\end{prop}

\section{Background material: the Hilbert transform on the line and related spaces}

\subsection{The Szeg\H{o} projections and the Hardy 
$H^1$-space}
\label{subsec-H1anreal}
For a reference on the basic facts of Hardy spaces and BMO (bounded mean
oscillation), we refer to, e.g., the monographs of Duren and Garnett
\cite{Dur}, \cite{Gar}, as well as those of Stein \cite{Steinbook1},
\cite{Steinbook2}, and Stein and Weiss \cite{SteinWeissbook}.

\emph{Let $H^1_+(\R)$ and $H^1_-(\R)$ be the subspaces of $L^1(\R)$ consisting
of those functions whose Poisson extensions to the upper half plane
\[
\C_+:=\{z\in\C:\,\,\im z>0\}
\]
are holomorphic and conjugate-holomorphic, respectively}. Here, we use
the term conjugate-holo\-mor\-phic (or anti-holomorphic) to mean that the
complex conjugate of the function in question is holomor\-phic.

It is well-known that any function $f\in H^1_+(\R)$ has vanishing integral,
\begin{equation}
\langle f,1\rangle_\R=\int_\R f(t)\diff t=0,\qquad f\in H^1_+(\R).
\label{eq-int=0}
\end{equation}
In other words, $H^1_+(\C)\subset L^1_0(\R)$, where
\begin{equation}
L^1_0(\R):=\big\{f\in L^1(\R):\,\,\langle f,1\rangle_\R=0 \big\}.
\label{eq-defL0.11}
\end{equation}
%In fact, there is a related Fourier analysis characterization of the
%Hardy space $H^1_+(\R)$ and $H^1_-(\R)$: for $f\in L^1(\R)$,
%\begin{equation}
%f\in H^1_+(\R)\,\,\,\Longleftrightarrow \,\,\
%\forall y\ge0:\,\,\,\int_\R \e^{\imag yt}f(t)\diff t=0
%\label{eq-intcharH1}
%\end{equation}
%and
%\begin{equation}
%f\in H^1_-(\R)\,\,\,\Longleftrightarrow \,\,\
%\forall y\le0:\,\,\,\int_\R \e^{\imag yt}f(t)\diff t=0.
%\label{eq-intcharH1-2}
%\end{equation}
By a version of Liouville's theorem, 
\[
H^1_+(\R)\cap H^1_-(\R)=\{0\}, 
\]
which allows us to think of the space
\[
H^1_\stars(\R):=H^1_+(\R)\oplus H^1_-(\R)
\]
as a linear subspace of $L^1_0(\R)$. We will call $H^1_\stars(\R)$ the
\emph{real $H^1$-space of the line $\R$}, although it is $\C$-linear and
the elements are generally complex-valued. It is not difficult to show that
$H^1_\stars(\R)$ is norm dense as a subspace of $L^1_0(\R)$. 
The elements of $f\in H^1_\stars(\R)$ are just the functions
$f\in L^1_0(\R)$ which may be written in the form
\begin{equation}
f=f_1+f_2,\quad\text{ where}\,\,\,f_1\in H^1_+(\R),\,\,f_2\in H^1_-(\R).
\label{eq-H1decomp}
\end{equation}
%plus the fact that $H^1_+(\R)\cap H^1_-(\R)=\{0\}$, which is a Fourier-analytic
%consequence of \eqref{eq-intcharH1} and \eqref{eq-intcharH1-2}.
%Obviously, we have the inclusion $H^1_\stars(\R)\subset L^1_0(\R)$; it is
%perhaps slightly less obvious that $H^1_\stars(\R)$ is dense in $L^1_0(\R)$ in
%the norm of $L^1(\R)$. 
As already mentioned, the  decomposition \eqref{eq-H1decomp} is unique.
As for notation, \emph{we let $\proj_+$ and $\proj_-$ denote the
projections $\proj_+f:=f_1$ and $\proj_-f:=f_2$ in the decomposition}
\eqref{eq-H1decomp}.
%; they will be referred to as the Szeg\"o projections. 
These \emph{Szeg\H{o} projections} $\proj_+,\proj_-$ can of course be
extended beyond this $H^1_\stars(\R)$ setting; more about this in the following
subsection.

\subsection{The Hilbert and the modified Hilbert transform}
With respect to the dual action
\[
\langle f,g\rangle_\R=\int_\R f(t)g(t)\diff t,
\]
we may identify the dual space of $H^1_\stars(\R)$ with $\mathrm{BMO}(\R)/\C$.
Here, $\mathrm{BMO}(\R)$ is the space of functions of
\emph{bounded mean oscillation}; this is the celebrated 
\emph{Fefferman duality theorem} \cite{Fef}, \cite{FeSt}. As for notation, 
we write ``$\cdot/\C$'' to express that we mod out with respect to the 
constant functions.
One of the main results in the theory is the theorem of Fefferman and Stein
\cite{FeSt} which tells us that
\begin{equation}
\mathrm{BMO}(\R)=L^\infty(\R)+\tilde\Hop L^\infty(\R).
\label{eq-split1.3}
\end{equation}
or, in words, a function $g$ is in $\mathrm{BMO}(\R)$ if and only if
it may be written in the form $g=g_1+\tilde\Hop g_2$, where
$g_1,g_2\in L^\infty(\R)$. Here, $\tilde\Hop$ denotes the
\emph{modified Hilbert transform}, defined for $f\in L^\infty(\R)$ by the formula
\begin{equation}
\tilde\Hop f(x):=\frac{1}{\pi}\mathrm{pv}\int_\R f(t)\Bigg\{\frac{1}{x-t}+
\frac{t}{1+t^2}\Bigg\}\diff t=
\lim_{\epsilon\to0^+}\int_{\R\setminus[x-\epsilon,x+\epsilon]} f(t)\Bigg\{\frac{1}{x-t}+
\frac{t}{1+t^2}\Bigg\}\diff t.
\label{eq-tildeHilbert01}
\end{equation}
The decomposition \eqref{eq-split1.3} is clearly not unique. The non-uniqueness
of the decomposition is equal to the intersection space
\begin{equation}
H^\infty_\stars(\R):=L^\infty(\R)\cap\tilde\Hop L^\infty(\R),
\label{eq-split1.3.1}
\end{equation}
which we refer to as the \emph{real $H^\infty$-space}.

%\end{document}

We should compare the modified Hilbert transform $\tilde\Hop$ with the
standard \emph{Hilbert transform}
$\Hop$, which acts boundedly on $L^p(\R)$ for $1<p<+\infty$, and maps
$L^1(\R)$ into $L^{1,\infty}(\R)$ for $p=1$. Here, $L^{1,\infty}(\R)$ denotes the
\emph{weak-$L^1$ space}, see e.g. \eqref{eq-quasinorm1.0099}.
The Hilbert transform of a function $f$, assumed integrable on the line $\R$
with respect to the measure $(1+t^2)^{-1/2}\diff t$, is defined as
the principal value integral
\begin{equation}
\Hop f(x):=\frac{1}{\pi}\mathrm{pv}\int_\R f(t)\frac{\diff t}{x-t}=
\lim_{\epsilon\to0^+}\frac{1}{\pi}
\int_{\R\setminus[x-\epsilon,x+\epsilon]} f(t)\frac{\diff t}{x-t}.
\label{eq-Hilbert02}
\end{equation}
If $f\in L^p(\R)$, where $1\le p<+\infty$, then both $\Hop f$ and
$\tilde\Hop f$ are well-defined a.e., and it is easy to see that the difference
$\tilde\Hop f-\Hop f$ \emph{equals to a constant}.
It is often useful to think of the natural harmonic extensions of the Hilbert
transforms $\Hop f$ and $\tilde\Hop f$ to the upper half-plane $\C_+$ given by
\begin{equation}
\Hop f(z):=\frac{1}{\pi}\int_\R \frac{\re z-t}{|z-t|^2}f(t)\diff t,\quad
\tilde\Hop f(z):=\frac{1}{\pi}\int_\R \Bigg\{
\frac{\re z-t}{|z-t|^2}+\frac{t}{t^2+1}\Bigg\}f(t)\diff t.
\label{eq-hilbtrupper1.1}
\end{equation}
So, as a matter of normalization, we have that $\tilde\Hop f(\imag)=0$.
This tells us the value of the constant mentioned above:
$\tilde\Hop f-\Hop f=-\Hop f(\imag)$.

%\end{document}

Returning to the real $H^1$-space, we note the following characterization
of the space in terms of the Hilbert transform:
for $f\in L^1(\R)$,
\begin{equation}
f\in H^1_\stars(\R) \iff f\in L^1_0(\R)\,\,\,\,\text{and}\,\,\,
\Hop f\in L^1_0(\R).
\label{eq-Katzprop1.1}
\end{equation}
%see Proposition \ref{prop-Katz} later on.

The Szeg\H{o} projections $\proj_+$ and $\proj_-$ which were mentioned in 
Subsection \ref{subsec-H1anreal} are more generally defined in terms of 
the Hilbert transform:
\begin{equation}
\proj_+f:=\frac12(f+\imag \Hop f),\quad \proj_-f:=\frac12(f-\imag \Hop f).
\label{eq-projform1}
\end{equation}
In a similar manner, for $f\in L^\infty(\R)$, based on the modified Hilbert
transform $\tilde\Hop$ we may define the corresponding projections (which are
actually projections modulo the constant functions)
\begin{equation}
\tilde\proj_+f:=\frac12(f+\imag\tilde\Hop f),\quad \tilde\proj_-f
:=\frac12(f-\imag \tilde\Hop f),
\label{eq-projform2}
\end{equation}
so that, by definition, $f=\tilde\proj_+f+\tilde\proj_-f$.
%If we are given two functions $f\in  H^1_\stars(\R)$ and $g\in L^\infty(\R)$,
%the dual action $\langle\cdot,\cdot\rangle_\R$ naturally splits into
%holomorphic and conjugate-holomorphic parts:
%\begin{equation}
%\langle f,g\rangle_\R=\langle \proj_+f,\tilde\proj_-g\rangle_\R+
%\langle \proj_-f,\tilde\proj_+g\rangle_\R.
%\label{eq-split1.1}
%\end{equation}

%\section{Operators on a space of distributions on the line}
%\label{sec-zarcalc2}

%\subsection{An involution and the modified Hilbert transform
%on BMO}
%\label{subsec-invol}
%For a positive real parameter $\beta$, let $\Jop_\beta^*$ be the
%involutive operator defined by
%\begin{equation}
%\Jop_\beta^*f(x):=f(-\beta/x),\qquad x\in\R^\times.
%\label{eq-Jop1.1}
%\end{equation}

%We recall the definition \eqref{eq-tildeHilbert01}
%of the modified Hilbert transform $\tilde\Hop$.
%
%\begin{lem}
%For $f\in\mathrm{BMO}(\R)$ and a positive real $\beta$, we have that
%\[
%(\Jop_\beta^*\tilde\Hop f)(x)=(\tilde\Hop\Jop_\beta^* f)(x)+c_\beta(f),
%\]
%where $c_\beta(f)$ is the constant
%\[
%c_\beta(f):=\tilde\Hop f(\imag\beta)
%=(\beta^2-1)\int_\R\frac{tf(t)\,\diff t}{(1+t^2)(\beta^2+t^2)}.
%\]
%\label{lem-Jbetacomm1.1}
%\end{lem}

%\end{document}

\section{Operators on a space of distributions on the line}
\label{sec-HilbL1}

\subsection{The Hilbert transform on $L^1$}
\label{subsec-HilbtransL1}
For background material on the Hilbert transform and related topics, see,
e.g. the monographs \cite{Dur}, \cite{Gar}, \cite{Steinbook1}, 
\cite{Steinbook2}, and \cite{SteinWeissbook}.

%Let $L^{1,\infty}(\R)$ denote the \emph{weak $L^1$-space}, i.e., the space
%of Lebesgue measurable functions $f:\R\to\C$ such that the set
%\[
%E_f(\lambda):=\{x\in\R:\,\,|f(x)|>\lambda\}, \qquad \lambda\in\bar\R_+,
%\]
%enjoys the estimate (the absolute value of a measurable subset of $\R$
%stands for its Lebesgue measure)
%\[
%|E_f(\lambda)|\le \frac{C_f}{\lambda},\qquad \lambda\in\R_+;
%\]
%the optimal constant $C_f$ is written $\|f\|_{L^{1,\infty}(\R)}$; it is the
%\emph{$L^{1,\infty}(\R)$-quasinorm} of $f$.
%
%By identifying functions that coincide
%almost everywhere, the space $L^{1,\infty}(\R)$ is a
%\emph{quasi-Banach space}. 
It is well-known that the Hilbert transform as
given by \eqref{eq-Hilbert02} maps $\Hop: L^1(\R)\to L^{1,\infty}(\R)$. 
Since functions in $L^{1,\infty}(\R)$ have no obvious interpretation as
distributions, it is better to define $\Hop f$ right away as a distribution 
for $f\in L^1(\R)$.
%distribution. However, there is another interpretation of the Hilbert 
%transform as a mapping from $L^1(\R)$ into a space of distributions on $\R$, 
%and it is good to know that these interpretations of $\Hop f$ for a given 
%$f\in L^1(\R)$ are in a one-to-one correspondence. The weak $L^1$-space 
%associated with an interval $I$ (or a set of positive Lebesgue measure), 
%written $L^{1,\infty}(I)$, is defined analogously.
%
%If for the moment we use the symbol $\mathbf{F}$ to denote the Fourier 
%transform, then the Hilbert transform is 
%$\Hop=-\imag\mathbf{F}^{-1}\Mop_{\sign}
%\mathbf{F}$, where $\Mop_{\sign}$ stands for multiplication by the sign 
%function $\sign$. Thus, after taking the Fourier transform, the 
%distributional interpretation of the Hilbert is that of multiplication by the 
%unimodular function which takes the value $-\imag$ on the positive half-line, 
%and the value $\imag$ on the negative half-line.
The distributional interpretation is as follows:
\begin{equation}
\langle\varphi,\Hop f\rangle_\R:=-\langle\Hop\varphi,f\rangle_\R,
\label{eq-Hilbtrans-distr1}
\end{equation}
where $\varphi$ is a test function with compact support, and $f\in L^1(\R)$.
Note that $\Hop\varphi$, the Hilbert transform of the test function,
may be defined without the need of the principal value integral:
\[
\Hop\varphi(x)=\frac{1}{2\pi}\int_\R\frac{\varphi(x-t)-\varphi(x+t)}{t}\diff t;
\]
it is a $C^\infty$ function on $\R$ with decay $\Hop \varphi(x)=\Ordo(|x|^{-1})$
as $|x|\to+\infty$. As a consequence, it is clear from
\eqref{eq-Hilbtrans-distr1} how to extend the notion $\Hop f$ to functions
$f$ with $x\mapsto(1+x^2)^{-1/2}f(x)$ in $L^1(\R)$. Note that as a result of the
work of Kolmogorov, the equivalence \eqref{eq-Katzprop1.1} holds equally well
when $\Hop f$ is interpreted as a distribution and as a weak-$L^1$ function.

%\end{document}
%Our next proposition characterizes the space $H^1_\stars(\R)$. 
%For the proof, we
%need the notation for the \emph{open unit disk}:
%\[
%\D:=\{z\in\C:\,\,|z|<1\}.
%\]

%\begin{prop}
%Suppose $f\in L^1(\R)$. Then the following are equivalent:
%
%\noindent{\rm(i)}
%$f\in H^1_\stars(\R)$.

%\noindent{\rm(ii)} $\Hop f\in L^1(\R)$, where $\Hop f$ is
%understood as a distribution on the line $\R$.

%\noindent{\rm(iii)} $\Hop f\in L^1(\R)$, where $\Hop f$ is
%understood as an almost everywhere defined function in $L^{1,\infty}(\R)$.

%\label{prop-Katz}
%\end{prop}

%A first application of Proposition \ref{prop-Katz} gives us the following
%result.

%\begin{cor}
%Suppose $f\in L^1(\R)$, and that $\Hop f=0$ pointwise almost everywhere on
%$\R$. Then $f=0$ almost everywhere.
%\end{cor}

%\begin{rem}
%We note that there are the closely related theories of reflectionless measures
%(see, e.g., \cite{Pol}) and of real outer functions \cite{GaSa}.
%\end{rem}

\subsection{The real $H^\infty$ space}
The \emph{real $H^\infty$ space} is denoted by $H^\infty_\stars(\R)$, and it
consists of all functions $f\in L^\infty(\R)$ of the form
\begin{equation}
f=f_1+f_2,\qquad f_1\in H^\infty_+(\R),\,\,\,f_2\in H^\infty_-(\R).
\label{eq-Hinftydecomp1.01}
\end{equation}
Here, $H^\infty_+(\R)$ consists of all functions in $L^\infty(\R)$ whose
Poisson extension to the upper half-plane is holomorphic, while
$H^\infty_-(\R)$ consists of all functions in $L^\infty(\R)$
whose Poisson extension to the upper half-plane is conjugate-holomorphic
(alternatively, the Poisson extension to the lower half-plane is holomorphic).
The decomposition \eqref{eq-Hinftydecomp1.01} is unique up to additive
constants. 
%Equipped with the natural norm, $H^\infty_\stars(\R)$ is a
%Banach space.
It is easy to obtain the following equivalence, analogous to 
\eqref{eq-Katzprop1.1}:
\begin{equation}
f\in H^\infty_\stars(\R)\,\,\,\Longleftrightarrow\,\,\,
f,\tilde\Hop f\in L^\infty(\R).
\label{eq-propHinftychar1.1}
\end{equation}

%\end{document}

\subsection{The predual of the real $H^\infty$ space}
We shall be concerned with the following space of distributions on the line
$\R$:
\[
\Lspaceo:=L^1(\R)+\Hop L^1_0(\R),
\]
which we supply with the appropriate norm \eqref{eq-normLspaceo}, that is,
\begin{equation*}
\|u\|_{\Lspaceo}:=\inf\big\{\|f\|_{L^1(\R)}+\|g\|_{L^1(\R)}:\,\,
u=f+\Hop g,\,\, f\in L^1(\R),\,\,g\in L^1_0(\R)\big\},
%\label{eq-normLspaceo:c}
\end{equation*}
which makes $\Lspaceo$ a Banach space.

%\end{document}

We recall that $L^1_0(\R)$ is  the codimension-one subspace of $L^1(\R)$
which consists of the functions whose integral over $\R$ vanishes.
Given $f\in L^1(\R)$ and $g\in L^1_0(\R)$, the action of $u:=f+\Hop g$ on
a test function $\varphi$ is (compare with \eqref{eq-Hilbtrans-distr1})
\begin{equation}
\langle \varphi,f+\Hop g\rangle_\R=\langle \varphi,f\rangle_\R-
\langle \Hop \varphi,g\rangle_\R=\langle \varphi,f\rangle_\R-
\langle \tilde\Hop \varphi,g\rangle_\R;
\label{eq-Hilbtrans-distr2}
\end{equation}
we observe that the last identity uses that $\langle 1,g\rangle_\R=0$ and the
fact that the functions $\tilde\Hop\varphi$ and $\Hop\varphi$ differ
by a constant.
\medskip

%\noindent{\sc Observation.} In view of
%Proposition \ref{prop-Hinftychar1.1}, the right hand side of
%\eqref{eq-Hilbtrans-distr2} makes sense for 
%$\varphi\in H^\infty_\stars(\R)$. To
%be more precise, in accordance with \eqref{eq-Hilbtrans-distr2}, every
%$\varphi\in H^\infty_\stars(\R)$ defines a continuous linear functional on
%$\Lspaceo$.
%\medskip

%\end{document}

It remains to identify the dual space of $\Lspaceo$ with $H^\infty_\stars(\R)$.

\begin{prop}
Each continuous linear functional $\Lspaceo\to\C$ corresponds to a function
$\varphi\in H^\infty_\stars(\R)$ in accordance with \eqref{eq-Hilbtrans-distr2}.
In short, the dual space of $\Lspaceo$ equals $H^\infty_\stars(\R)$.
\label{prop-predualHinfty-1.1}
\end{prop}

This is Proposition 7.3.1 in \cite{HMcpam}.
%The space $\Lspaceo$ is a Banach space, and Proposition
%\ref{prop-predualHinfty-1.1} \emph{asserts that its dual space is
%$H^\infty_\stars(\R)$} (the real $H^\infty$ space).
We will refer to $\Lspaceo$ as the (canonical) \emph{predual of the real
$H^\infty$ space}.

\begin{rem}
Since an $L^1$-function $f$ gives rise to an absolutely continuous measure
$f(t)\diff t$, it is natural to think of $\Lspaceo$ as embedded into the
space $\mathfrak{M}(\R):=M(\R)+\Hop M_0(\R)$, where $M(\R)$ denotes
the space of complex-valued finite Borel measures on $\R$, and $M_0(\R)$ is
the subspace of measures $\mu\in M(\R)$ with $\mu(\R)=0$. The Hilbert
transforms of singular measures noticeably differ from those of
absolutely continuous measures (see \cite{PSZ}).
\end{rem}

\subsection{The ``valeur au point'' function associated with 
an element
of $\Lspaceo$}
We recall that $\Lspaceo$ consists of distributions on the real line.
However, the definition
\[
\Lspaceo=L^1(\R)+\Hop L^1_0(\R)
\]
would allow us to also think of this space as a subspace of $L^{1,\infty}(\R)$,
the weak $L^1$-space. It is a natural question to wonder about the relationship
between the distribution and the $L^{1,\infty}$ function.
We stick with the distribution theory definition of $\Lspaceo$, and
associate with a given $u\in\Lspaceo$ the ``valeur au point'' function
$\pev[u]$ at almost all points of the line. The precise definition of
$\pev[u]$ is as follows.

\begin{defn} For a fixed $x\in\R$, let  $\chi=\chi_x$ is a compactly supported
$C^\infty$-smooth function on $\R$ with $\chi(t)=1$ for all $t$ in an open
neighborhood of the point $x$. Also, let
\[
P_{x+\imag\epsilon}(t):=\pi^{-1}\frac{\epsilon}{\epsilon^2+(x-t)^2}
\]
be the Poisson kernel. The \emph{valeur au point function} associated with
the distribution $u$ on $\R$ is the function $\pev[u]=\pev[u\chi]$ given by
\begin{equation}
\pev[u](x):=
\lim_{\epsilon\to0^+}\langle \chi P_{x+\imag\epsilon}, u\rangle_\R,\qquad
x\in\R,
\label{eq-pv1001}
\end{equation}
wherever the limit exists.
\end{defn}

In principle, $\pev[u](x)$ might depend on the choice of the cut-off function
$\chi$. Lemma 7.4.2 in \cite{HMcpam} guarantees that this is not the case, 
and that almost everywhere, it gives the same result as the weak-$L^{1}$ 
interpretation of the Hilbert transform on $L^1(\R)$. 
%The following lemma guarantees that this is not the case in the
%relevant situation.
%\begin{lem}
%For $u=f+\Hop g\in\Lspaceo$, where $f\in L^1(\R)$ and $g\in L^1_0(\R)$,
%the valeur au point function $\pev[u](x)$ does not depend on the choice of
%the cut-off $\chi$.
%Moreover, we have that
%\[
%\pev[u](x)=f(x)+\Hop g(x),\qquad \text{a.e.}\,\,\,x\in\R,
%\]
%where on the right hand side, the function $\Hop g(x)$ is defined pointwise as
%a principal value.
%\label{lem-indep-of-chi01}
%\end{lem}
%This is Lemma 7.4.2 in \cite{HMcpam}. 
A basic result is the following.

\begin{prop}
{\rm (Kolmogorov)}
The mapping $\pev:\Lspaceo\to L^{1,\infty}(\R)$, $u\mapsto\pev[u]$, is
injective and continuous.
\label{prop-weakL1cont}
\end{prop}

This is a combination of Propositions 7.4.3 and 7.4.4 in \cite{HMcpam}.
%The next result allows us to identify $u$ with $\pev[u]$, even locally.

%\begin{prop}
%{\rm (Kolmogorov)}
%If $u\in\Lspaceo$ and $\pev[u]=0$ almost everywhere on $\R$, then $u=0$ as
%a distribution.
%\label{prop-distrptwise1}
%\end{prop}
%
%The local version of Proposition \ref{prop-distrptwise1} runs as follows.

%\begin{prop}
%If $u\in\Lspaceo$ and $\pev[u]=0$ almost everywhere on an open interval
%$I\subset\R$, then the distribution $u$ is supported on $\R\setminus I$.
%\label{prop-distrptwise1:local1}
%\end{prop}

%\subsection{Dual action on intervals}
%If $I\subset\R$ is an open interval, and $f,g:I\to\C$ are two Borel measurable
%functions with $fg\in L^1(I)$, then we may define \emph{the dual action on}
%$I$:
%\[
%\langle f, g\rangle_I:=\int_{I}f(t)g(t)\diff t;
%\]
%this is a special case of dual action on a more general measurable set (see
%Subsection \ref{subsec-dualaction}).
%For instance, if $f$ is a test function with compact support in $I$, and
%$g$ is locally integrable on $I$, then the dual action is well-defined. More
%generally, we will write $\langle\cdot,\cdot\rangle_I$ to denote the dual
%action of distributions on test functions on the given interval $I$.
%Naturally, this agrees with the notation we have introduced so far for the
%case $I=\R$.

\subsection{The restriction of $\Lspaceo$ to an interval}
\label{subsec-4.6.01}

If $u$ is a distribution on an open interval $J$, then the \emph{restriction
of $u$ to an open subinterval $I$, denoted $u|_I$}, is the distribution defined
by
\[
\langle\varphi,u|_I\rangle_I:=\langle\varphi,u\rangle_J,
\]
where $\varphi$ is a $C^\infty$-smooth test function whose support is compact
and contained in $I$.

\begin{defn}
Let $I$ be an open interval of the real line. Then $u\in\LspaceI$ means by
definition that $u$ is a distribution on $I$ such that there exists a
distribution $v\in\Lspaceo$ such that $u=v|_I$.
\label{defn-4.6.1}
\end{defn}

%\begin{rem}
%The following observation is pretty trivial, but quite useful.
%If we are given two open intervals $I$ and $J$ of the line $\R$,
%with $I\subset J$,  then the restriction operation $v\mapsto v|_I$ makes
%sense $\LspaceJ\to\LspaceI$.
%\end{rem}

Kolmogorov's theorem (Proposition \ref{prop-weakL1cont}) has a local 
version as well.
% on a given interval $J$.

%\begin{cor}
%Suppose $I,J\subset\R$ are open intervals with $I\subset J$. If
%$u\in\LspaceJ$ and $\pev[u]=0$ almost everywhere on $I$, then the support
%of the distribution $u$ has empty intersection with $I$.
%\label{cor-distrptwise1:local2}
%\end{cor}
%
%The following result will prove quite useful.
%
%\begin{prop}
%Let $I$ be a nonempty bounded open interval of the line $\R$.
%Then $L^1(I)$ is a norm dense subspace of $\LspaceI$.
%\end{prop}
%
%We may also translate Proposition \ref{prop-weakL1cont} to this local context.

\begin{prop}
{\rm (Kolmogorov)}
Let $I$ be a nonempty open interval of the line $\R$. Then the ``valeur au
point'' mapping is injective and continuous $\pev:\LspaceI\to L^{1,\infty}(I)$.
\end{prop}

This is a combination of Corollaries 7.6.3 and 7.6.6 in \cite{HMcpam}.

\section{Background material: function spaces on the 
circle}
%\label{sec-background.dynamics}

\subsection{The Hardy space $H^1$ on the circle}
Let $L^1(\R/2\Z)$ denote the space of (equivalence classes of) $2$-periodic  
Borel  measurable functions $f:\R\to\C$ subject to the integrability condition
\[
\|f\|_{L^1(\R/2\Z)}:=\int_{I_1}|f(t)|\diff t<+\infty,
\]
where $I_1=]\!-\!1,1[$ as before. 
%As usual, we identify functions that agree except on a null set. 
Via the exponential mapping $t\mapsto\e^{\imag\pi t}$, which is $2$-periodic and
maps the real line $\R$ onto the unit circle $\Te$, we may identify the
space $L^1(\R/2\Z)$ with the standard Lebesgue space $L^1(\Te)$ of the
unit circle. This will allow us to develop the elements of Hardy space theory
in the setting of $2$-periodic functions. We shall need the subspace
$L^1_0(\R/2\Z)$ consisting of all $f\in L^1(\R/2\Z)$ with
\[
\langle f,1\rangle_{I_1}=\int_{I_1}f(t)\diff t=0;
\]
it has codimension $1$ in $L^1(\R/2\Z)$. The Hardy space $H^1_+(\R/2\Z)$ is
defined as the subspace of $L^1(\R/2\Z)$ consisting of functions
$g\in L^1(\R/2\Z)$ whose Poisson extension to the unit disk $\D$ is holomorphic
and vanishes at the origin, and analogously, $H^1_-(\R/2\Z)$ consists of the 
functions $g$ in $L^1(\R/2Z)$ whose complex conjugate $\bar g$ is in 
$H^1_+(\R/2\Z)$. In terms of the Poisson extensions to the upper half-plane 
instead, $f\in H^1_+(\R/2\Z)$ if the extension is holomorphic and vanishes
at $+\imag\infty$, whereas  $f\in H^1_-(\R/2\Z)$ if the extension is 
conjugate-holomorphic and vanishes at $+\imag\infty$. 
%\begin{equation}
%\int_{-1}^1 \e^{\imag \pi n t}g(t)\diff t=0,\qquad n=0,1,2,\ldots.
%\label{eq-FourChar-H1-circle}
%\end{equation}
%The space $H^1_+(\R/2\Z)$ is the periodic analogue of the Hardy space
%$H^1_+(\R)$, and it can be understood in terms of the Hardy $H^1$-space of
%the disk. If $H^1_+(\Te)$ denotes the standard Hardy space on the unit disk
%(restricted to the boundary unit circle), \emph{then $g\in H^1_+(\R/2\Z)$
%means that $g(x)=f(\e^{\imag\pi x})$ for some $f\in H^1_+(\Te)$ with $f(0)=0$}.
%In particular, the functions in $H^1_+(\R/2\Z)$ have holomorphic extensions
%to the upper half-plane which are $2$-periodic. As a matter of definition,
%\emph{$H^1_-(\R/2\Z)$ consists of the functions $g$ in $L^1(\R/2Z)$ whose
%complex conjugate $\bar g$ is in $H^1_+(\R/2\Z)$}.
We then introduce the \emph{real $H^1$-space}
\[
H^1_{\stars}(\R/2\Z):=H^1_+(\R/2\Z)\oplus H^1_-(\R/2\Z),
\]
where we think of the elements of the sum space as $2$-periodic functions on
$\R$ (as before the symbol $\oplus$ means direct sum, which is possible since 
$H^1_+(\R/2\Z)\cap H^1_-(\R/2\Z)=\{0\}$).
We note that, for instance,  $H^1_\stars(\R/2\Z)\subset L^1_0(\R/2\Z)$. 
%We will think of $H^1_{\stars}(\R/2\Z)$ as the \emph{real $H^1$ space of
%$2$-periodic functions}.

\subsection{The Hilbert transform on $2$-periodic functions and
distributions}
For $f\in L^1(\R/2\Z)$, we let $\Hop_2$ be the convolution operator
\begin{equation}
\Hop_2 f(x):=
\frac{1}{2}\pv
\int_{I_1} f(t)\cot\frac{\pi(x-t)}{2}\diff t,
\label{eq-Hilbert04}
\end{equation}
where again $\pv$  stands for principal value, which means we take the 
limit as $\epsilon\to0^+$ of the integral where the set
\[
\{x\}+2\Z+[-\epsilon,\epsilon]
\]
is removed from the interval $I_1=]\!-\!1,1[$. It is obvious from the
periodicity of the cotangent function that $\Hop_2 f$,   if it exists
as a limit, is $2$-periodic. 
%By a standard trigonometric identity,
%\[
%\frac{1}{2}\cot\frac{\pi y}{2}=\lim_{N\to+\infty}\frac{1}{\pi}\sum_{n=-N}^{N}
%\frac{1}{y+2n},
%\]
%where the convergence is uniform on compact subsets of the line. 
Alternatively, by a change of variables, we have that 
\begin{equation}
\Hop_2 f(x)=
\frac{1}{2}\lim_{\epsilon\to0^+}
\int_{I_1\setminus I_\epsilon} f(x-t)\cot\frac{\pi t}{2}\diff t,
\label{eq-Hilbert04.1}
\end{equation}
(here, as usual, $I_\epsilon=]\!-\!\epsilon,\epsilon[$).
% from which we
%conclude, by uniform convergence and periodicity, that
%\begin{multline}
%\Hop_2 f(x)=
%\frac{1}{\pi}\lim_{N\to+\infty}\lim_{\epsilon\to0^+}
%\sum_{n=-N}^{N}\int_{I_1\setminus I_\epsilon} f(x-t)\frac{\diff t}{t+2n}
%\\
%=\frac{1}{\pi}\lim_{\epsilon\to0^+}
%\int_{I_1\setminus I_\epsilon} f(x-t)\frac{\diff t}{t}+
%\frac{1}{\pi}\lim_{N\to+\infty}\sum_{n:|n|\le N, n\ne0}
%\int_{I_1} f(x-t)\frac{\diff t}{t+2n}
%\\
%=\frac{1}{\pi}\lim_{\epsilon\to0^+}
%\int_{I_1\setminus I_\epsilon} f(x-t)\frac{\diff t}{t}+
%\frac{1}{\pi}\lim_{N\to+\infty}\sum_{n:|n|\le N, n\ne0}
%\int_{[2n-1,2n+1]} f(x-t)\frac{\diff t}{t}
%\\
%=\lim_{N\to+\infty}\lim_{\epsilon\to0^+}\frac{1}{\pi}
%\int_{I_{2N+1}\setminus I_\epsilon} f(x-t)\frac{\diff t}{t}.
%\label{eq-Hilbert04.2}
%\end{multline}
It is well-known that the operator $\Hop_2$ is just the natural extension 
of the Hilbert transform $\Hop$ to the $2$-periodic functions. 
We observe the peculiarity that $\Hop_21=0$, which follows from the fact that
the cotangent function is odd. 
%contrasts with the non-periodic case (where no nontrivial function is
%mapped to the zero function).
Like the situation for the real line $\R$, the periodic Hilbert transform 
$\Hop_2$ maps $L^1(\R/2\Z)$ into the weak $L^1$-space $L^{1,\infty}(\R/2\Z)$. 
However, as we prefer to work within the framework of distribution theory, 
we proceed as follows.

Let $C^\infty(\R/2\Z)$ denote the space of $C^\infty$-smooth $2$-periodic
functions on $\R$. It is easy to see that
\[
\varphi\in C^\infty(\R/2\Z) \implies \Hop_2\varphi\in C^\infty(\R/2\Z).
\]
To emphasize the importance of the circle $\Te\cong\R/2\Z$, we write
\begin{equation}
\langle f,g\rangle_{\R/2\Z}:=\int_{-1}^1 f(t)g(t)\diff t,
\label{eq-dualaction-per1}
\end{equation}
for the dual action when $f$ and $g$ are $2$-periodic.

\begin{defn}
For a test function $\varphi\in C^\infty(\R/2\Z)$ and a distribution $u$ on
the circle $\R/2\Z$, we put
\[
\langle \varphi,\Hop_2 u\rangle_{\R/2\Z}:=-\langle\Hop_2\varphi,u\rangle_{\R/2\Z}.
\]
\end{defn}

This defines the Hilbert transform $\Hop_2 u$ for any distribution $u$ on the
circle $\R/2\Z$.

%The analogue of Proposition \ref{prop-Katz} for the circle reads as follows.
%Note that the formula definition of the ``valeur au point'' function makes
%sense also for $u$ in the space of distributions
%$L^1(\R/2\Z)+\Hop_2 L^1(\R/2\Z)$. Moreover, the independence of the cut-off
%function is quite analogous to the real line case (Lemma
%\ref{lem-indep-of-chi01}) and left to the interested reader.
%
%\begin{prop}
%Suppose $f\in L^1_0(\R/2\Z)$. Then the following are equivalent:
%
%\noindent{\rm(i)}
%$f\in H^1_\stars(\R/2\Z)$.
%
%\noindent{\rm(ii)} $\Hop_2 f\in L^1(\R/2\Z)$, where $\Hop_2 f$ is
%understood as a distribution on the line $\R$.
%
%\noindent{\rm(iii)} $\pev[\Hop_2 f]\in L^1(\R/2\Z)$.
%\label{prop-Katz2}
%\end{prop}

\subsection{The real $H^\infty$-space of the circle}
The \emph{real $H^\infty$-space on the circle $\R/2\Z$} is denoted by
$H^\infty_\stars(\R/2\Z)$, and consists of all the functions in $H^\infty_\stars
(\R)$ that are $2$-periodic. It follows from \eqref{eq-propHinftychar1.1}
that 
\begin{equation}
f\in H^\infty_\stars(\R/2\Z)\,\,\,\Longleftrightarrow\,\,\,
f,\Hop_2 f\in L^\infty(\R/2\Z).
\label{eq-prop-Hinftychar1.2}
\end{equation}

\subsection{A predual of $2$-periodic real $H^\infty$}
We put
\[
\Lspaceoper:=L^1(\R/2\Z)+\Hop_2 L^1_0(\R/2\Z),
\]
understood as a space of $2$-periodic distributions on the line $\R$.
More precisely, if $u=f+\Hop_2 g$, where $f\in L^1(\R/2\Z)$ and
$g\in L^1_0(\R/2\Z)$, then the action on a test function
$\varphi\in C^\infty(\R/2\Z)$ is given by
\begin{equation}
\langle\varphi,u\rangle_{\R/2\Z}:=\langle\varphi,f\rangle_{\R/2\Z}
-\langle \Hop_2\varphi, g\rangle_{\R/2\Z}.
\label{eq-Hilbtransf-distrib-per1.1}
\end{equation}
But a $2$-periodic distribution should be possible to think of as a
distribution on the line, which means that need to understand the action
on standard test functions in $C^\infty_c(\R)$. If $\psi\in C^\infty_c(\R)$,
we simply put
\begin{equation}
\langle\psi,u\rangle_{\R/2\Z}:=\langle\Perop_2\psi,u\rangle_{\R/2\Z},
\label{eq-extper1.01}
\end{equation}
where $\Perop_2\psi\in C^\infty(\R/2\Z)$ is given by
\begin{equation}
\Perop_2\psi(x):=\sum_{j\in\Z}\psi(x+2j).
\label{eq-Peropdef1.1}
\end{equation}
We will refer to $\Perop_2$ as the \emph{periodization operator}.

As in the case of the line $\R$, we may identify  $\Lspaceoper$ with the 
predual of the real $H^\infty$-space  $H^\infty_\stars(\R/2\Z)$:
\[
\Lspaceoper^*=H^\infty_\stars(\R/2\Z)
\]
with respect to the standard dual action $\langle\cdot,\cdot\rangle_{\R/2\Z}$.
%
%\begin{prop}
%Each continuous linear functional $\Lspaceoper\to\C$ corresponds to a function
%$\varphi\in H^\infty_\stars(\R/2\Z)$ in accordance with
%\eqref{eq-Hilbtransf-distrib-per1.1}.
%In short, the dual space of $\Lspaceoper$ is isomorphic to
%$H^\infty_\stars(\R/2\Z)$.
%\label{prop-predualHinfty-1.2}
%\end{prop}
%
%We suppress the proof, which is analogous to that of Proposition
%\ref{prop-predualHinfty-1.1}.

The definition of the ``valeur au point function'' $\pev[u]$ makes sense
for $u\in\Lspaceoper$ and as in the case of the line, it does not depend on
the choice of the particular cut-off function. The following assertion is the
analogue of Proposition \ref{prop-weakL1cont}; the proof is suppressed.

\begin{prop}
{\rm (Kolmogorov)}
The ``valeur au point'' mapping $\pev:\Lspaceoper\to L^{1,\infty}(\R/2\Z)$,
$u\mapsto\pev[u]$, is injective and continuous.
\label{prop-weakL1cont.per}
\end{prop}

\section{A sum of two preduals and its localization to intervals}

\subsection{The sum space $\Lspaceo\oplus\Lspaceoper$}
Suppose $u$ is distribution on the line $\R$ of the form
\begin{equation}
u=v+w,\quad\text{where}\quad v\in\Lspaceo,\,\,\, w\in\Lspaceoper.
\label{eq-sumspace001}
\end{equation}
The natural question appears as to whether the distributions $v,w$ on the
right-hand side are unique for a given $u$. This is indeed so 
(Proposition 9.1.1 in \cite{HMcpam}):
\begin{equation}
\Lspaceo\cap\Lspaceoper=\{0\}.
\label{eq-prop-intersec1.001}
\end{equation}
In view of \eqref{eq-prop-intersec1.001}, it makes sense to write
$\Lspaceo\oplus\Lspaceoper$ for the space of tempered distributions $u$
of the form \eqref{eq-sumspace001}. We endow $\Lspaceo\oplus\Lspaceoper$ with
the induced Banach space norm
\[
\|u\|_{\Lspaceo\oplus\Lspaceoper}:=\|v\|_{\Lspaceo}+\|w\|_{\Lspaceoper},
\]
provided $u,v,w$ are related via \eqref{eq-sumspace001}.

\subsection{The localization of $\Lspaceo\oplus\Lspaceoper$ to a
bounded open interval}
\label{subsec-6.2}

In the sense of Subsection \ref{subsec-4.6.01}, we may restrict a given
distribution $u\in \Lspaceo\oplus\Lspaceoper$ to a given open interval $I$.
It is natural to wonder what the space of such restrictions looks like.

\begin{prop}
% $(0<\gamma<+\infty)$
The restriction of the space $\Lspaceo\oplus\Lspaceoper$ to a bounded open
interval $I$ equals the space $\LspaceI$.
\label{prop-rest1.1}
\end{prop}

This is Proposition 9.2.1 in \cite{HMcpam}.
%\end{document}

\section{An involution, its adjoint, and the
%weighted composition and
periodization operator}

\subsection{An involutive operator}
\label{subsec-7.1}

For each positive real number $\beta$, let $\Jop_\beta$ denote  the
involution given by
\begin{equation*}
\Jop_\beta f(x):=\frac{\beta}{x^2}\,f(-\beta/x),\qquad x\in\R^\times.
%\label{eq-Jopstar}
\end{equation*}
We use the standard notation $\R^\times:=\R\setminus\{0\}$.
If $f\in L^1(\R)$ and $\varphi\in L^\infty(\R)$, the change-of-variables formula
yields that
\begin{equation}
\langle \varphi,\Jop_\beta f\rangle_\R=\int_\R \varphi(t)\,f(-\beta/t)\,
\frac{\beta\diff t}{t^2}=\int_\R \varphi(-\beta/t)\,f(t)\,
\diff t=\langle \Jop_\beta^* \varphi,f\rangle_\R,
\label{eq-Jopdef1}
\end{equation}
where $\Jop_\beta^*$ is the involution
\[
\Jop_\beta^* \varphi(t):=\varphi(-\beta/t),\qquad t\in\R^\times.
\]
%
% With respect to the dual action $\langle\cdot,\cdot\rangle_\R$,
% the adjoint of the involution $\Jop_\beta^*$ is the operator
%\begin{equation*}
%\Jop_\beta^* g(x):=g(-\beta/x),\qquad x\in\R^\times.
%\label{eq-Jopstar}
%\end{equation*}
%We use the standard notation $\R^\times:=\R\setminus\{0\}$.
%We record some basic properties of the involution $\Jop_\beta$. 
%For instance, by the
%change-of-variables formula, $\Jop_\beta:L^1(\R)\to L^1(\R)$ is an
%\emph{isometry}.
It is a consequence of the change-of-variables formula that $\Jop_\beta$ is an 
isometric isomorphism $L^1(\R)\to L^1(\R)$. 
%
%\begin{prop} Fix $0<\beta<+\infty$. The operator  $\Jop_\beta$ is an 
%isometric isomorphism $L^1(\R)\to L^1(\R)$. In addition,  $\Jop_\beta$ maps
%$H^1_+(\R)\to H^1_+(\R)$ and $H^1_-(\R)\to H^1_-(\R)$ and consequently
%$\Jop_\beta:H^1_\stars(\R)\to H^1_\stars(\R)$ as well.
%\label{prop-1.001}
%\end{prop}
%
%This is Proposition 10.1.1 in \cite{HMcpam}.

Next, we \emph{extend $\Jop_\beta$ to a bounded operator 
$\Lspaceo\to\Lspaceo$}. 
The arguments in Subsection 10.1 of \cite{HMcpam} show that the correct
extension of $\Jop_\beta$ to an operator $\Lspaceo\to\Lspaceo$ reads as 
follows.  

\begin{defn}
For $u\in\Lspaceo$ of the form $u=f+\Hop g\in\Lspaceo$, where
$f\in L^1(\R)$ and $g\in L^1_0(\R)$, we define the $\Jop_\beta u$ to be
the distribution on $\R$ given by the formula
\begin{equation*}
\langle \varphi,\Jop_\beta u\rangle_\R=\langle \varphi,\Jop_\beta
(f+\Hop g)\rangle_\R:=
\langle\varphi,\Jop_\beta f\rangle_\R+\langle\varphi,\Hop\Jop_\beta g\rangle_\R
=\langle\varphi,\Jop_\beta f\rangle_\R-\langle\tilde\Hop\varphi,
\Jop_\beta g\rangle_\R,
\end{equation*}
for test functions $\varphi\in H^\infty_\stars(\R)$.
\label{defn-7.1.3}
\end{defn}

%As already noted, this is in complete agreement with the way we would
%previously understand $\Jop_\beta u$ as a distribution on $\R^\times$, using
%smooth test functions having compact support on the punctured line
%$\R^\times$; see \eqref{eq-Jopdef1} and \eqref{eq-Jopdef2}.

The involutive properties of $\Jop_\beta$ and its adjoint are then naturally 
preserved (Proposition 10.1.4 in \cite{HMcpam}).

%\begin{prop}
%Fix $0<\beta<+\infty$.
%The involution $\Jop_\beta$ acts continuously
%$\Lspaceo\to\Lspaceo$, and the involution $\Jop_\beta^*$ acts
%continuously $H^\infty_\stars(\R)\to H^\infty_\stars(\R)$. Moreover, on their
%respective spaces, $\Jop_\beta^2$ and $ (\Jop_\beta^*)^2$ both equal 
%the identity operator.
%\label{prop-7.1.3}
%\end{prop}
%
%This is Proposition 10.1.4 in \cite{HMcpam}.

\subsection{The periodization operator}
\label{subsec-7.2}

We recall the definition of the \emph{periodization operator} $\Perop_2$:
\[
\Perop_2f(x):=\sum_{j\in\Z}f(x+2j).
\]
In \eqref{eq-Peropdef1.1}, we defined the $\Perop_2$ on test functions.
It is however clear that it remains well-defined with much less smoothness
required of $f$. The terminology comes from the property that whenever it is
well-defined, the function $\Perop_2 f$ is $2$-periodic automatically.
%A first result is the following.
It is obvious from the definition that $\Perop_2$ acts contractively 
$L^1(\R)\to L^1(\R/2\Z)$.
%
%\begin{prop}
%The operator $\Perop_2$ acts contractively $L^1(\R)\to L^1(\R/2\Z)$.
%Moreover, $\Perop_2$ maps $H^1_+(\R)$ onto $H^1_+(\R/2\Z)$ and
%$H^1_-(\R)$ onto $H^1_-(\R/2\Z)$.
%\label{prop-1.002}
%\end{prop}

The 
%identity \eqref{eq-Pi2id1.1} is a special case of a more general identity,
basic property of the periodization operator is the following, for 
$f\in L^1(\R)$ and $F\in L^\infty(\R/2\Z)$ (see, e.g., (10.2.2) 
in \cite{HMcpam}):
%(compare with \eqref{eq-extper1.01}):
\begin{equation}
\langle F,\Perop_2 f\rangle_{\R/2\Z}=\langle F,f\rangle_\R,
\qquad n\in\Z.
\label{eq-Pi2id1.1'}
\end{equation}
We need to extend $\Perop_2$ in a natural fashion to the space $\Lspaceo$.
If $\varphi\in C^\infty(\R/2\Z)$ is a test function on the circle, we glance
at \eqref{eq-Pi2id1.1'}, and
for $u\in\Lspaceo$ with $u=f+\Hop g$, where $f\in L^1(\R)$ and
$g\in L^1_0(\R)$, we set
\begin{equation}
\langle \varphi, \Perop_2 u\rangle_{\R/2\Z}:=\langle \varphi,u\rangle_\R
=\langle\varphi,f\rangle_\R-\langle\tilde\Hop\varphi,g\rangle_\R.
\label{eq-Pi2id1.2}
\end{equation}
This defines $\Perop_2 u$ as a distribution on the circle (compare with
\eqref{eq-Hilbtrans-distr2}).

\begin{prop}
For $u\in\Lspaceo$ of the form $u=f+\Hop g$, where $f\in L^1(\R)$ and
$g\in L^1_0(\R)$, we have that  $\Perop_2 u=\Perop_2 f+\Hop_2\Perop_2 g$.
In particular, $\Perop_2$ maps $\Lspaceo\to\Lspaceoper$ continuously.
\label{prop-7.2.2}
\end{prop}

This is Proposition 10.2.2 in \cite{HMcpam}.

\section{Reformulation of the spanning problem of 
Theorem \ref{thm-2.0}}

\subsection{An equivalence}
\label{subsec-dualform}

%By a scaling argument, which reduces to $\alpha:=1$, Theorem \ref{thm-2.0}
%is \emph{equivalent to the following problem having an affirmative 
%solution if and only if $0<\beta\le1$}.
Let us write $\Z_{+,0}:=\{0,1,2,\ldots\}$. 

\begin{lem}
Fix $0<\beta<+\infty$. Then the following conditions {\rm(a)} and {\rm(b)}
are equivalent:

\noindent{\rm{(a)}}
The linear span of the functions
\[
e_n(t):=\e^{\imag\pi nt},\,\,e_m^{\langle\beta\rangle}(t):=
\e^{-\imag\pi\beta m/t},\qquad m,n\in\Z_{+,0},
\]
is weak-star dense in $H^\infty_+(\R)$.

\noindent{\rm{(b)}}
For $f\in L^1_0(\R)$, the following implication holds:
\begin{equation*}
\Perop_2 f,\,\Perop_2\Jop_\beta f\in H^1_+(\R/2\Z)
\quad\Longrightarrow\quad f\in H^1_+(\R).
\end{equation*}
\label{lem-equiv8.1.1.001}
\end{lem}

We remark that the functions $\e^{\imag\pi nt}$ and $\e^{-\imag\pi\beta m/t}$ 
for $m,n\in\Z_{+,0}$ belong to $H^\infty_+(\R)$ (after all, 
they have bounded holomorphic extensions to $\C_+$), so that part (a) 
makes sense. 

\begin{proof}[Proof of Lemma \ref{lem-equiv8.1.1.001}]
With respect to the dual action $\langle\cdot,\cdot\rangle_\R$ on the line,
the predual of $H^\infty_+(\R)$ is the quotient space $L^1(\R)/H^1_+(\R)$.
With this in mind, the assertion of part (a) is seen to be equivalent 
to the following:
For any $f\in L^1(\R)$, the implication
\begin{equation}
\Big\{\forall m,n\in\Z_{+,0}:\,\langle e_n,f\rangle_\R=
\langle e_m^{\langle\beta\rangle},f\rangle_\R=0\Big\} 
\quad\Longrightarrow\quad f\in H^1_+(\R)
\label{eq-dual1.1}
\end{equation}
holds. By testing with e.g. $n=0$, we note that we might as well assume that
$f\in L^1_0(\R)$ in \eqref{eq-dual1.1}. By the basic property 
\eqref{eq-Pi2id1.1'} of the periodization operator $\Perop_2$, we have that 
\begin{equation}
\langle e_n,f\rangle_\R=\langle e_n,\Perop_2f\rangle_{\R/2\Z},
\label{eq-duality1.001}
\end{equation}
from which we conclude that
\[
\Big\{\forall n\in\Z_{+,0}:\,\langle e_n,f\rangle_\R=0\Big\}
\quad\Longleftrightarrow\quad\Perop_2 f\in H^1_+(\R/2\Z).
\]
Since $\Jop_\beta^* e_m=e_m^{\langle\beta\rangle}$, where $\Jop_\beta^*$ is the
involution studied in Subsection 
\ref{subsec-7.1}, a repetition of the above gives that for $f\in L^1_0(\R)$, 
we have the equivalence
\[
\Big\{\forall m\in\Z_{+,0}:\,\langle e_m^{\langle\beta\rangle},f\rangle_\R=0\Big\}
\quad\Longleftrightarrow\quad
\Perop_2\Jop_\beta f\in H^1_+(\R/2\Z).
\]
By splitting the annihilation conditions in \eqref{eq-dual1.1}, we see that
they are equivalent to having both $\Perop_2 f$ and $\Perop_2\Jop_\beta f$ in
$H^1_+(\R/2\Z)$. In other words, conditions (a) and (b) are equivalent. 
\end{proof}

%\begin{prob}
%For which values of the positive real parameter $\beta$ is the linear span
%of the functions
%\[
%e_n(t):=\e^{\imag\pi nt},\,\,e_m^{\langle\beta\rangle}(t):=
%\e^{-\imag\pi\beta m/t},\qquad m,n\in\Z_{+,0},
%\]
%weak-star dense in $H^\infty_+(\R)$?
%\label{prob-fund1}
%\end{prob}

%Here, we use the notation $\Z_{+,0}:=\{0,1,2,\ldots\}$. 

\begin{rem}
By the argument involving point separation in $\C_+$ from \cite{HM}, 
the condition $\beta\le1$ is necessary for part (a) of Lemma 
\ref{lem-equiv8.1.1.001} to hold. Actually, as mentioned in the
introduction, the methods of \cite{CHM} supply infinitely many linearly
independent counterexamples for $\beta>1$.
\end{rem}

%Also, by testing with $n=0$, we note that we might as well assume that
%$f\in L^1_0(\R)$ in \eqref{eq-dual1.1}.
%In view of \eqref{eq-Pi2id1.1'},
%\begin{equation}
%\langle e_n,f\rangle_\R=\langle e_n,\Perop_2f\rangle_{\R/2\Z},
%\label{eq-duality1.001}
%\end{equation}
%so that for $f\in L^1(\R)$ we have the equivalence
%\[
%\Big\{\forall n\in\Z_{+,0}:\,\langle e_n,f\rangle_\R=0\Big\}
%\quad\Longleftrightarrow\quad\Perop_2 f\in H^1_+(\R/2\Z).
%\]
%Since $\Jop_\beta^* e_m=e_m^{\langle\beta\rangle}$, where $\Jop_\beta^*$ 
%is the involution studied in Subsection 
%\ref{subsec-invol} and 
%\ref{subsec-7.1}, we see that
%\[
%\langle f, e_m^{\langle\beta\rangle}\rangle_\R=
%\langle f,\Jop_\beta^* e_m\rangle_\R
%=\langle \Jop_\beta f,e_m\rangle_\R,
%\]
%which leads 
%for $f\in L^1(\R)$, we have the equivalence
%\[
%\Big\{\forall m\in\Z_{+,0}:\,\langle e_m^{\langle\beta\rangle},
%f\rangle_\R=0\Big\}
%\quad\Longleftrightarrow\quad
%\Perop_2\Jop_\beta f\in H^1_+(\R/2\Z).
%\]
%This analysis allows us to rephrase the question \eqref{eq-dual1.1} and 
%hence Problem \ref{prob-fund1}.
%
%\begin{prob} Fix $0<\beta\le1$.
%Is it true that for $f\in L^1_0(\R)$,
%\begin{equation*}
%\Perop_2 f,\,\Perop_2\Jop_\beta f\in H^1_+(\R/2\Z)
%\quad\Longrightarrow\quad f\in H^1_+(\R)?
%\label{eq-dual1.2}
%\end{equation*}
% \label{prob-fund2}
%\end{prob}
%It is rather obvious that the reverse implication holds.
\begin{rem}
If we think of $\Perop_2 f$ and $\Perop_2\Jop_\beta f$ as $2$-periodic 
``shadows'' of $f$ and $\Jop_\beta f$,
the issue at hand in part (b) of Lemma \ref{lem-equiv8.1.1.001}
is whether knowing that the two shadows are in the
right space we may conclude the function comes from the space $H^1_+(\R)$.
We note here that the main result of \cite{HM} may be understood as the
assertion that \emph{$f$ is determined uniquely by the two ``shadows''
$\Perop_2 f$ and $\Perop_2\Jop_\beta f$ if and only if $\beta\le1$}.
%This offers some rather weak support for the plausibility of the implication
%of Problem \ref{prob-fund2}, specially taking into account that if we 
%replace $\Lspaceo$ by $L^{1,\infty}(\R)$, Theorem 1.8.2 fails to be true 
%in a drastic manner, cf. Remark \ref{example}.
\end{rem}

\subsection{An alternative statement in terms of the space
$\Lspaceo$}
\label{subsec-reformulation2}

Let $\Lspaces$ denote the space
\[
\Lspaces:=L^1_0(\R)+\Hop L^1_0(\R)\subset\Lspaceo,
\]
which has codimension $1$ in $\Lspaceo$.

\begin{lem}
Fix $0<\beta\le1$. Then {\rm(a)}$\implies${\rm(b)}, where {\rm(a)} and 
{\rm(b)} are the following assertions:

\noindent{\rm{(a)}}
For $u\in\Lspaces$, the following implication holds:
\begin{equation*}
\Perop_2 u=\Perop_2\Jop_\beta u=0\quad\Longrightarrow\quad u=0.
\end{equation*}

\noindent{\rm{(b)}}
For $f\in L^1_0(\R)$, the following implication holds:
\begin{equation*}
\Perop_2 f,\,\Perop_2\Jop_\beta f\in H^1_+(\R/2\Z)
\quad\Longrightarrow\quad f\in H^1_+(\R).
\end{equation*}
\label{lem-equiv8.1.1.002}
\end{lem}

\begin{proof}
We connect $u\in\Lspaces$ with $f\in L^1_0(\R)$ via the conjugate-holomorphic
Szeg\H{o} projection
$u:=\proj_-f=\frac{1}{2}(f-\imag \Hop f)$. If $\Perop_2 f\in H^1_+(\R/2\Z)$,
then by a Liouville-type argument, $\Perop_2 u=0$ holds. Analogously, if 
$\Perop_2\Jop_\beta f\in H^1_+(\R/2\Z)$, then we obtain that
$\Perop_2\Jop_\beta u=0$. So, from the implication of part (a), we obtain
from the assumptions in (b) that $u=0$, that is, that $f\in H_+(\R)$. 
This means that the implication of (a) implies that of (b), as claimed.
\end{proof}

\begin{rem}
Condition (b) of Lemma \ref{lem-equiv8.1.1.002} has acquired the same 
general appearance as in the analysis of the $L^\infty(\R)$ problem, but 
at the cost of considering the larger space $\Lspaces$ in place of 
$L^1_0(\R)$. This is unavoidable, as the weak-star topology of the real 
Hardy space $H^\infty_\stars(\R)$ is finer than that of $L^\infty(\R)$.
Our proof of Theorem \ref{thm-2.0} passes through Lemmas 
\ref{lem-equiv8.1.1.001} and \ref{lem-equiv8.1.1.002}, and we ultimately
show that the implication (a) of Lemma \ref{lem-equiv8.1.1.002} is valid for
$0<\beta\le1$. It then follows from Lemmas \ref{lem-equiv8.1.1.001} and 
\ref{lem-equiv8.1.1.002} that assertion (a) of Lemma \ref{lem-equiv8.1.1.001}
is valid in the range $0<\beta\le1$. In its turn, the proof that the 
implication (a) of Lemma \ref{lem-equiv8.1.1.002} holds for $0<\beta\le1$ 
is based on an extension of ergodic theory for Gauss-type maps, developed in
Sections 9-14. 
\end{rem}
%then also proves Theorem \ref{thm-2.0}
%(see Proposition 11.2.2 of \cite{HMcpam}).

%\begin{prop}
%If the answer to Problem \ref{prob-fund3} is affirmative, then the answers
%to Problems \ref{prob-fund1} and \ref{prob-fund2} are affirmative as well,
%and the assertion of Theorem \ref{thm-2.0} is valid.
%\label{prop-fund3.01}
%\end{prop}

%\begin{proof}
%We already know that Problems \ref{prob-fund1} and \ref{prob-fund2} are
%equivalent. Let $f\in L^1(\R)$ be such that $\Perop_2 f\in H^1_+(\R/2\Z)$ and
%$\Perop_2\Jop_\beta f\in H^1_+(\R/2\Z)$. Then, as a first step,
%$f\in L^1_0(\R)$ by the identity \eqref{eq-Pi2id1.1} with $n=0$.
%We recall the notation $\proj_-:=\frac12(\id-\imag\Hop)$ for the projection
%to the conjugate-holomorphic functions in $\C_+$.
%Next, we consider the distribution $u:=\proj_-f=\frac{1}{2}(f-\imag\Hop f)
%\in\Lspaces$, and use the identities \eqref{eq-8.1.3} and \eqref{eq-8.1.4}
%together with \eqref{eq-proj=0} to see that
%$\Perop_2 u=\Perop_2\Jop_\beta u=0$.
%Now, given that Problem \ref{prob-fund3} has an affirmative answer, we
%have that $\proj_-f=u=0$, which is only possible for $f\in L^1(\R)$
%if $f\in H^1_+(\R)$. We conclude that  Problems \ref{prob-fund1} and
%\ref{prob-fund2} have affirmative answers as well.
%Finally, given the discussion in Subsection \ref{subsec-quad1.01},
%the correctness of the assertion of Theorem \ref{thm-2.0} follows as well.
%\end{proof}
%

\section{A subtransfer operator on a space of 
distributions}
\label{sec-subtransferop}

\subsection{Restrictions of $\Lspaceo$ to a symmetric interval 
and to its complement}

In Subsection \ref{subsec-4.6.01}, we defined the restriction of $\Lspaceo$ to
an open interval. Here we the restriction to the complement of a closed
interval as well. For a positive real parameter $\gamma$, we consider the
symmetric interval $I_\gamma$ and its closure $\bar I_\gamma$ as in Subsection
\ref{subsec-notation-intervals},
\begin{equation*}
I_\gamma=]\!-\!\gamma,\gamma[,\quad \bar I_\gamma=[-\gamma,\gamma].
%\label{eq-Igamma1.1}
\end{equation*}
We recall that by Definition \ref{defn-4.6.1}, the space $\LspaceIg$ is
defined as
\[
\LspaceIg:=\big\{u\in {\mathcal D}'(I_\gamma):\,
\exists U\in\Lspaceo\,\,
\text{with}\,\, U|_{I_\gamma}=u\big\}
\]
and analogously we may define $\LspaceIgcompl$ for the complementary interval
$\R\setminus\bar I_\gamma$:
\[
\LspaceIgcompl:=\big\{u\in {\mathcal D}'(\R\setminus \bar I_\gamma):
\,\exists U\in\Lspaceo\,\,\text{with}\,\, U|_{\R\setminus \bar I_\gamma}=u\big\}.
\]
Here, ${\mathcal D}'$ has the standard interpretation of the space of
Schwartzian distributions on the given interval.
Of course, in the sense of distribution theory, taking the restriction to an
open subset has the interpretation of considering the linear functional
restricted to test functions supported on that given open subset.
The norm on each of the spaces $\LspaceIg$ and $\LspaceIgcompl$ is the
associated quotient norm, where we mod out with respect to the distributions
in $\Lspaceo$ whose support is contained in the complementary  closed   set
(cf. Subsection \ref{subsec-4.6.01}).

We will need to work with restrictions to $I_\gamma$ and
$\R\setminus\bar I_\gamma$ repeatedly, so it is good idea to introduce
appropriate notation.

\begin{defn}
We let $\Rop_\gamma$ denote the operation of restricting a distribution to
the interval $I_\gamma$.
Analogously, we let $\Rop_{\gamma}^\dagger$ denote the operation of
restricting a distribution to the open set $\R\setminus\bar I_\gamma$.
\end{defn}

\subsection{The involution on the local spaces}

We need to understand the action of the involution
$\Jop_\beta$ defined in Subsection \ref{subsec-7.1} on the local spaces
$\LspaceIg$ and $\LspaceIgcompl$.

\begin{prop} Fix $0<\beta,\gamma<+\infty$.
The involution $\Jop_\beta$ defines continuous maps
\[
\Jop_\beta:\LspaceIg\to\LspaceIbgcompl\quad\text{and}\quad
\Jop_\beta:\LspaceIgcompl\to\LspaceIbg.
\]
%Moreover, in both instances, $\Jop_\beta$ is an involution, 
%since $\Jop_\beta^2$
%equals the identity.
\label{prop-mappinJop1.1}
\end{prop}

%This is Proposition in \cite{HMcpam}.

\begin{proof}
The assertion is rather immediate from the mapping properties of $\Jop_\beta$
(see Subsection \ref{subsec-7.1}) and the localization procedure.
\end{proof}

\subsection{Splitting of the periodization operator}
We split the periodization operator $\Perop_2$ in two parts:
$\Perop_2=\id+\Superop_2$, where $\id$ is the identity and $\Superop_2$ is
the operator defined by
\[
\Superop_2 u(x):=\sum_{j\in\Z^\times}u(x+2j),
\]
whenever the right-hand side is meaningful in the sense of distributions.
Here, we use the notation $\Z^\times:=\Z\setminus\{0\}$.
In view of Proposition \ref{prop-7.2.2},  the proof of the following 
proposition is immediate.

\begin{prop}
The operator $\Superop_2$ maps
$\Lspaceo\to\Lspaceo\oplus\Lspaceoper$ continuously.
\label{prop-Superop2}
\end{prop}

\begin{defn}
Let
$\Superoprem:\LspaceIonecompl\to\LspaceIone$
be defined as follows. Given a distribution $u\in\LspaceIonecompl$,
we find a $U\in \Lspaceo$ whose restriction is $\Rop_1 U=u$. Then we put
(use Proposition \ref{prop-rest1.1})
\[
\Superoprem u:=
\Rop_1\Superop_2 U\in\Rop_1(\Lspaceo\oplus\Lspaceoper)=\LspaceIone.
\]
\label{def-Sigma2}
\end{defn}

We will call $\Superoprem$ the \emph{compression} of $\Superop_2$.
However, we still need to verify that this definition is consistent, that is,
that the right-hand side $\Rop_1\Superop_2 U$ is independent of the choice of
the extension $U$.

\begin{prop}
The operator $\Superoprem:
\LspaceIonecompl\to\LspaceIone$
is well-defined and bounded. Moreover, we have that
$\Rop_1\Superop_2 U=\Superoprem\Rop_1^\dagger U$ holds for
$U\in\Lspaceo$.
\label{prop-Superoprem}
\end{prop}

\begin{proof}
To see that $\Superoprem$ is well-defined, we need to check
that if $U\in \Lspaceo$ and its restriction to $\R\setminus\bar I_1$
vanishes (this means that $\supp U\subset \bar I_1$), then
$\Rop_1\Superop_2 U=0$. From the definition of the operator $\Superop_2$,
we understand that
\[
\supp\Superop_2 U\subset \supp U+2\Z^\times\subset \bar I_1+2\Z^\times=\R
\setminus I_1.
\]
In particular, the restriction to $I_1$ of $\Superop_2 U$ vanishes,
as required. Similarly, we argue that $\Superoprem$ bounded, based on
Proposition \ref{prop-rest1.1} and Definition \ref{def-Sigma2}. 
Finally, the asserted identity
$\Rop_1\Superop_2 U=\Superoprem\Rop_1^\dagger U$
just expresses how the operator $\Superoprem$ is defined.
\end{proof}

\subsection{Further analysis of the uniqueness problem}
The complementary restriction operators have the following properties:
\begin{equation}
\Rop_{1}^\dagger\Jop_\beta u=\Jop_\beta\Rop_\beta u,\qquad
u\in\Lspaceo,
\label{eq-commut1.01}
\end{equation}
and, for $0<\beta\le\gamma<+\infty$,
\begin{equation}
\Rop_{1}^\dagger\Jop_\beta u=\Jop_\beta\Rop_\beta u,\qquad
u\in{\mathfrak L}(I_\gamma).
\label{eq-commut1.02}
\end{equation}
They will help us analyze further the tentative implication (a) of Lemma 
\ref{lem-equiv8.1.1.002}.

\begin{prop} Fix $0<\beta\le1$.
Suppose that for $u\in\Lspaceo$ we have $\Perop_2 u=0$ and
$\Perop_2\Jop_\beta u=0$.
Then the restrictions $u_0:=\Rop_1 u\in\LspaceIone$ and
$u_1:=\Rop_{1}^\dagger u\in\LspaceIonecompl$ each solve the equations
\[
u_0=\Superoprem\Jop_\beta\Rop_\beta
\Superoprem\Jop_\beta\Rop_\beta u_0,
\qquad u_1=\Rop_1^\dagger\Jop_\beta\Superoprem
\Rop_1^\dagger\Jop_\beta\Superoprem u_1,
\]
and are given in terms of each other by
\[
u_0=-\Superoprem u_1,\qquad u_1=-\Rop_1^\dagger\Jop_\beta
\Superoprem\Jop_\beta\Rop_\beta u_0.
\]
\label{prop-5.3.1}
\end{prop}

\begin{proof}
To begin with, we write the given conditions $\Perop_2 u=0$ and
$\Perop_2\Jop_\beta u=0$ in the form
\[
u=-\Superop_2u,\qquad \Jop_\beta u=-\Superop_2\Jop_\beta u;
\]
after that, we restrict to the interval $I_1$:
\begin{equation*}
\Rop_1u=-\Superoprem\Rop_{1}^\dagger u,
\qquad \Rop_1\Jop_\beta u=\Jop_\beta\Rop^\dagger_\beta u
=-\Superoprem\Rop_{1}^\dagger
\Jop_\beta u=-\Superoprem\,
\Jop_\beta\Rop_\beta u.
%\label{eq-loop-1.01}
\end{equation*}
We then simplify the second condition a little by applying $\Jop_\beta$ to
both sides:
\begin{equation}
\Rop_1u=-\Superoprem\Rop_{1}^\dagger u,
\qquad\Rop^\dagger_\beta u=-\Jop_\beta\Superoprem\,
\Jop_\beta\Rop_\beta u.
\label{eq-loop-1.02}
\end{equation}
By combining these two identities in two separate ways, we find that
\begin{equation}
\Rop_1u=\Superoprem\Jop_\beta\Rop_\beta
\Superoprem\Jop_\beta\Rop_\beta u,
\qquad \Rop^\dagger_\beta u=\Jop_\beta\Superoprem
\Rop_1^\dagger \Jop_\beta\Superoprem\,\Rop_1^\dagger u.
\label{eq-loop-1.03}
\end{equation}
The assertions now follow, if we use \eqref{eq-commut1.01} and
\eqref{eq-commut1.02}.
\end{proof}

\subsection{Two subtransfer operators on spaces of 
distributions}
As usual, we assume that $0<\beta\le1$, and consider the operators
\begin{equation}
\Tope_\beta:=\Superoprem\Jop_\beta\Rop_\beta:\,\,
\LspaceIone\to\LspaceIone,
\label{eq-Sop1.001}
\end{equation}
and
\begin{equation}
\Vop_\beta:=\Rop_1^\dagger\Jop_\beta\Superoprem:\,\,
\LspaceIonecompl\to\LspaceIonecompl.
\label{eq-Tope1.001}
\end{equation}
These operators are \emph{extensions to the respective space of distributions
of standard subtransfer operators}. We met e.g. $\Tope_\beta$ back in Subsection
\ref{subsec-gausstype}.
%\ref{sec-background.dynamics}, see   \eqref{eq-Uop.Wop}.
Indeed, if $u\in L^1(I_1)$ and
$v\in L^1(\R\setminus\bar I_1)$, then
\begin{equation}
\Tope_\beta u(x)=\sum_{j\in\Z^\times}\frac{\beta}{(x+2j)^2}\,u\bigg(-
\frac{\beta}{x+2j}\bigg),\qquad x\in I_1,
\label{eq-Sop1.002}
\end{equation}
and
\begin{equation}
\Vop_\beta v(x)=\frac{\beta}{x^2}\sum_{j\in\Z^\times}v\bigg(-\frac{\beta}{x}+2j
\bigg),\qquad x\in\R\setminus\bar I_1.
\label{eq-Tope1.002}
\end{equation}
In terms of these two subtransfer operators, the formulation of Proposition
\ref{prop-5.3.1} simplifies pleasantly.

\begin{prop} Fix $0<\beta\le1$.
Suppose that for $u\in\Lspaceo$ we have $\Perop_2 u=0$ and
$\Perop_2\Jop_\beta u=0$.
Then the restrictions $u_0:=\Rop_1 u\in\LspaceIone$ and
$u_1:=\Rop_{1}^\dagger u\in\LspaceIonecompl$ satisfy
\[
u_0=\Tope_\beta^2 u_0,
\qquad u_1=\Vop_\beta^2 u_1,\qquad
u_0=-\Superoprem u_1,\qquad u_1=-\Rop_1^\dagger\Jop_\beta
\Tope_\beta u_0.
\]
\label{prop-5.4.1}
\end{prop}

\begin{proof}
The proof is immediate from the  definitions of $\Tope_\beta$ and $\Vop_\beta$. 
\end{proof}

\begin{prop}
Suppose that for $u\in\Lspaceo$ we have that the two restrictions vanish,
i.e., $\Rop_1 u=0$ and $\Rop_{1}^\dagger u=0$ as elements of
$\LspaceIone$ and $\LspaceIonecompl$, respectively. Then $u=0$.
\label{prop-5.4.2}
\end{prop}

\begin{proof}
The assumption implies that the ``valeur au point'' function $\pev[u]$
vanishes on $\R^\times=\R\setminus\{0\}$. But then $\pev[u]$ vanishes a.e.,
so that by Kolmogorov's Proposition \ref{prop-weakL1cont}, 
%\ref{prop-distrptwise1}, 
the claim $u=0$ follows.
\end{proof}

\begin{rem}
Fix $0<\beta\le1$.
Suppose that we are given a distribution $u_0\in\LspaceIone$ which is a fixed
point for the subtransfer operator: $\Tope_\beta^2 u_0=u_0$. Then the formula
\[
u_1:=-\Rop_1^\dagger\Jop_\beta\Tope_\beta u_0
\]
defines a distribution $u_1\in\LspaceIonecompl$. We quickly show that
$u_1=\Vop_\beta^2 u_1$ and $u_0=-\Superoprem u_1$, so that
all the conditions of Proposition \ref{prop-5.4.1} are indeed accounted for.
This means that all the solutions pairs $(u_0,u_1)$ may be parametrized by
the distribution $u_0$ alone.
\end{rem}

\subsection{The subtransfer operator $\Tope_\beta$ acting on
``valeur au point'' functions}
\label{subsec-10.6}
The subtransfer operators $\Tope_\beta$ and $\Vop_\beta$ are defined on
distributions, but the formulas \eqref{eq-Sop1.002} and \eqref{eq-Tope1.002}
often make sense pointwise in the almost everywhere sense for functions
which are not summable on the respective interval. In the sequel, we focus
on $\Tope_\beta$; the case of $\Vop_\beta$ is analogous. The question appears
whether for a given distribution $u\in\LspaceIone$, with ``valeur au point''
function $\pev[u]\in L^{1,\infty}(I_1)$, the action of $\Tope_\beta$ on
$\pev[u]$ by formula \eqref{eq-Sop1.002} when it converges a.e. has the same
result as taking $\pev[\Tope_\beta u]$. To analyze this, we need the finite sum
operators ($N=2,3,4,\ldots$)
\begin{equation}
\Tope_\beta^{[N]}u(x)=\sum_{j\in\Z^\times:|j|\le N}\frac{\beta}{(x+2j)^2}\,u\bigg(-
\frac{\beta}{x+2j}\bigg),\qquad x\in I_1.
\label{eq-Sop1.003}
\end{equation}
This finite sum operator naturally acts both on the distribution $u$ and
on its ``valeur au point'' function $\pev[u]$. As for the distributional
interpretation, it is more properly understood as
\begin{equation}
\Tope_\beta^{[N]}:=\SuperopremN\Jop_\beta\Rop_\beta,
\label{eq-Sop1.003.1}
\end{equation}
where
\[
\SuperopremN:\LspaceIonecompl\to\LspaceIone
\]
is defined in the same fashion as $\Superoprem$ based on
the operator
\[
\Superop_{2,N}V(x):=\sum_{j\in\Z^\times:|j|\le N}V(x+2j),
\]
which maps $\Lspaceo\to\Lspaceo$. Whether we apply the operator
``valeur au point'' before or after $\Tope_\beta^{[N]}$ does not influence the
result:

\begin{prop}
For $u\in \LspaceIone$, we have that
\[
\pev\big[\Tope_\beta^{[N]}u\big](x)=\Tope_\beta^{[N]}\pev[u](x)
\]
almost everywhere on the interval $I_1$.
\end{prop}

\begin{proof}
Since the sum defining $\Tope_\beta^{[N]}u$ is finite, it suffices to handle
a single term. This amounts to showing that
\[
\pev\bigg[\frac{\beta}{(x+2j)^2}\,u\bigg(-
\frac{\beta}{x+2j}\bigg)\bigg]=\frac{\beta}{(x+2j)^2}\,\pev[u]\bigg(-
\frac{\beta}{x+2j}\bigg)
\]
holds almost everywhere on $I_1$, which is elementary.
\end{proof}

We can now show that $\Tope_\beta^{[N]}u$ approximates $\Tope_\beta u$ as $N\to
+\infty$ in terms of the ``valeur au point''.

\begin{prop}
For $u\in\LspaceIone$, we have that $\Tope_\beta^{[N]}(\pev[u])\to
\pev[\Tope_\beta u]$ as $N\to+\infty$ in the quasinorm of
$L^{1,\infty}(I_1)$.
\label{prop-finite1.1}
\end{prop}

\begin{proof}
We use the factorization \eqref{eq-Sop1.003.1}, which says that
$\Tope_\beta^{[N]}=\SuperopremN\Jop_\beta\Rop_\beta$.
For $v\in\LspaceIonecompl$,
we have the convergence
$\SuperopremN v\to\Superoprem v$ in
$\LspaceIone$ as $N\to+\infty$ (cf. the proof of Proposition
\ref{prop-rest1.1}),
which leads to $\Tope_\beta^{[N]}u\to\Tope_\beta u$ in $\LspaceIone$ as
$N\to+\infty$, for fixed $u\in\LspaceIone$.
The asserted convergence now follows from a combination of Proposition
\ref{prop-finite1.1} with the weak-type estimate
(Proposition \ref{prop-weakL1cont}).
\end{proof}

The Hilbert transform $\Hop$ maps
$L^1_0(\R)\to\Hop L^1_0(\R)\subset\Lspaceo$, and the restriction $\Rop_1$
maps $\Lspaceo\to\LspaceIone$, so that
$\Rop_1\Hop$ maps $L^1_0(\R)\to\LspaceIone$. By considering also the function
$P_{\imag}(t)=\pi^{-1}(1+t^2)^{-1}$, which is in $L^1(\R)$ but not in $L^1_0(\R)$,
we realize that $\Rop_1\Hop$ maps $L^1(\R)$ into $\LspaceIone$.
We formalize this as a lemma.

\begin{lem}
The operator $\Rop_1\Hop$ maps $L^1(\R)$ into $\LspaceIone$.
\label{lem-5.7.1}
\end{lem}

\subsection{Norm expansion of the transfer operator on
$\LspaceIone$}
\label{subsec-normexp1}
We now supply the proof of Theorem \ref{thm-nocontract}.

\begin{proof}[Proof of Theorem \ref{thm-nocontract}] 
Since $\Tope_\beta=\Superoprem\Jop_\beta\Rop_\beta$, and $\Rop_\beta$ maps
$\LspaceIone$ into ${\mathfrak{L}}(I_\beta)$ boundedly (more or less as a 
matter of definition), the boundedness of $\Tope_\beta$ is a consequence of
Propositions \ref{prop-mappinJop1.1} and \ref{prop-Superoprem}.

We turn to the assertion that the norm of $\Tope_\beta$ exceeds $1$ as an 
operator on $\LspaceIone$. We recall that the norm on the space $\LspaceIone$
is induced as a quotient norm based on \eqref{eq-normLspaceo}. 
It is straightforward to identify the dual space of $\Lspaceo$ with 
$H^\infty_\stars(\R)$, where the norm on $H^\infty_\stars(\R)$ that is 
dual to \eqref{eq-normLspaceo} is given by
\[
\|g\|_{\stars}:=\max\big(\|g\|_{L^\infty(\R)},
\inf_{c\in\C}\|\tilde\Hop g+c\|_{L^\infty(\R)}\big).
\]
In the same fashion, the dual space of $\LspaceIone$ is identified with
\[
H^\infty_\stars(I_1)=\{g\in H^\infty_\stars(\R):\,\,\supp g\subset \bar I_1\},
\]
and the corresponding norm on $H^\infty_\stars(I_1)$ is $\|\cdot\|_{\stars}$.
Now, by general Functional Analysis, we know that $\|\Tope_1\|=\|\Tope_1^*\|$,
where 
$\Tope_1^*=\Rop_\beta^*\Jop_\beta^*(\Superoprem)^*:H^\infty_\stars(I_1)\to 
H^\infty_\stars(I_1)$ and the space $H^\infty_\stars(I_1)$ is endowed with the norm
$\|\cdot\|_{\stars}$. Moreover, in view of the standard properties of the 
involution $\Jop_\beta^*$, it maps $H^\infty_\stars(\R\setminus I_1)\to
H^\infty_\stars(I_\beta)$ isometrically. In addition, 
$\Rop_\beta^*$ is just the canonical injection 
$H^\infty_\stars(I_\beta)\to H^\infty_\stars(I_1)$, which is isometric as well.
In conclusion, we see that $\|\Tope_1^*\|=\|(\Superoprem)^*\|$, where 
$(\Superoprem)^*$ maps 
$H^\infty_\stars(I_1)\to H^\infty_\stars(\R\setminus\bar I_1)$
and both spaces are endowed with the norm $\|\cdot\|_{\stars}$. Here, 
$H^\infty_\stars(\R\setminus\bar I_1)$ denotes the subspace
\[
H^\infty_\stars(\R\setminus I_1)=\{g\in H^\infty_\stars(\R):\,\,
\supp g\subset \R\setminus I_1\}.
\]
We proceed to show that $\|(\Superoprem)^*\|>1$. It should be mentioned at 
some point that the space $H^\infty_\stars(I_1)$ may be identified with 
$H^\infty_0(\C\setminus\bar I_1)$,
the space of bounded holomorphic functions in the slit plane $\C\setminus I_1$
which also vanish at infinity. The identification is via the Cauchy transform,
it is an isomorphism but it is not isometric; actually, arguably, the 
supremum norm on $\C\setminus I_1$ might be more natural than the norm on 
$H^\infty_\stars(I_1)$ coming from the chosen norm \eqref{eq-normLspaceo} on 
$\LspaceIone$. 
For $0<\gamma\le1$, let us consider the function
\[
G_\gamma(z)=-z^2+z(z+\gamma)\sqrt{\frac{z-\gamma}{z+\gamma}}+\frac{\gamma^2}{2},
\]
where  the  square root is given by the principal branch of the argument in 
$\C\setminus\bar\R_{-}$. Then $G_\gamma\in H^\infty_0(\C\setminus\bar I_1)$, 
and the 
corresponding element of $H^\infty_\stars(I_1)$ is 
$g_\gamma(x):=x\sqrt{\gamma^2-x^2}1_{I_\gamma}(x)$, which is odd,
with Hilbert transform
\[
\Hop g_\gamma(x)=x^2-\frac{\gamma^2}{2}
-1_{\R\setminus I_\gamma}(x)|x|\sqrt{x^2-\gamma^2},
\]
which is even. 
Both $g_\gamma$ and $\Hop g_\gamma$ are H\"older continuous, with 
$\|g_\gamma\|_{L^\infty(\R)}=\frac{1}{2}\gamma^2$
and
\[
\inf_{c\in\C}\|\tilde\Hop g_\gamma+c\|_{L^\infty(\R)}=
\inf_{c_0\in\C}\|\Hop g_\gamma+c_0\|_{L^\infty(\R)}=\|\Hop g_\gamma\|_{L^\infty(\R)}=
\frac{\gamma^2}{2},
\]
which we see from a calculation of the range of the function $\Hop g_\gamma$,
which equals the interval $[-\frac12\gamma^2,\frac12\gamma^2]$. 
This gives that $\|g_\gamma\|_{\stars}=\frac12\gamma^2$. 
We proceed to estimate the norm $\|(\Superoprem)^*g_\gamma\|_{\stars}$ 
from below. From the definition of the operator $\Superoprem$, we see that
\[
(\Superoprem)^*g_\gamma(x)=\sum_{j\in\Z^\times}g_\gamma(x+2j),
\qquad x\in\R\setminus\bar I_1,
\]
and the corresponding Hilbert transform is
\[
\Hop(\Superoprem)^*g_\gamma(x)=\sum_{j\in\Z^\times}\Hop g_\gamma(x+2j)=
\sum_{j=1}^{+\infty}\big(\Hop g_\gamma(x+2j)+\Hop g_\gamma(x-2j)\big),
\qquad x\in\R.
\]
In the sum in the middle it is important to consider symmetric partial sums,
which is reflected in the rightmost expression.
As the sum defining $(\Superoprem)^*g_\gamma(x)$ has at most one nonzero term
for each given $x\in\R$, we see that 
$\|(\Superoprem)^*g_\gamma\|_{L^\infty(\R\setminus I_1)}=\frac12\gamma^2$. 
In order to obtain the norm $\|(\Superoprem)^*g_\gamma\|_\stars$, we proceed to 
evaluate
\[
\inf_{c_0\in\C}\big\|\Hop(\Superoprem)^*g_\gamma(x)-c_0\big\|_{L^\infty(\R)}.
\]
Since the functions involved are H\"older continuous and real-valued, 
we realize that if we may find two points $x_1,x_2\in\R$ with 
\begin{equation}
\Hop(\Superoprem)^*g_\gamma(x_1)-\Hop(\Superoprem)^*g_\gamma(x_2)>\gamma^2,
\label{eq-testcrit1.101}
\end{equation}
then it would follow that
\[
\inf_{c\in\C}\big\|\Hop(\Superoprem)^*g_\gamma(x)-c\big\|_{L^\infty(\R)}>
\frac{\gamma^2}{2},
\]
and as a consequence, $\|\Superoprem\|=\|(\Superoprem)^*\|>1$, as claimed.
We will restrict our attention to values of $\gamma$ that are close to $0$.
Taylor's formula applied to the square root function shows that 
\[
\Hop g_\gamma(x)=\frac{\gamma^4}{8x^2}+\Ordo\bigg(\frac{\gamma^6}{x^4}\bigg)
\]
uniformly for $|x|>1$. Since $\Hop g$ is even, the value at the point 
$x_2:=2$ of the function $\Hop(\Superoprem)^*g_\gamma$ then equals  
\[
\Hop(\Superoprem)^*g_\gamma(x_2)=\Hop(\Superoprem)^*g_\gamma(2)=\Hop g_\gamma(0)+
\Hop g(2)+2\sum_{j=2}^{+\infty}\Hop g_\gamma(2j)
=-\frac{\gamma^2}{2}+\frac{\pi^2-3}{96}
\gamma^4+\Ordo(\gamma^6),
\]
while the value at $x_1:=\gamma+2N$ tends to the following value as 
$N\to+\infty$ through the integers:
\begin{multline*}
\lim_{N\to+\infty}\Hop(\Superoprem)^*g_\gamma(\gamma+2N)=\Hop g_\gamma(\gamma)+
\sum_{j=1}^{+\infty}\big(\Hop g_\gamma(\gamma+2j)+\Hop g_\gamma(2j-\gamma)\big) 
\\
=\frac{\gamma^2}{2}+2\sum_{j=1}^{+\infty}\Hop g_\gamma(2j)+\Ordo(\gamma^6)
=\frac{\gamma^2}{2}+\frac{\pi^2\gamma^4}{96}+\Ordo(\gamma^6).
\end{multline*}
Finally, since 
\[
\lim_{N\to+\infty}\Hop(\Superoprem)^*g_\gamma(\gamma+2N)
-\Hop(\Superoprem)^*g_\gamma(2)=\gamma^2+\frac{\gamma^4}{32}+\Ordo(\gamma^6)
>\gamma^2
\]
for small values of $\gamma$, we obtain \eqref{eq-testcrit1.101} for 
$x_1=\gamma+2N$ and $x_2=2$, provided $\gamma$ is small and the positive integer
$N$ is large.
\end{proof}

%We turn to the pointwise control of the iterates under $\Tope_\beta$ of
%distributions of the form $\Rop_1\Hop g$ for $g\in L^1(\R)$ that vanish
%on $I_\beta$.

\subsection{An operator identity of commutator type}

We recall that by Lemma \ref{lem-5.7.1}, the operator $\Rop_1\Hop$ maps
$L^1(\R)\to\LspaceIone$.

\begin{lem}
Fix $0<\beta\le1$.
For $f\in L^1(I_1)$, extended to vanish off $I_1$,
we have the identity
\[
\Tope_\beta\Rop_1\Hop f=\Rop_1\Hop\Topep_\beta f+\Tope_\beta\Rop_1
\Hop\Jop_\beta\Topep_\beta f
-\Rop_1\Hop\Jop_\beta f,
\]
as elements of the space $\LspaceIone$.
\label{lem-5.7.3}
\end{lem}

\begin{proof}
%To simplify the presentation, we replace the $L^1(I_1)$ function by a Dirac
%point mass $f=\delta_\xi$ at an arbitrary point $\xi\in I_1$. If we can
In line with the presentation in the introduction, in particular, 
\eqref{eq-Topepdef1}, we show that the claimed equality holds for 
$f=\delta_\xi$, i.e.,
\begin{equation}
\Tope_\beta\Rop_1\Hop (\delta_\xi-\Jop_\beta\Topep_\beta\delta_\xi)
=\Rop_1\Hop(\Topep_\beta \delta_\xi-\Jop_\beta \delta_\xi)
\label{eq-9.oper1.00}
\end{equation}
holds, for almost every $\xi\in I_1$. The equality then holds for all
$f\in L^1(I_1)$ by ``averaging'', as in \eqref{eq-Topepdef1}.
%\eqref{eq-ptmassinterpret1}.
The canonical extension of the involution $\Jop_\beta$
and the transfer operator $\Topep_\beta$ to such point masses
$\delta_\xi$ reads:
\begin{equation}
\Jop_\beta\delta_\xi=\delta_{-\beta/\xi},\qquad\Topep_\beta\delta_\xi=
\delta_{\{-\beta/\xi\}_2},
\label{eq-oper1.00.05}
\end{equation}
where, as in Subsection \ref{subsec-notation-intervals}, the expression 
$\{t\}_2$ stands for the real number in the interval $]\!\!-\!1,1]$ with 
the property that $t-\{t\}_2\in2\Z$. It follows that
\begin{equation}
\Topep_\beta\delta_\xi-\Jop_\beta\delta_\xi=\delta_{\{-\beta/\xi\}_2}-\delta_{-\beta/\xi},
\qquad \Jop_\beta\Topep_\beta\delta_\xi=\delta_{-\beta/\{-\beta/\xi\}_2},
\label{eq-9.oper1.01}
\end{equation}
so that \emph{for $\xi\in I_1\setminus\bar I_\beta$},
\[
\delta_{\xi}-\Jop_\beta\Topep_\beta\delta_\xi=0\quad\text{and}\quad
\Topep_\beta\delta_\xi-\Jop_\beta\delta_\xi=0.
\]
It follows that for $\xi\in I_1\setminus\bar I_\beta$, both the left-hand
and the right-hand sides of the claimed equality \eqref{eq-9.oper1.00} vanish,
and the equality is trivially true.
It remains to consider $\xi\in\bar I_\beta$. For $\eta\in\R$, the canonical
extension of the Hilbert transform to a Dirac point mass at $\eta$ is
\[
\Hop \delta_\eta =\frac1{\pi}\pv\frac{1}{x-\eta},
\]
and we calculate that for two points $\eta,\eta'\in\R^\times$,
\[
\Tope_\beta\Rop_1\Hop(\delta_\eta-\delta_{\eta'})=
\frac1{\pi}\pv\sum_{j\in\Z^\times}\Bigg(\frac{1}{x+2j+\frac{\beta}{\eta}}
-\frac{1}{x+2j+\frac{\beta}{\eta'}}\Bigg)\quad\text{on}\,\,\,I_1;
\]
here, we may observe that the \emph{principal value} interpretation is only
needed with respect to at most two terms of the series. A particular instance
is when
\[
\frac{\beta}{\eta'}=\frac{\beta}{\eta}-2k,\quad\text{for some}\,\,\, k\in\Z,
\]
in which case we get telescopic cancellation:
\[
\Tope_\beta\Rop_1\Hop[\delta_\eta-\delta_{\eta'}]=
\frac1{\pi}\pv\sum_{j\in\Z^\times}\Bigg(\frac{1}{x+2j+\frac{\beta}{\eta}}
-\frac{1}{x+2(j-k)+\frac{\beta}{\eta}}\Bigg)=
\frac{1}{\pi}\pv\Bigg\{\frac{1}{x-2k+\frac{\beta}{\eta}}
-\frac{1}{x+\frac{\beta}{\eta}}\Bigg\}
\]
on the interval $I_1$. We apply this to the case $\eta:=\xi\in I_1$ and
$\eta':=-\beta/\{-\beta/\xi\}_2$, in which case $k\in\Z$ is given by
\[
2k=\frac{\beta}{\xi}+\{-\beta/\xi\}_2,
\]
and obtain that
\begin{equation}
\Tope_\beta\Rop_1\Hop(\delta_\xi-\delta_{-\beta/\{-\beta/\xi\}_2})=
\frac{1}{\pi}\pv\Bigg\{\frac{1}{x-\{-\beta/\xi\}_2}
-\frac{1}{x+\frac{\beta}{\xi}}\Bigg\}\quad\text{on}\,\,\, I_1.
\label{eq-9.00011}
\end{equation}
The natural requirements that $\xi\ne0$ and that $\{-\beta/\xi\}_2\ne0$
excludes a countable collection of $\xi\in I_1$, which has Lebesgue measure
$0$.
By \eqref{eq-9.oper1.01}, this is the left-hand side expression of
\eqref{eq-9.oper1.00}, and another application \eqref{eq-9.oper1.01} gives
that the right-hand side expression of \eqref{eq-9.oper1.00} equals
\begin{equation}
\Rop_1\Hop(\delta_{\{-\beta/\xi\}_2}-\delta_{-\beta/\xi})=
\frac{1}{\pi}\pv\Bigg\{\frac{1}{x-\{-\beta/\xi\}_2}
-\frac{1}{x+\frac{\beta}{\xi}}\Bigg\}\quad\text{on}\,\,\, I_1.
\label{eq-9.00012}
\end{equation}
From equations \eqref{eq-9.00011} and \eqref{eq-9.00012}, together
with \eqref{eq-9.oper1.01}, we find that the claimed identity 
\eqref{eq-9.oper1.00} is correct for almost every $\xi\in I_1$.
\end{proof}

\begin{prop}
Fix $0<\beta\le1$.
For $f\in L^1(I_1)$, extended to vanish off $I_1$,
we have the identity
\[
\Tope_\beta^n\Rop_1\Hop f=\Rop_1\Hop\Topep_\beta^n f
+\sum_{j=0}^{n-1}\Big\{\Tope_\beta^{n-j}\Rop_1\Hop\Jop_\beta
\Topep_\beta^{j+1}f-\Tope_\beta^{n-j-1}\Rop_1\Hop\Jop_\beta
\Topep_\beta^{j}f\Big\},
\]
as elements of the space $\LspaceIone$, for $n=2,3,4,\ldots$.
\label{prop-5.7.3}
\end{prop}

\begin{proof}
We argue by induction. First, the identity actually holds for $n=1$, by
Lemma \ref{lem-5.7.3}; here, the sum from $j=0$ to $j=-1$ should be understood
as $0$.

Next, we assume that the identity is valid for $n=k$, and would like to show
that it holds for $n=k+1$ as well.
From the induction hypothesis, we know that
\begin{multline}
\Tope_\beta^{k+1}\Rop_1\Hop f=\Tope_\beta\Tope_\beta^{k}\Rop_1\Hop f
=\Tope_\beta\Rop_1\Hop\Topep_\beta^k f
\\
+\sum_{j=0}^{k-1}\Big\{\Tope_\beta^{k-j+1}\Rop_1\Hop\Jop_\beta
\Topep_\beta^{j+1}f-\Tope_\beta^{k-j}\Rop_1\Hop\Jop_\beta
\Topep_\beta^{j}f\Big\}.
\label{eq-5.7.2}
\end{multline}
In view of Lemma \ref{lem-5.7.3},
\[
\Tope_\beta\Rop_1\Hop\Topep_\beta^k f=
\Rop_1\Hop\Topep_\beta^{k+1} f+\Tope_\beta\Rop_1
\Hop\Jop_\beta\Topep_\beta^{k+1} f
-\Rop_1\Hop\Jop_\beta\Topep_\beta^k f,
\]
and applied to \eqref{eq-5.7.2} we obtain that
\begin{multline*}
\Tope_\beta^{k+1}\Rop_1\Hop f
=\Rop_1\Hop\Topep_\beta^{k+1} f+\Tope_\beta\Rop_1
\Hop\Jop_\beta\Topep_\beta^{k+1} f
-\Rop_1\Hop\Jop_\beta\Topep_\beta^k f
\\
+\sum_{j=0}^{k-1}\Big\{\Tope_\beta^{k-j+1}\Rop_1\Hop\Jop_\beta
\Topep_\beta^{j+1}f-\Tope_\beta^{k-j}\Rop_1\Hop\Jop_\beta
\Topep_\beta^{j}f\Big\}
\\
=\Rop_1\Hop\Topep_\beta^{k+1} f
+\sum_{j=0}^{k}\Big\{\Tope_\beta^{k-j+1}\Rop_1\Hop\Jop_\beta
\Topep_\beta^{j+1}f-\Tope_\beta^{k-j}\Rop_1\Hop\Jop_\beta
\Topep_\beta^{j}f\Big\}
\end{multline*}
This is the desired identity for $n=k+1$, which completes the proof.
\end{proof}

\section{$\Tope_\beta$-iterates of Hilbert transforms}
\label{sec-iterates1}

\subsection{Smooth Hilbert transforms}
\label{subsec-smoothH}
We fix  $0<\beta\le1$.  Recall that for a  function
$g\in L^1(\R)$, its Hilbert transform is
\begin{equation}
\Hop g(x)=\frac{1}{\pi}\pv\int_{\R}\frac{g(t)}{x-t}\diff t,\qquad x\in\R.
\label{eq-Hilb:def1.01}
\end{equation}
In here, \emph{we are interested in the specific
case when the function $g$ vanishes on the interval $I_\beta$}. Then the
Hilbert transform $\Hop g$ is smooth on $I_\beta$, and there is no need to
considering principal values when we restrict our attention to $I_\beta$.
In terms of the involution
\begin{equation}
\Jop_\beta g(x)=\frac{\beta}{x^{2}}\,g\bigg(-\frac{\beta}{x}\bigg),
\label{eq-Jopbeta.def1}
\end{equation}
we see that $\Jop_\beta g\in L^1(I_1)$ and that
\begin{equation}
\Hop g(x)=\frac{1}{\pi}\int_{\R\setminus I_\beta}\frac{g(t)}{x-t}\diff t
=\frac{1}{\pi}\int_{I_1}\frac{t}{\beta+tx}\Jop_\beta g(t) \diff t,\qquad
x\in I_\beta;
\label{eq-Hilb:naive1.01}
\end{equation}
the advantage is that we now integrate over the symmetric unit interval $I_1$.
In terms of the kernel
\[
\Kfun_\beta(t,x):=\frac{t}{\beta+tx}
\]
and the associated integral operator
\begin{equation}
\Qop_\beta f(x):=\frac{1}{\pi}\int_{I_1}\Kfun_\beta(t,x)f(t)\diff t
=\frac{1}{\pi}\int_{I_1}\frac{t}{\beta+tx}\,f(t)\diff t,\qquad x\in I_\beta,
\label{eq-Kbeta.def1}
\end{equation}
\eqref{eq-Hilb:naive1.01} simply asserts that
\begin{equation}
\Qop_\beta f(x)=\Hop\Jop_\beta f(x),\qquad x\in I_\beta,
\label{eq-Hilb-Kb1.01}
\end{equation}
for $f\in L^1(I_1)$, extended to vanish off $I_1$.
It is elementary to estimate that
\begin{equation}
|\Kfun_\beta(t,x)|=\frac{|t|}{\beta+tx}\le\frac{2\beta}{\beta^2-x^2}
=2\kappa_\beta(x),\qquad x\in I_\beta,\,\,\,\,t\in\bar I_1,
\label{eq-5.000102}
\end{equation}
where $\kappa_\beta$ is as in \eqref{eq-kappadef1}, which yields that
\begin{equation}
|\Qop_\beta f(x)|\le\frac{1}{\pi}\int_{I_1}|\Kfun_\beta(t,x)f(t)|\diff t
\le\frac{2}{\pi}\,\|f\|_{L^1(I_1)}\kappa_\beta(x),\qquad x\in I_\beta.
\label{eq-Hilb-Kb1.02}
\end{equation}
In general,    $\Qop_\beta f$  is not in  $L^1(I_\beta)$.
But at least \eqref{eq-Hilb-Kb1.02} guarantees that $\Qop_\beta f$ is
well-defined pointwise with an effective bound.
We will want to consider the $\Tope_\beta$-iterates of the function
$\Qop_\beta f$. Since,  as a matter of fact,  the subtransfer operator
$\Tope_\beta$ only cares about the values of the function in question on the
interval $I_\beta$, we may use the above estimate \eqref{eq-Hilb-Kb1.02}
together with the observation made in Subsection \ref{subsec-3.5}
to see that the $\Tope_\beta$-iterates of $\Qop_\beta f$ are well-defined
pointwise. We are also able to supply an effective estimate of those
iterates, which we first do for $0<\beta<1$.

\begin{prop} Fix $0<\beta<1$.
Suppose $f\in L^1(I_1)$. Then we have the estimate
\[
|\Tope_\beta^{n}\Qop_\beta f(x)|\le\frac{4\beta^{n-1}}{\pi(1-\beta)}\|f\|_{L^1(I_1)},
\qquad x\in I_1,\quad n=2,3,4,\ldots,
\]
so that
$\Tope_\beta^{n}\Qop_\beta f\to0$ geometrically as $n\to+\infty$, uniformly on
the interval $I_1$.
\label{prop-5.6.5}
\end{prop}

\begin{proof}
As observed above, pretty much by definition, $\Tope_\beta g$ is only concerned
with the behavior of $g$ on the interval $I_\beta$. It follows from the
positivity of the operator $\Tope_\beta$ that
\begin{equation}
|\Tope_\beta\Qop_\beta f(x)|\le
\frac{2}{\pi}\,\|f\|_{L^1(I_1)}\Tope_\beta\kappa_\beta(x)
=\frac{2}{\pi}\,\|f\|_{L^1(I_1)}\kappa_1(x),\qquad x\in I_1,
\label{eq-UbKb1.1}
\end{equation}
where in the last step, we used Lemma \ref{prop-kappa1}.
Now, the same type of argument, based on Proposition \ref{prop-kappa1},
yields
\[
|\Tope_\beta^n\Qop_\beta f(x)|\le
\frac{2}{\pi}\,\|f\|_{L^1(I_1)}\Tope_\beta^{n-1}\kappa_1(x)
\le\frac{4\beta^{n-1}}{\pi(1-\beta)}\,\|f\|_{L^1(I_1)},\qquad x\in I_1,
\]
as claimed.
\end{proof}

For $\beta=1$, the situation is  slightly more delicate.

\begin{prop} Fix $\beta=1$.
Suppose $f\in L^1(I_1)$. We then have the estimate
\[
|\Tope_1^n\Qop_1 f(x)|\le\frac{2}{\pi}(1-x^2)^{-1}\|f\|_{L^1(I_1)},
\qquad x\in I_1,\quad n=1,2,3,\ldots,
\]
and in addition, $\Tope_1^n\Qop_1 f(x)\to0$ as $n\to+\infty$,
uniformly on compact subsets of $I_1$.
\label{prop-5.6.5'}
\end{prop}

\begin{proof}
The derivation of \eqref{eq-UbKb1.1} applies also in the case $\beta=1$, so
that
\begin{equation}
|\Tope_1\Qop_1 f(x)|\le
\frac{2}{\pi}\,\|f\|_{L^1(I_1)}\Tope\kappa_1(x)
=\frac{2}{\pi}\,\|f\|_{L^1(I_1)}\kappa_1(x),\qquad x\in I_1,
\label{eq-UbKb1.2}
\end{equation}
which is the claimed estimate for $n=1$. For $n>1$, we use the positivity
of $\Tope_1$ again, to obtain from \eqref{eq-UbKb1.2} that
\begin{equation}
|\Tope_1^n\Qop_1 f(x)|\le
\frac{2}{\pi}\,\|f\|_{L^1(I_1)}\Tope^{n-1}\kappa_1(x)
=\frac{2}{\pi}\,\|f\|_{L^1(I_1)}\kappa_1(x),\qquad x\in I_1,
\label{eq-UbKb1.3}
\end{equation}
which establishes the claimed estimate.

We proceed to obtain the uniform convergence to $0$ locally on compact
subsets of $I_1$. To this end, we use the representation \eqref{eq-Kbeta.def1}
to see that
\begin{equation}
\Tope_1^n\Qop_1 f(x)=\frac{1}{\pi}\int_{I_1}\Tope_1^{n}\Kfun_1(t,\cdot)(x)
f(t)\diff t.
\label{eq-U1K1.iter1}
\end{equation}
We verify that for $0<a<1$,
\[
|\Kfun_1(t,x)|\le \Kfun_1(a,x)=\frac{a}{1+ax},\qquad t\in[0,a],\,\,\,x\in I_1,
\]
and that
\[
|\Kfun_1(t,x)|\le -\Kfun_1(-a,x)=\frac{a}{1-ax},\qquad
t\in[-a,0],\,\,\,x\in I_1.
\]
As a consequence, using the positivity of $\Tope_1$, we may derive that
\begin{equation*}
|\Tope_1^n \Kfun_1(t,\cdot)(x)|\le \Tope_1^n\Kfun_1(a,\cdot)(x)
\le2\kappa_1(x),\qquad t\in[0,a],\,\,\,x\in I_1,
%\label{eq-UbKb1.4}
\end{equation*}
and that
\begin{equation*}
|\Tope_1^n \Kfun_1(t,\cdot)(x)|\le \Tope_1^n
(-\Kfun_1(-a,\cdot))(x)\le2\kappa_1(x),
\qquad t\in[-a,0],\,\,\,x\in I_1,
%\label{eq-UbKb1.5}
\end{equation*}
Next, we apply the triangle inequality to the integral \eqref{eq-U1K1.iter1}:
\begin{multline}
|\Tope_1^n\Qop_1 f(x)|\le\frac{1}{\pi}\int_{I_a}|\Tope_1^n\Kfun_1(t,\cdot)(x)
f(t)|\diff t+\frac{1}{\pi}\int_{I_1\setminus I_a}
|\Tope_1^n\Kfun_1(t,\cdot)(x)f(t)|\diff t
\\
\le\frac{1}{\pi}\Tope_1^n\Kfun_1(a,\cdot)(x)\int_{[0,a]}
|f(t)|\diff t+\frac{1}{\pi}\Tope_1^n(-\Kfun_1(-a,\cdot))(x)\int_{[-a,0]}
|f(t)|\diff t
+\frac{2}{\pi}\kappa_1(x)\int_{I_1\setminus I_a}
|f(t)|\diff t.
\label{eq-iterates1.101}
\end{multline}
Note that in the last term, we used the estimate \eqref{eq-5.000102}
with $\beta=1$. By Proposition \ref{prop-ergodicity1.01}(vi),
$\Tope_1^n \Kfun_1(a,\cdot)\to0$ and $\Tope_1^n \Kfun_1(-a,\cdot)\to0$
as $n\to+\infty$ in the $L^1$ sense on compact subintervals
of $I_1$. It is a consequence of the regularity of the functions
$\Kfun_1(a,\cdot)$ and $\Kfun_1(-a,\cdot)$ that the convergence is actually
uniform on compact subintervals.
By fixing $a$ so close to $1$ that the rightmost integral of
\eqref{eq-iterates1.101} is as small as we like, we see that
$\Tope_1^n\Qop_1 f\to0$ as $n\to+\infty$, uniformly on compact subsets of
$I_1$.
This completes the proof.
\end{proof}

\section{Asymptotic decay of the $\Tope_\beta$-orbit of a
distribution in $\LspaceIone$ for $0<\beta<1$}
\label{sec-asymptotic.decay<1}

\subsection{An application of asymptotic decay for $0<\beta<1$}
\label{subsec-thm-2.0-beta<1}

%The following result is a key step in our analysis. It states that although
%$\Tope_\beta$ is not a norm contraction on $\LspaceIone$, in a weak
%asymptotic sense it behaves like one.

%\begin{thm}
%Fix $0<\beta<1$, and suppose $u_0\in\LspaceIone$. Then as $N\to+\infty$, 
%$\pev[\Tope_\beta^N u_0]\to0$ in $L^{1,\infty}(I_1)$.
%\label{thm-decay<1}
%\end{thm}

We now supply the argument which
shows how, in the subcritical parameter regime $\alpha\beta<1$,
Theorem \ref{thm-2.0} follows from the asymptotic decay 
result Theorem \ref{thm-basic1.001}, which is of extended ergodicity type.

\begin{proof}[Proof of Theorem \ref{thm-2.0} for $\alpha\beta<1$]
As observed right after the formulation of Theorem \ref{thm-2.0},
a scaling argument allows us to reduce the redundancy and \emph{fix
$\alpha=1$, in which case the condition $0<\alpha\beta<1$ reads
$0<\beta<1$}.
In view of Subsections \ref{subsec-dualform} and \ref{subsec-reformulation2}, 
it will be sufficient to show that for $u\in\Lspaceo$,
\begin{equation}
\Perop_2 u=\Perop_2\Jop_\beta u=0
\quad\Longrightarrow\quad u=0.
\label{eq-impli1.1}
\end{equation}
So, we assume that $u\in\Lspaceo$ has $\Perop_2 u=\Perop_2\Jop_\beta u=0$.
Let $u_0:=\Rop_1 u\in\LspaceIone$ and $u_1:=\Rop_1^\dagger u\in\LspaceIonecompl$
denote the restrictions of the distribution $u$ to the symmetric interval
$I_1$ and to the complement $\R\setminus\bar I_1$, respectively.
\emph{We will be done once we are able to show that $u_0=0$}, because then
$u_1$ vanishes as well, as a result of Proposition \ref{prop-5.4.1}:
\[
u_1=-\Rop_1^\dagger\Jop_\beta\Tope_\beta u_0=0.
\]
Indeed, we have Proposition \ref{prop-5.4.2}, which tells us that
$u_0=\Rop_1 u=0$ and $u_1=\Rop_1^\dagger u=0$ together imply that $u=0$.

Finally, to obtain that $u_0=0$, we observe that in addition, Proposition
\ref{prop-5.4.1} says that $u_0$ has the important property
$u_0=\Tope_\beta^2 u_0$. By iteration, then, we have $u_0=\Tope_\beta^{2n} u_0$
for $n=1,2,3,\ldots$, and by letting $n\to+\infty$,
Theorem \ref{thm-basic1.001} tells us that $u_0=0$ is the only solution in
$\LspaceIone$, which completes the proof.
\end{proof}

\subsection{The proof of the asymptotic decay result for 
$0<\beta<1$}
\label{subsec-thm-basic1.001}
We now proceed with the proof of Theorem \ref{thm-basic1.001}. Note that we
have to be particularly careful because the operator 
$\Tope_\beta:\LspaceIone\to\LspaceIone$ has norm $>1$, by 
Theorem \ref{thm-nocontract}. 
However, it clear that it acts contractively on the subspace $L^1(I_1)$.

\begin{proof}[Proof of Theorem \ref{thm-basic1.001}]
We decompose $u_0=f+\Rop_1\Hop g$, where $f\in L^1(I_1)$ and $g\in L^1_0(\R)$,
and observe that by Proposition \ref{prop-ergodicity1.01}(iv),
\begin{equation}
\|\Tope_\beta^{N} f\|_{L^1(I_1)}\to0\quad\text{as}\,\,\,\,N\to+\infty.
\label{eq-est1.0001}
\end{equation}
So the iterates \emph{$\Tope_\beta^{N} f$ tend to $0$ in $L^1(I_1)$ and hence in
$L^{1,\infty}(I_1)$ as well}. We turn to the $\Tope_\beta^2$-iterates of
$\Rop_1\Hop g$. First, we split
\[
g=g_1+g_2, \quad\text{where}\quad g_1\in L^1(I_\beta),\,\,\,
g_2\in L^1(\R\setminus I_\beta);
\]
here, it is tacitly assumed that the functions $g_1,g_2$ are extended to
vanish on the rest of the real line $\R$. As the operator $\Jop_\beta$ maps
$L^1(\R\setminus I_\beta)\to L^1(I_1)$ isometrically, and
$\Hop g_2=\Qop_\beta\Jop_\beta g_2$ holds on $I_1$ by \eqref{eq-Hilb-Kb1.01},
Proposition \ref{prop-5.6.5} gives us the pointwise estimate (we write
``$\pev$'' although it is not absolutely needed)
\begin{equation}
\big|\pev[\Tope_\beta^{N}\Rop_1\Hop g_2](x)\big|\le
\frac{4\beta^{N-1}}{(1-\beta)\pi}\|g_2\|_{L^1(\R)},\qquad x\in I_1.
\label{eq-est1.0002}
\end{equation}
In particular, \emph{the $\Tope_\beta$-iterates of $\Rop_1\Hop g_2$ tend to $0$
geometrically in $L^{\infty}(I_1)$}.
We still need to analyze the
$\Tope_\beta^2$-iterates of $\Rop_1\Hop g_1$.
We apply $\Tope_\beta^{k}$ to the two sides of the
identity of Proposition \ref{prop-5.7.3}, with $g_1$ in place of $f$ and with
$k=2,3,4,\ldots$, to obtain that
\begin{equation}
\Tope_\beta^{n+k}\Rop_1\Hop g_1
=\Tope_\beta^{k}\Rop_1\Hop\Topep_\beta^n g_1
+\sum_{j=0}^{n-1}\Big\{\Tope_\beta^{n+k-j}\Rop_1\Hop\Jop_\beta
\Topep_\beta^{j+1}g_1-\Tope_\beta^{n+k-j-1}\Rop_1\Hop\Jop_\beta
\Topep_\beta^{j}g_1\Big\}.
\label{eq-est1.0002.1}
\end{equation}
For $l=0,1,2,\ldots$, the function $\Topep_\beta^{l} g_1$ is in
$L^1(I_1)$, so that again by Proposition
\ref{prop-5.6.5},  since $\Rop_1\Hop\Jop_\beta=\Qop_\beta$, we have
\begin{equation}
\big|\pev[\Tope_\beta^{r}\Rop_1 \Hop\Jop_\beta\Topep_\beta^l g_1](x)\big|
\le\frac{4\beta^{r-1}}{\pi(1-\beta)}\|g_1\|_{L^1(I_\beta)},\qquad
x\in I_1,\,\,\,r=2,3,4,\ldots,
\label{eq-est1.0003}
\end{equation}
where we use that the transfer operator $\Topep_\beta$ acts contractively on
$L^1(I_1)$, by Proposition \ref{prop-ergodicity1.01}(i).
An application of the ``valeur au point'' estimate \eqref{eq-est1.0003} to
each term of the sum on the right-hand side of the identity
\eqref{eq-est1.0002.1} gives that
\begin{equation}
\big|\pev[\Tope_\beta^{n+k}\Rop_1\Hop g_1
-\Tope_\beta^{k}\Rop_1\Hop\Topep_\beta^n g_1](x)\big|\le
\frac{8\beta^{k-1}}{\pi(1-\beta)^2}\,\|g_1\|_{L^1(I_\beta)},\qquad
\almostev\,\, x\in I_1.
\label{eq-est1.0007}
\end{equation}
Next, we split the function $g_1$ as follows:
\[
g_1=h_{0,n}+h_{1,n},\qquad h_{0,n}\in L^1(\calE_{\beta,n+1}),\,\,\,h_{1,n}\in
L^1(I_\beta\setminus \calE_{\beta,n+1}),
\]
where the set $\calE_{\beta,n+1}$ is as in \eqref{eq-EsetN}, and with the
understanding that $h_{0,n},h_{1,n}$ both vanish elsewhere on the real line.
Next, we observe that $\Topep_\beta^n h_{1,n}\in L^1(I_1\setminus\bar I_\beta)$.
This can be seen from the defining property of the set $\calE_{\beta,n+1}$ and 
the relation between the map $\tau_\beta$ and the corresponding transfer 
operator $\Topep_\beta$, see \eqref{eq-Topepdef1}.
We then apply Proposition \ref{prop-5.6.5} to arrive at
\begin{multline}
\big|\pev[\Tope_\beta^{k}\Rop_1\Hop\Topep_\beta^n h_{1,n}](x)\big|\le
\frac{4\beta^{k-1}}{\pi(1-\beta)}\|\Topep_\beta^n h_{1,n}\|_{L^1(I_1)}
\\
\le \frac{4\beta^{k-1}}{\pi(1-\beta) }\|h_{1,n}\|_{L^1(I_\beta)}\le
\frac{4\beta^{k-1}}{\pi(1-\beta)}\|g_1\|_{L^1(I_\beta)},\qquad
x\in I_1.
\label{eq-est1.0008}
\end{multline}
By combining \eqref{eq-est1.0007} with the estimate \eqref{eq-est1.0008},
we obtain that
\begin{equation}
\big|\pev[\Tope_\beta^{n+k}\Rop_1\Hop g_1
-\Tope_\beta^{k}\Rop_1\Hop\Topep_\beta^n h_{0,n}](x)\big|\le
\frac{12\beta^{k-1}}{\pi(1-\beta)^2}
\,\|g_1\|_{L^1(I_\beta)},\qquad \almostev\,\,x\in I_1.
\label{eq-est1.0009}
\end{equation}
The norm of $h_{0,n}\in L^1(\calE_{\beta,n+1})$ equals
\[
\int_{\calE_{\beta,n+1}}|h_{0,n}(t)|\diff t=\int_{\calE_{\beta,n+1}}|g_1(t)|\diff t=
\|1_{\calE_{\beta,n+1}}g_1\|_{L^1(I_1)},
\]
and it approaches $0$ as $n\to+\infty$, by Proposition 
\ref{prop-ergodicity1.01}(iv).
Since the transfer operator $\Topep_\beta$ is a norm contraction on
$L^1(I_1)$, we know that
\[
\|\Topep_\beta^n h_{0,n}\|_{L^1(I_1)}\le\|h_{0,n}\|_{L^1(I_1)}=
\|1_{\calE_{\beta,n+1}}g_1\|_{L^1(I_1)},
\]
and, consequently, for fixed $k$ we have that
\[
\Tope_\beta^{k}\Rop_1\Hop\Topep_\beta^n h_{0,n}\to0\quad \text{in}\quad
\LspaceIone,\quad\,\,\text{as}\quad n\to+\infty.
\]
As convergence in $\LspaceIone$ entails convergence in $L^{1,\infty}(I_1)$
for the corresponding ``valueur au point'' function, we obtain from
\eqref{eq-est1.0009},  by application of the $L^{1,\infty}(I_1)$ quasinorm
triangle inequality,  that
\begin{equation}
\limsup_{n\to+\infty}\big\|\pev[\Tope_\beta^{n+k}\Rop_1\Hop g_1]
\big\|_{L^{1,\infty}(I_1)}\le\frac{24\beta^{k-1}}{\pi(1-\beta)^2}
\,\|g_1\|_{L^1(I_\beta)},\qquad \almostev\,\,\,x\in I_1.
\label{eq-est1.0010}
\end{equation}
Note that the limit on the left-hand side
\emph{does not depend on the parameter $k$}.
This permits us to let $k\to+\infty$ in a second step, and we obtain that
\begin{equation}
\lim_{N\to+\infty}\big\|\pev[\Tope_\beta^{N}\Rop_1\Hop g_1]\big\|_{L^{1,\infty}(I_1)}=0.
\label{eq-est1.0011}
\end{equation}
Finally, gathering the terms, we obtain from \eqref{eq-est1.0001},
\eqref{eq-est1.0002}, and \eqref{eq-est1.0011}, that
\begin{multline}
\pev[\Tope_\beta^{N}u_0]=\pev[\Tope_\beta^{N}(f+\Rop_1\Hop g)]
\\
=\pev[\Tope_\beta^{N}f]+\pev[\Tope_\beta^{N}\Rop_1\Hop g_1]
+\pev[\Tope_\beta^{N}\Rop_1\Hop g_2]\to0\quad\text{as}\quad N\to+\infty,
\label{eq-est1.0013}
\end{multline}
in the quasinorm of $L^{1,\infty}(I_1)$, as claimed.
\end{proof}

\begin{rem}
 \label{example}
 One may wonder if Theorem  \ref{thm-basic1.001} (and hence Corollary 
\ref{eq-ergodic0.00001}) would remain 
true if the space $\Lspaceo$ were to be replaced by the larger space
$L^{1,\infty}(I_1)$.
To look into this issue, we keep $0<\beta<1$, and consider the function
\[
f(x):=\frac{1}{x-x_1}-\frac{1}{x-x_2},\quad\text{where}\quad
x_1:=1+\sqrt{1-\beta},\,\,\,x_2:=-1+\sqrt{1-\beta}.
\]
Then $x_1x_2=-\beta$, so that $\frac{\beta}{x_1}=-x_2$ and
$\frac{\beta}{x_2}=-x_1$, and, in addition,
$\frac{\beta}{x_1}-\frac{\beta}{x_2}=x_1-x_2=2$, which leads to
\begin{multline*}
\Tope_\beta f(x)=\sum_{j\in \Z^\times}\frac{\beta}{(2j+x)^2}
f\bigg(-\frac{\beta}{2j+x}\bigg)=\sum_{j\in \Z^\times}\frac{\beta}{(2j+x)^2}
\Bigg(\frac{1}{-\frac{\beta}{2j+x}-x_1}-\frac{1}{-\frac{\beta}{2j+x}-x_2}\Bigg)
\\
=\sum_{j\in \Z^\times}\Bigg(\frac{\beta}{(2j+x)(\beta+(2j+x)x_2)}-
\frac{\beta}{(2j+x)(\beta+(2j+x)x_1)}\Bigg)
\\
=\sum_{j\in \Z^\times}\Bigg(\frac{x_1}{\beta+(2j+x)x_1}-
\frac{x_2}{\beta+(2j+x)x_2}\Bigg)
=\sum_{j\in \Z^\times}\Bigg(\frac{1}{2j+x+\frac{\beta}{x_1}}-
\frac{1}{2j+x+\frac{\beta}{x_2}}\Bigg)
\\
=\frac{1}{x+\frac{\beta}{x_2}}-
\frac{1}{x+\frac{\beta}{x_1}}=\frac{1}{x-x_1}-\frac{1}{x-x_2}=f(x),
\end{multline*}
by telescoping sums. The function $f$ is a nontrivial element of
$L^{1,\infty}(I_1)$ and it is $\Tope_\beta$-invariant: $\Tope_\beta f=f$.
Many other choices of the points $x_1,x_2$ would work as well. For $\beta=1$,
the indicated points $x_1,x_2$ coincide, so that $f=0$, but it is enough 
to choose instead $x_1:=2+\sqrt{3}$ and $x_2=-2+\sqrt{3}$ to obtain a 
nontrivial function $f$ in $L^{1,\infty}(I_1)$ which is $\Tope_1$-invariant. 
This illustrates how Theorem \ref{thm-basic1.001} and Corollary 
\ref{eq-ergodic0.00001} would utterly fail to hold if the space $\Lspaceo$ 
were to be replaced by $L^{1,\infty}(I_1)$.
\label{rmk-16.2.1}
\end{rem}

\section{The Hilbert kernel and its dynamical decomposition}
\label{sec-Hilbkernel}

\subsection{Odd and even parts of the Hilbert kernel}
As in subsection \ref{subsec-smoothH}, we write
\[
\Kfun_1(t,x):=\frac{t}{1+tx},
\]
which is  a variant of  the \emph{Hilbert kernel}. Indeed, it arises in
connection with the Hilbert transform, see e.g. \eqref{eq-Hilb-Kb1.01}.
We split the function $\Kfun_1$ according to odd and even parts:
\[
\Kfun_1(t,x)=\Kfun_1^{I}(t,x)-\Kfun_1^{II}(t,x),\qquad
\quad \Kfun_1^{I}(t,x):=\frac{t}{1-x^2t^2},\quad
\Kfun_1^{II}(t,x):=\frac{t^2x}{1-t^2x^2}.
\]
For fixed $t\in I_1=]\!-\!1,1[$, we may calculate the action of the
transfer operator $\Tope_1$ on the function $\Kfun_1(t,\cdot)$ using standard
trigonometric identities:
\begin{multline}
\Tope_1\Kfun_1(t,\cdot)(x)=\sum_{j\in\Z^\times}\frac{1}{(2j+x)^2}\,
\frac{t}{1+t(-\frac{1}{2j+x})}=\sum_{j\in\Z^\times}\bigg\{\frac{1}{2j+x-t}-
\frac{1}{2j+x}\bigg\}
\\
=\frac{\pi}{2}\cot\bigg(\frac{\pi}{2}(x-t)\bigg)
-\frac{\pi}{2}\cot\bigg(\frac{\pi x}{2}\bigg)-\frac{t}{x(x-t)}.
\label{eq-U1k1.1}
\end{multline}

\subsection{The dynamically reduced Hilbert kernel}
Next, let $\kfun_1$ be the associated function
\[
\kfun_1(t,x):=(\id-\Tope_1)\Kfun_1(t,\cdot)(x),
\]
so that by \eqref{eq-U1k1.1},
\begin{equation}
\kfun_1(t,x)=\frac{t}{x(x-t)}+\frac{t}{1+tx}
+\frac{\pi}{2}\cot\bigg(\frac{\pi x}{2}\bigg)
-\frac{\pi}{2}\cot\bigg(\frac{\pi}{2}(x-t)\bigg).
\label{eq-Kt1.1}
\end{equation}
The function $x\mapsto\kfun_1(t,x)$ has removable singularities at $x=0$
and $x=t$, and poles at $x=-2+t$ and $x=2+t$. Therefore, the function
$x\mapsto\kfun_1(t,x)$ has Taylor series at the origin with radius of
convergence equal to $2-|t|$, for $t\in I_1$.
For fixed $t\in I_1$, we know that $\kfun_1(t,\cdot)\in L^1_0(I_1)$, where
\[
L^1_0(I_1):=\big\{f\in L^1(I_1):\,\langle 1,f\rangle_{I_1}=0\big\},
\]
the reason being that
\[
\langle 1,\kfun_1(t,\cdot)\rangle_{I_1}=\langle 1,\Kfun_1(t,\cdot)\rangle_{I_1}-
\langle 1,\Tope_1 \Kfun_1(t,\cdot)\rangle_{I_1}=\langle 1,
\Kfun_1(t,\cdot)\rangle_{I_1}-
\langle \Lop_11,\Kfun_1(t,\cdot)\rangle_{I_1}=0,
\]
since $\Lop_11=1$.
We will  refer to $\kfun_1(t,x)$ as the
\emph{dynamically reduced Hilbert kernel}.

For the endpoint parameter parameter value $t=1$, the expression for the
kernel $\kfun_1$ is
\begin{equation}
\kfun_1(1,x)=-\frac{1}{x}-\frac{2x}{1-x^2}
+\frac{\pi}{2}\cot\bigg(\frac{\pi x}{2}\bigg)
+\frac{\pi}{2}\tan\bigg(\frac{\pi}{2}x\bigg).
\label{eq-Kt1.2}
\end{equation}
The function $x\mapsto \kfun_1(1,x)$ has removable singularities at
$x=0$, $x=1$, and $x=-1$, and the radius of convergence for its Taylor series
at the origin equals $2$. So, in
particular, the function $x\mapsto \kfun_1(1,x)$ extends to a smooth function
on the closed interval $\bar I_1=[-1,1]$.

\subsection{Odd and even parts of the dynamically reduced
Hilbert kernel}

We split the dynamically reduced Hilbert kernel $\kfun_1(t,x)$ according
to odd and even parts with respect to $x$:
\[
\kfun_1(t,x)=\kfun_1^{I}(t,x)-\kfun_1^{II}(t,x),
\]
where
\begin{equation}
\kfun_1^{I}(t,x):=\frac{1}{2}\big(\kfun_1(t,x)+\kfun_1(t,-x)\big),\quad
\kfun_1^{II}(t,x):=\frac{1}{2}\big(\kfun_1(t,-x)-\kfun_1(t,x)\big).
\label{eq-Kt1.3}
\end{equation}
Obviously, $\Kfun_1^I(t,x)$ and $\kfun_1^I(t,x)$ are even functions of $x$,
while $\Kfun_1^{II}(t,x)$ and $\kfun_1^{II}(t,x)$ are odd.
By Proposition \ref{lem-symmetry1},
the operator $\Tope_1$ preserves even and odd symmetry, and it follows that
\begin{equation}
\kfun_1^{I}(t,x)=(\id-\Tope_1)\Kfun_1^{I}(t,\cdot)(x),\quad
\kfun_1^{II}(t,x)=(\id-\Tope_1)\Kfun_1^{II}(t,\cdot)(x).
\label{eq-Kt1.4}
\end{equation}
By inspection, the function $\kfun_1(1,\cdot)$ is odd, so that
$\kfun_1^I(1,\cdot)=0$ and $\kfun_1^{II}(1,\cdot)=-\kfun_1(1,\cdot)$. This
of course corresponds to the observation that $\Kfun^I_1(1,x)=(1-x^2)^{-1}$
is the density of the invariant measure.
Based on \eqref{eq-Kt1.4}, the standard Neumann series cancellation shows the
following:
for fixed $t\in I_1$, we have the decompositions
\begin{equation}
\begin{cases}
\displaystyle\sum_{j=0}^{n-1}\Tope_1^j\kfun_1^{I}(t,\cdot)(x)=\Kfun_1^{I}(t,x)-
\Tope_1^n\Kfun_1^{I}(t,\cdot)(x)
\\
\displaystyle\sum_{j=0}^{n-1}\Tope_1^j\kfun_1^{II}(t,\cdot)(x)=\Kfun_1^{II}(t,x)-
\Tope_1^n\Kfun_1^{II}(t,\cdot)(x).
\end{cases}
\label{eqlem-Kt1.5}
\end{equation}
Note that on the right-hand side of \eqref{eqlem-Kt1.5}, the
first term will tend to dominate as $n\to+\infty$, by Proposition
\ref{prop-ergodicity1.01}(vi). We proceed to analyze the odd part.

\begin{lem}
{\rm(Dynamic decomposition lemma)}
For fixed $t\in I_1$, we have the decomposition
\[
\Kfun_1^{II}(t,x)=\sum_{j=0}^{+\infty}\Tope_1^j\kfun_1^{II}(t,\cdot)(x),
\qquad x\in I_1,
\]
with uniform convergence on compact subsets of $I_1$, as well as norm
convergence in $L^1(I_1)$.
\label{lem-decomp12.1.1}
\end{lem}

\begin{proof}
For fixed $t\in I_1$, we have that $\Kfun_1^{II}(1,\cdot)\in L^1_0(I_1)$, so an
application of Proposition \ref{prop-ergodicity1.01}(v) shows that
$\Tope_1^n\Kfun_1^{II}(t,\cdot)\to0$ in norm in $L^1(I_1)$ as $n\to+\infty$.
Combined with \eqref{eqlem-Kt1.5}, this shows that the Neumann series
converges in the $L^1(I_1)$ norm, as required.
Next, the uniform convergence on compact subset follows from the $L^1(I_1)$
convergence, combined with a comparison with the invariant measure
density and a normal families argument. We leave the necessary details to
the reader.
\end{proof}

\begin{figure}[h]
\includegraphics
[trim = 0cm 15.5cm 0cm 2cm, clip=true,
totalheight=0.32\textheight]{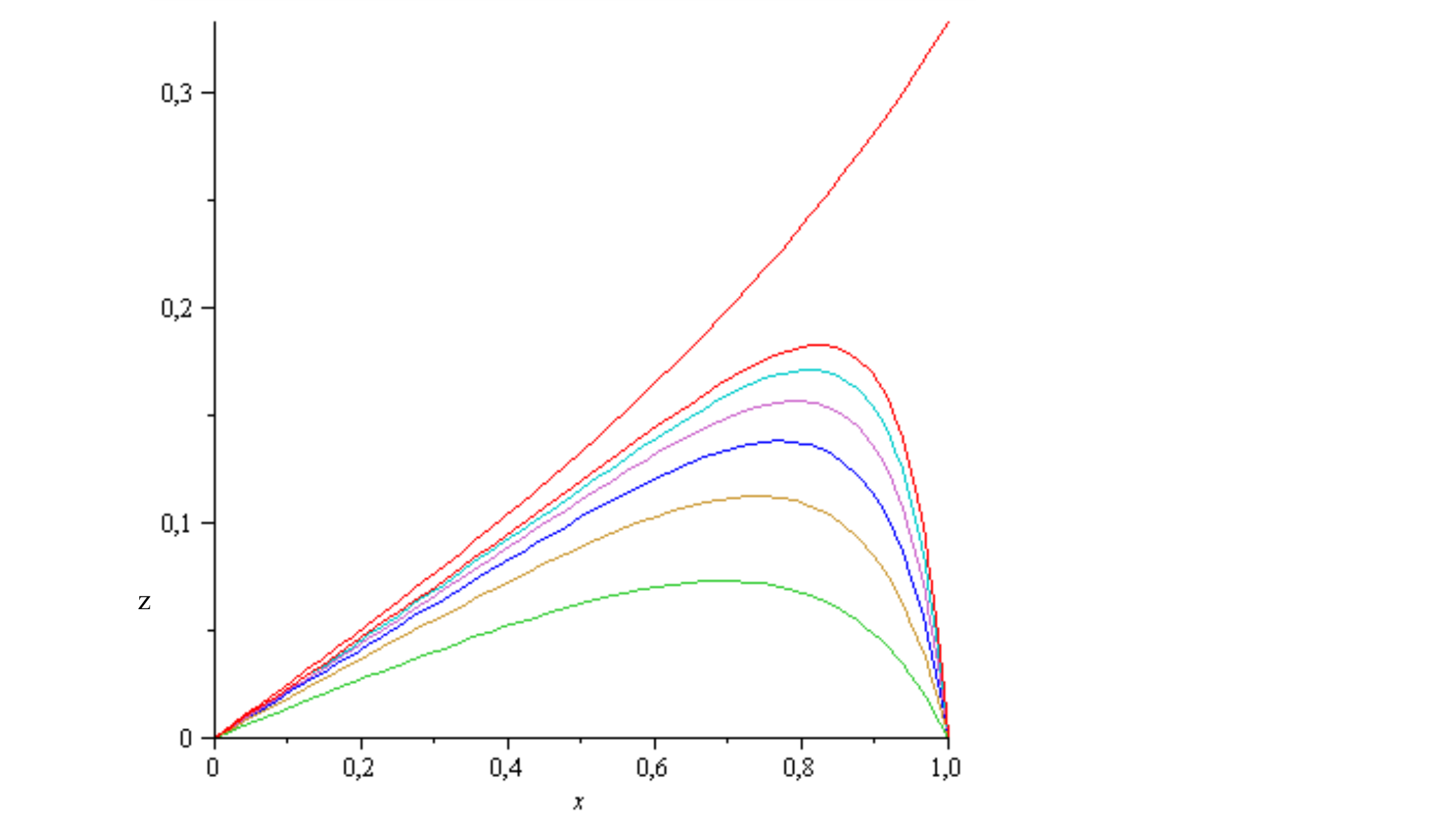}
\caption{Illustration of the dynamic decomposition lemma (Lemma
\ref{lem-decomp12.1.1}). The top curve is
$\Kfun_1^{II}(t,\cdot)$, while the curves below are the partial sums
$\sum_{j=0}^{N}\Tope_1^j\kfun_1^{II}(t,\cdot)$, with $N=0,1,2,3,4,5$. We use the
parameter value $t=0.5$. }
\end{figure}

\subsection{The fundamental estimate of the odd part of the
dynamically reduced Hilbert kernel}
%\label{subsec-fundestDRH1.1}
We will \emph{focus our attention to the odd part}, which  involves
$\Kfun_1^{II}$ and $\kfun_1^{II}$. Note that in view of the previous subsection,
especially the formula \eqref{eq-U1k1.1}, the function $\kfun_1^{II}$ may be
expanded in the series
\begin{equation}
\kfun_1^{II}(t,x)=\frac{t^2x}{1-t^2x^2}+\sum_{j=1}^{+\infty}\bigg\{
\frac{2x}{4j^2-x^2}-\frac{x}{(2j-t)^2-x^2}-\frac{x}{(2j+t)^2-x^2}\bigg\}.
\label{eq-expkfun1.1}
\end{equation}

We need effective control from above and below of the summands in Lemma
\ref{lem-decomp12.1.1}.

Lemma \ref{lem-decomp12.1.1} is basic to our analysis. We need effective
control of the summands from above and below. Our result reads as follows.

\begin{thm}
{\rm (Uniform control of summands)}
For fixed $t\in I_1$ and $j=0,1,2,3,\ldots$, we have the following estimate:
\[
0<\Tope_1^j\kfun_1^{II}(t,\cdot)(x)<\Tope_1^j\kfun_1^{II}(1,\cdot)(x),
\qquad x\in I_1^+.
\]
\label{th-12.1.2}
\end{thm}

We postpone the presentation of the proof until Subsections
\ref{subsec-fundestaboveDRH1.1} (see Proposition \ref{prop-est.above.iterk1.1})
and \ref{subsec-fundestbelowDRH1.1} (see Corollary \ref{cor-classSpres1.1};
later, the concluding steps of the proof are presented).

Of course, by the odd symmetry with respect to $x$, there is a corresponding 
estimate which holds on the left-side interval $I_1^-:=]\!-\!1,0[$ as well.

\subsection{Estimation from above of the odd part of the
dynamically reduced Hilbert kernel}
\label{subsec-fundestaboveDRH1.1}
The estimate from above in Theorem \ref{th-12.1.2} will be obtained as
a consequence of the following property.

\begin{lem}
For fixed $t\in I_1$, the function
$x\mapsto \kfun_1^{II}(1,x)-\kfun_1^{II}(t,x)$ is odd and strictly
increasing on $I_1$.
\label{lem-12.1.3}
\end{lem}

\begin{rem}
In view of Proposition \ref{prop-3.8.2},
$\kfun_1^{II}(t,1)=\kfun_1^{II}(t,-1)=0$ holds for each fixed $t\in I_1$.
The function $x\mapsto \kfun_1^{II}(t,x)$ is odd, and it will be shown later
that it is increasing on some interval $I_\eta$ with $0<\eta<1$, and decreasing
on the remainder set $I_1\setminus I_\eta$ (where the parameter $\eta=\eta(t)$
depends on $t$).
However, things are a little different for $t=1$. In particular, the endpoint
value at $x=\pm1$ is different, as $\kfun_1^{II}(1,1)=\frac12$ and
$\kfun_1^{II}(1,-1)=-\frac12$.
\end{rem}

\begin{proof}[Proof of Lemma \ref{lem-12.1.3}]
It is obvious that the function $x\mapsto \kfun_1^{II}(1,x)-\kfun_1^{II}(t,x)$
is odd. In view of the identity \eqref{eq-expkfun1.1},
\begin{multline*}
\kfun_1(1,x)-\kfun_1(t,x)=(\id-\Tope_1)[\Kfun_1(1,\cdot)-\Kfun_1(t,\cdot)](x)=
\frac{1}{1+x}-\frac{t}{1+tx}
+\sum_{j\in\Z^\times}\bigg\{\frac{1}{2j+x-t}-\frac{1}{2j+x-1}\bigg\},
\end{multline*}
and by forming the odd part with respect to the variable $x$, we obtain that
\begin{multline*}
\kfun_1^{II}(1,x)-\kfun_1^{II}(t,x)=(\id-\Tope_1)
[\Kfun_1^{II}(1,\cdot)-\Kfun_1^{II}(t,\cdot)](x)=
\frac{x}{1-x^2}-\frac{t^2x}{1-t^2x^2}
\\
+\frac12\sum_{j\in\Z^\times}\bigg\{\frac{1}{2j-x-t}-\frac{1}{2j-x-1}-
\frac{1}{2j+x-t}+\frac{1}{2j+x-1}\bigg\}
\\
=\frac{x}{1-x^2}-\frac{t^2x}{1-t^2x^2}+
\sum_{j=1}^{+\infty}\bigg\{\frac{x}{(2j-t)^2-x^2}+\frac{x}{(2j+t)^2-x^2}-
\frac{x}{(2j-1)^2-x^2}-\frac{x}{(2j+1)^2-x^2}\bigg\}.
\end{multline*}
In terms of the function
\[
F(t,x):=\partial_x \Kfun_1^{II}(t,x)=t^2\frac{1+t^2x^2}{(1-t^2x^2)^2},
\]
we calculate that
\begin{multline}
\partial_x\big(\kfun_1^{II}(1,x)-\kfun_1^{II}(t,x)\big)
=F(1,x)-F(t,x)
\\
+\sum_{j=1}^{+\infty}\bigg(
F(r_j(t),x)+F(r_j(-t),x)-F(r_j(1),x)-F(r_j(-1),x)\bigg),
\label{eq-12.2.5.1}
\end{multline}
where use the notation $r_j(t):=1/(2j-t)$ (then $t\mapsto r_j(t)$ is a
positive and increasing function for $j=1,2,3,\ldots$). Since the right-hand
side of \eqref{eq-12.2.5.1}
expresses \emph{an even function of $t$, we may restict our attention to
$t\in I_1^+$}. The derivative with respect to $t$ of the function $F(t,x)$ is
\[
G(t,x):=\partial_t F(t,x)=\partial_t\partial_x \Kfun_1^{II}(t,x)
=2t\frac{1+3t^2x^2}{(1-t^2x^2)^4},
\]
and by representing differences as definite integrals of the derivative, we
obtain that
\begin{multline}
\partial_x\big(\kfun_1^{II}(1,x)-\kfun_1^{II}(t,x)\big)
=\int_t^1 G(\tau,x)\diff \tau
+\sum_{j=1}^{+\infty}\Bigg(\int_{r_j(-1)}^{r_j(-t)}G(\tau,x)\diff\tau
-\int_{r_j(t)}^{r_j(1)}G(\tau,x)\diff\tau\Bigg).
%\\
%=\int_t^{1/(2-t)}G(\tau,x)\diff \tau
%+\sum_{j=1}^{+\infty}\Bigg\{\int_{r_j(-1)}^{r_j(-t)}G(\tau,x)\diff\tau
%-\int_{r_{j+1}(t)}^{r_j(-1)}G(\tau,x)\diff\tau\Bigg\}.
\label{eq-12.2.5.2}
\end{multline}
Here, we used the trivial observation that $r_{j+1}(1)=r_j(-1)$.
Moreover, as the function $t\mapsto G(t,x)$ is monotonically strictly
increasing, we have that
\[
\big(r_j(-t)-r_j(-1)\big)\,G(r_j(-1),x)
<\int_{r_j(-1)}^{r_j(-t)}G(\tau,x)\diff\tau,\quad
\int_{r_{j+1}(t)}^{r_j(-1)}G(\tau,x)\diff\tau<
\big(r_j(-1)-r_{j+1}(t)\big)\,G(r_j(-1),x),
\]
and since a trivial calculation shows that
\[
r_j(-t)-2r_j(-1)+r_{j+1}(t)=\frac{2(t-1)^2}{(2j+t)(2j+2-t)(2j+1)}>0
\]
for $j=1,2,3,\ldots$ and $t\in I_1^+$,
we obtain that
\[
\int_{r_j(-1)}^{r_j(-t)}G(\tau,x)\diff\tau
-\int_{r_{j+1}(t)}^{r_j(-1)}G(\tau,x)\diff\tau>0.
\]
Then, by \eqref{eq-12.2.5.2}, and the observation that
\[
\int_t^{1/(2-t)}G(\tau,x)\diff \tau>0,\qquad t\in I_1^+,\,\,x\in I_1,
\]
it follows that the function $x\mapsto \kfun_1^{II}(1,x)-\kfun_1^{II}(t,x)$ is
strictly increasing, as claimed.
\end{proof}

We may now derive the upper bound in Theorem \ref{th-12.1.2}.

\begin{prop}
For $j=0,1,2,\ldots$ and for fixed $t\in I_1$, the function
$\Tope_1^j[\kfun_1^{II}(1,\cdot)-\kfun_1^{II}(t,\cdot)]$ is odd and increasing. 
In particular, we have that for $j=0,1,2,\ldots$ and $t\in I_1$,
\[
\Tope_1^j\kfun_1^{II}(t,\cdot)(x)<\Tope_1^j\kfun_1^{II}(1,\cdot)(x),\qquad
x\in I_1^+.
\]
\label{prop-est.above.iterk1.1}
\end{prop}

\begin{proof}
This follows from a combination of Lemma \ref{lem-12.1.3} and Proposition
\ref{lem-symmetry2}(i).
\end{proof}

\begin{rem} In general, the positivity of all the powers 
$\kfun_1^{II}(t,\cdot)$ on $I_1^+=]0,1[$ cannot be deduced from the simple 
observation that $\Kfun_1^{II}(t,\cdot)$, $0<t\leq1$, is odd, increasing, and 
positive on $I_1^+$. For instance, the function $f(x)=x^3$ is odd, increasing, 
and positive on $I_1^+$. However, it can be seen that the function 
$(\id-\Tope_1)f=f-\Tope_1 f$ changes signs on $I_1^+$.
\end{rem}

\section{Power series, Hurwitz zeta function, and total positivity}
\label{sec-Hurwitz}

\subsection{A class of power series with at most one positive zero}

The lower bound in Theorem \ref{th-12.1.2} requires a more sophisticated
analysis. To this end, we introduce a class of Taylor series.

Let $\Powerseriesclass$ denote the class of convergent Taylor series
\[
f(x)=\sum_{j=0}^{+\infty}\hat f(j)x^j,\qquad x\in
I_\gamma=]\!-\!\gamma,\gamma[;
\]
in short, we write $f\in\Powerseriesclass$. Moreover, we write
$f\in\PowerseriesclassR$ to express that the Taylor coefficients are real,
that is, $\hat f(j)\in\R$ holds for all $j=0,1,2,\ldots$.

\begin{defn}
Fix $0<\gamma<+\infty$.
If   $f\in\PowerseriesclassR$, we write
$f\in\PowerseriesclassRdown$ to express that either

(a) $\hat f(j)\ge0$ for all $j=0,1,2,\ldots$, or

(b) $\hat f(j)\le0$ for all $j=0,1,2,\ldots$, or

(c) there exists an index $j_0=j_0(f)\in \Z_{+,0}$ such that $\hat f(j)\ge0$
for $j\le j_0$ while $\hat f(j)\le0$ for $j>j_0$.
\label{def-6.1.1}
\end{defn}

Under (c), assuming we have excluded the cases (a) and (b), we let $j_0(f)$ be
the \emph{maximal} index with the property (among all the possibilities).
Then $j_0(f)\in\Z_{+,0}$.

\begin{lem}
Suppose $f\in\PowerseriesclassRdown$, where $\gamma>0$.
Then, unless $f$ vanishes identically, there exists at most one zero of
$f$ on the interval $]0,\gamma[$. If such a point $x_0\in ]0,\gamma[$ with
$f(x_0)=0$ exists, then $f(x)>0$ holds for $0<x<x_0$, while $f(x)<0$ for
$x_0<x<\gamma$.
\label{lem-zero.exist1.1}
\end{lem}

\begin{proof}
In accordance with the assumptions, we consider only $f(x)\not\equiv0$.
Note that in cases (a) or (b) of Definition \ref{def-6.1.1}, i.e., when
$\forall j:\hat f(j)\ge 0$ or $\forall j:\hat f(j)\le 0$, then,  on the
interval $[0,\gamma[$,  $f$ is either strictly increasing with
$f(0)=\hat f(0)\ge0$, or strictly decreasing with $f(0)\le0$, and in each
case $f$ has no zero in the interval $]0,\gamma[$.

It remains to deal with the case (c) of Definition \ref{def-6.1.1}, assuming
that neither (a) nor (b) is fulfilled.

We proceed by induction on the index
$j_0(f)$. First assume that $j_0(f)=0$. In this case, $f(x)$ is decreasing 
on $[0,\gamma[$,
and it is strictly decreasing unless it is constant. If $f(x)$
is constant, then the constant cannot be $0$, and then $f(x)$ would have no
zeros at all. If $f(x)$ is strictly decreasing instead, then it obviously
can have at most one zero in the given interval, and if such a zero exists,
then the values of $f(x)$ are positive to the left of $x_0$ and negative to
the right.

Assume now that $j_0(f)=r\geq 1$ and that the assertion
 of the lemma has been
established for all $f\in\mathfrak S^{\downarrow}_\R(\gamma)$ with $j_0(f)=r-1$.
Then the derivative $f'(x)$ is also in the class $\PowerseriesclassRdown$, 
and $j_0(f^\prime)=r-1$. By the induction hypothesis, there are two 
possibilities:

\noindent Case (i) There is no point $x_0$ on $]0,\gamma[$ such that 
$f^\prime(x_0)=0$, and

\noindent Case (ii) There is $x_1$, such that $f^\prime (x_1)=0$.

In case (i), $f^\prime$ must have constant sign on the interval $]0,\gamma[$. 
If the sign is positive, then $f(x)$ is increasing there, and since $f(0)\ge0$
the function $f(x)$ cannot have any zeros in $]0,\gamma[$.
If instead the sign of $f^\prime$ is negative, then $f(x)$ is decreasing. If
  $f(0)= 0$ the function $f(x)$ has no zeros at all on $]0,\gamma[$. If 
$f(0)>0$,  then then either $f(x)$ is positive on $]0,\gamma[$,
or it has precisely one zero $x_0\in]0,\gamma[$, in which case $f(x)$ is
positive to the left and negative to the right of $x_0$.

In case (ii), then  by the induction hypothesis, we also know that $f'(x)>0$ 
for $0<x<x_1$ and that $f'(x)<0$ for $x_1<x<\gamma$. This means that 
$f(x)$ is strictly increasing on $]0,x_1]$ and as $f(0)\ge0$,
it follows that $f(x)>0$ for $0<x\le x_1$. In the remaining interval
$x_1<x<\gamma$, $f(x)$ is decreasing, and it is either positive throughout
or has exactly one zero $x_0\in]x_1,\gamma[$, in which case $f(x)$ is positive
to the left on $x_0$ and negative to the right. The proof is complete.
\end{proof}

\subsection{The Taylor coefficients of the odd part of the
dynamically reduced Hilbert kernel}

We now analyze the symmetrized dynamically reduced Hilbert kernel
$\kfun_1^{II}(t,x)$.

\begin{prop}
For a fixed $0<t<1$, the function $x\mapsto \kfun_1^{II}(t,x)$ is odd and
belongs to the class $\mathfrak S^{\downarrow}_\R(\gamma)$ with $\gamma=2-t$. 
Indeed, $\kfun_1^{II}(t,\cdot)$ meets condition (c) of Definition 
\ref{def-6.1.1}.
\label{prop-classS1.1}
\end{prop}

Before we supply the full proof of the proposition, we need to do some
preparatory work.
The kernel $\kfun_1^{II}(t,x)$ is given by \eqref{eq-expkfun1.1}.
It is an odd function of $x$, and enjoys the Taylor expansion
\begin{equation}
\kfun_1^{II}(t,x)=\sum_{j=0}^{+\infty}\varkappa_j(t)x^{2j+1},\qquad x\in I_{1},
\label{eq-oddkernel.Taylor1.1}
\end{equation}
with radius of convergence $2-t$, where the coefficients can be readily
calculated:
\[
\varkappa_j(t):=\frac{2^{-2j-2}}{(2j+1)!}
\big\{2\psi^{(2j+1)}(1)-\psi^{(2j+1)}(1-\tfrac{t}{2})-\psi^{(2j+1)}(1+\tfrac{t}{2})
\big\}+t^{2j+2}.
\]
Here,
\begin{equation}
\psi(x):=\frac{\Gamma'(x)}{\Gamma(x)},\quad \psi^{(m)}(x)=
\frac{\diff^m}{\diff x^m}\frac{\Gamma'(x)}{\Gamma(x)}
\label{eq-polygamma1}
\end{equation}
is the \emph{poly-Gamma function}, and $\Gamma(x)$ is the standard \emph{Gamma
function}. A more convenient expression is obtained by direct Taylor expansion
of the terms in \eqref{eq-expkfun1.1}:
\begin{equation*}
\varkappa_j(t)=t^{2j+2}+\sum_{k=1}^{+\infty}\bigg\{
\frac{2}{(2k)^{2j+2}}-\frac{1}{(2k-t)^{2j+2}}-\frac{1}{(2k+t)^{2j+2}}\bigg\}.
%\label{}
\end{equation*}
In terms of the \emph{Hurwitz zeta function}
\[
\zeta(s,x):=\sum_{k=0}^{+\infty}(x+k)^{-s},
\]
the expression for $\varkappa_j(x)$ equals
\begin{equation*}
\varkappa_j(t)=t^{2j+2}+2^{-2j-2}\big\{2\zeta(2j+2,1)
-\zeta(2j+2,1-\tfrac{t}{2})-\zeta(2j+2,1+\tfrac{t}{2})\big\}.
%\label{}
\end{equation*}
Moreover, since $\zeta(s,x)=x^{-s}+\zeta(s,1+x)$, we may rewrite
this as
\begin{equation}
\varkappa_j(t)=t^{2j+2}-(2-t)^{-2j-2}
+2^{-2j-2}
\big\{2\zeta(2j+2,1)
-\zeta(2j+2,2-\tfrac{t}{2})-\zeta(2j+2,1+\tfrac{t}{2})\big\}.
\label{eq-vkappa1.1}
\end{equation}
We need the following result.

\begin{lem}
For fixed $\tau$ with $0<\tau\le\frac12$, the function
\[
\Lambda_\tau(s):=
(1-\tau)^s\big\{2\zeta(s,1)-\zeta(s,2-\tau)-\zeta(s,1+\tau)\big\}
\]
is positive and strictly decreasing on the interval $[3,+\infty[$, with limit
$\lim_{s\to+\infty}\Lambda_\tau(s)=0$.
\label{lem-Hurwitz}
\end{lem}

\begin{proof}
We will keep the variables $\tau$ and $s$ confined to the indicated
intervals $0<\tau\le\frac12$ and $3\le s<+\infty$.

By comparing term by term in the Hurwitz zeta sum, we see that the function
$\Lambda_\tau(s)$ is positive. Moreover, as $s\to+\infty$, the first term
$x^{-s}$ becomes dominant in the Hurwitz zeta series $\zeta(s,x)$,  and we
obtain that $\Lambda_\tau(s)$ has the indicated limit.

Next, we split the function in the following
way:
\begin{equation}
\Lambda_\tau(s)=\lambda_\tau(s)+R_\tau(s),\qquad
\lambda_\tau(s):=2(1-\tau)^s-\bigg(\frac{1-\tau}{1+\tau}\bigg)^s
\label{eq-Lambdaexp1.1}
\end{equation}
where $R_\tau(s)$ is given by
\[
R_\tau(s):=(1-\tau)^s\big\{2\zeta(s,2)-\zeta(s,2-\tau)-\zeta(s,2+\tau)\big\}.
\]
In the identity \eqref{eq-Lambdaexp1.1}, we should think of $\lambda_\tau(s)$
as the main term and of $R_\tau(s)$ as the remainder.
Clearly, we see that $\lambda_\tau(s)>0$, and that $\lambda_\tau(s)$ is
decreasing in $s$:
\begin{multline}
\lambda_\tau'(s)=2(1-\tau)^s\log(1-\tau)-\bigg(\frac{1-\tau}{1+\tau}\bigg)^s
\log\frac{1-\tau}{1+\tau}
\\
=(1-\tau)^s\bigg\{\log(1-\tau^2)-[1-(1+\tau)^{-s}]
\log\frac{1+\tau}{1-\tau}\bigg\}<0.
\label{eq-estimate.tau.s1.1}
\end{multline}
Moreover, by direct calculation
\[
\partial_s\frac{\lambda_\tau'(s)}{(1-\tau)^s}=-(1+\tau)^{-s}
\{\log(1+\tau)\}\log\frac{1+\tau}{1-\tau}<0,
\]
so that by \eqref{eq-estimate.tau.s1.1}, we have that
\begin{multline}
\lambda_\tau'(s)\le (1-\tau)^s \frac{\lambda_\tau'(3)}{(1-\tau)^3}
=(1-\tau)^s\bigg\{\log(1-\tau^2)-[1-(1+\tau)^{-3}]
\log\frac{1+\tau}{1-\tau}\bigg\}
\\
=(1-\tau)^s\bigg\{\log(1-\tau^2)-\frac{\tau(3+3\tau+\tau^2)}{(1+\tau)^3}
\log\frac{1+\tau}{1-\tau}\bigg\}
\\
\le-\tau^2(1-\tau)^s\bigg\{1+\frac{6+6\tau+2\tau^2}{(1+\tau)^4}\bigg\}\le
-\frac{233}{81}\tau^2(1-\tau)^s,
\label{eq-estimate.tau.s1.2}
\end{multline}
by an elementary estimate of the logarithm function, and by using our
constraints on $s$ and $\tau$.

We proceed to estimate the remainder term.
Since for positive $\tau$ and a $C^2$-smooth function $f$, we have that
\[
f(\tau)+f(-\tau)-2f(0)=\int_{-\tau}^{\tau}(\tau-|\theta|)f''(\theta)\diff\theta,
\]
and, in particular,
\begin{multline*}
\zeta(s,2-\tau)+\zeta(s,2+\tau)
-2\zeta(s,2)=\int_{-\tau}^{\tau}(\tau-|\theta|)\partial_\theta^2\zeta(s,2+\theta)
\diff\theta=s(s+1)\int_{-\tau}^{\tau}(\tau-|\theta|)\zeta(s+2,2+\theta)
\diff\theta,
\end{multline*}
so that, as a consequence,
\begin{equation*}
R_\tau(s)=-s(s+1)(1-\tau)^s\int_{-\tau}^{\tau}(\tau-|\theta|)\zeta(s+2,2+\theta)
\diff\theta.
\end{equation*}
By direct inspection, then, we see that $R_\tau(s)<0$.
Moreover, by differentiating the above formula with respect to $s$, we obtain
\begin{multline}
R_\tau'(s)=-(2s+1)(1-\tau)^s\int_{-\tau}^{\tau}(\tau-|\theta|)\zeta(s+2,2+\theta)
\diff\theta
\\
-s(s+1)\int_{-\tau}^{\tau}(\tau-|\theta|)
\partial_s[(1-\tau)^s\zeta(s+2,2+\theta)]\diff\theta
\\
=\frac{2s+1}{s(s+1)}R_\tau(s)-s(s+1)\int_{-\tau}^{\tau}(\tau-|\theta|)
\partial_s[(1-\tau)^s\zeta(s+2,2+\theta)]\diff\theta
\\
\le-s(s+1)\int_{-\tau}^{\tau}(\tau-|\theta|)
\partial_s[(1-\tau)^s\zeta(s+2,2+\theta)]\diff\theta.
\label{eq-estimate.tau.s1.3}
\end{multline}
We proceed to calculate the derivative which appears in
\eqref{eq-estimate.tau.s1.3}:
\begin{equation}
-\partial_s[(1-\tau)^s\zeta(s+2,2+\theta)]=(1-\tau)^{-2}\sum_{k=0}^{+\infty}
\bigg(\frac{2+\theta+k}{1-\tau}\bigg)^{-s-2}\log\frac{2+\theta+k}{1-\tau}.
\label{eq-estimate.tau.s1.4}
\end{equation}
To analyze this derivative properly, we need the following calculation:
\begin{equation}
\partial_T\big\{T^{-s-2}\log T\big\}=-T^{-s-3}((s+2)\log T-1)\le
-T^{-s-3}(5\log T-1)<0,\qquad
\tfrac32\le T<+\infty.
\label{eq-estimate.tau.s1.5}
\end{equation}
In other words, $T\mapsto T^{-s-2}\log T$ is decreasing on the interval
$[\frac32,+\infty[$. As a first application of the property
\eqref{eq-estimate.tau.s1.5}, we apply it to the identity
\eqref{eq-estimate.tau.s1.4} and obtain that
\begin{equation*}
0\le-\partial_s[(1-\tau)^s\zeta(s+2,2+\theta)]\le
-\partial_s[(1-\tau)^s\zeta(s+2,2-\tau)],\qquad -\tau\le\theta\le\tau,
%\label{eq-estimate.tau.s1.6}
\end{equation*}
which we may implement into \eqref{eq-estimate.tau.s1.3}:
\begin{multline}
R_\tau'(s)
\le-s(s+1)\int_{-\tau}^{\tau}(\tau-|\theta|)
\partial_s[(1-\tau)^s\zeta(s+2,2+\theta)]\diff\theta
\\
\le -s(s+1)\int_{-\tau}^{\tau}(\tau-|\theta|)
\partial_s[(1-\tau)^s\zeta(s+2,2-\tau)]\diff\theta=-s(s+1)\tau^2
\partial_s[(1-\tau)^s\zeta(s+2,2-\tau)].
\label{eq-estimate.tau.s1.7}
\end{multline}
Next, we implement the property \eqref{eq-estimate.tau.s1.5} again, in the
context of \eqref{eq-estimate.tau.s1.4} with $\theta=-\tau$:
\begin{multline}
-\partial_s[(1-\tau)^s\zeta(s+2,2-\tau)]=(1-\tau)^{-2}\sum_{k=0}^{+\infty}
\bigg(\frac{2-\tau+k}{1-\tau}\bigg)^{-s-2}\log\frac{2-\tau+k}{1-\tau}
\\
\le (1-\tau)^{s}(2-\tau)^{-s-2}\log\frac{2-\tau}{1-\tau}
+(1-\tau)^{-2}\int_{0}^{+\infty}
\bigg(\frac{2-\tau+x}{1-\tau}\bigg)^{-s-2}\log\frac{2-\tau+x}{1-\tau}
\diff x.
\label{eq-estimate.tau.s1.8}
\end{multline}
Here, we kept the first term (with $k=0$), and replaced each
later term indexed by $k$ by a corresponding integral over the adjacent
interval $[k-1,k]$. The integral expression in \eqref{eq-estimate.tau.s1.8}
may be calculated explicitly:
\begin{multline}
\int_{0}^{+\infty}
\bigg(\frac{2-\tau+x}{1-\tau}\bigg)^{-s-2}\log\frac{2-\tau+x}{1-\tau}
\diff x
\\
=(s+1)^{-2}(1-\tau)^{s+2}(2-\tau)^{-s-1}\bigg[1+
(s+1)\log\frac{2-\tau}{1-\tau}\bigg].
\label{eq-estimate.tau.s1.9}
\end{multline}
Finally, we put \eqref{eq-estimate.tau.s1.7},
\eqref{eq-estimate.tau.s1.8}, and \eqref{eq-estimate.tau.s1.9}
together, and obtain that
\begin{multline}
R_\tau'(s)
\le-s(s+1)\tau^2\partial_s[(1-\tau)^s\zeta(s+2,2-\tau)]
\\
\le \tau^2(1-\tau)^s\Bigg\{s(s+1)(2-\tau)^{-s-2}\log\frac{2-\tau}{1-\tau}
+\frac{s}{s+1}(2-\tau)^{-s-1}\bigg[1+
(s+1)\log\frac{2-\tau}{1-\tau}\bigg]
\Bigg\}.
\label{eq-estimate.tau.s1.10}
\end{multline}
The expression within brackets is optimized at the right end-point
$\tau=\frac12$:
\begin{multline}
s(s+1)(2-\tau)^{-s-2}\log\frac{2-\tau}{1-\tau}
+\frac{s}{s+1}(2-\tau)^{-s-1}\bigg[1+
(s+1)\log\frac{2-\tau}{1-\tau}\bigg]
\\
\le s(s+1)\bigg(\frac32\bigg)^{-s-2}\log3
+\frac{s}{s+1}\bigg(\frac32\bigg)^{-s-1}\bigg[1+
(s+1)\log3\bigg]\le\frac{27}{10},
\label{eq-estimate.tau.s1.11}
\end{multline}
where the rightmost inequality is an exercise in (one-variable) Calculus.
It follows from \eqref{eq-estimate.tau.s1.10} and
\eqref{eq-estimate.tau.s1.11} that
\begin{equation}
R_\tau'(s) \le\frac{27}{10}\tau^2(1-\tau)^s.
\label{eq-estimate.tau.s1.12}
\end{equation}
Finally, a combination of \eqref{eq-estimate.tau.s1.2} and
\eqref{eq-estimate.tau.s1.12} gives the desired result:
\[
\Lambda_\tau'(s)=\lambda_\tau'(s)+R_\tau'(s)\le \bigg(\frac{27}{10}
-\frac{233}{81}\bigg)\tau^2(1-\tau)^s=-\frac{143}{810}\tau^2(1-\tau)^s<0.
\]
The proof is complete.
\end{proof}

\begin{prop}
For fixed $t$, $0<t<1$, the function $j\mapsto(2-t)^{2j+2}\varkappa_j(t)$ is
strictly decreasing on $\Z_+=\{1,2,\ldots\}$, with limit
\[
\lim_{j\to+\infty}(2-t)^{2j+2}\varkappa_j(t)=-1.
\]
\label{lem-decreas.coeff1.1}
\end{prop}

\begin{proof}
In view of \eqref{eq-vkappa1.1}, we know that
\begin{equation*}
(2-t)^{2j+2}\varkappa_j(t)=[t(2-t)]^{2j+2}-1
+\Lambda_{t/2}(2j+2).
%\label{eq-vkappa1.2}
\end{equation*}
Since $0<t(2-t)<1$ holds for $t\in I_1^+$, the function
$j\mapsto[t(2-t)]^{2j+2}$ is decreasing, and the lemma becomes an immediate
consequence of Lemma \ref{lem-Hurwitz}.
\end{proof}

\begin{proof}[Proof of Proposition \ref{prop-classS1.1}]
It is clear from the known radius of convergence for $\kfun_1^{II}(t,\cdot)$
that $\kfun_1^{II}(t,\cdot)\in \PowerseriesclassII$. Moreover,
$\kfun_1^{II}(t,\cdot)$ is odd, and all the Taylor coefficients
(see \eqref{eq-oddkernel.Taylor1.1}) are clearly real-valued, while the
coefficient of the linear term is explicit and positive:
\[
\varkappa_0(t):=\frac{\pi^2}{12}+\frac{1}{t^2}
-\frac{\pi^2/4}{\sin^2(\frac{\pi}{2}t)}+t^2>0.
\]
Now, the proof of the proposition is an immediate consequence of Proposition
\ref{lem-decreas.coeff1.1}.
\end{proof}

\subsection{Positivity of the odd part of the dynamically reduced
Hilbert kernel and totally positive matrices}
\label{subsec-fundestbelowDRH1.1}
The transfer operator $\Tope_1$ can be applied to polynomials, or, more
generally, convergent power series.
For $j=0,1,2,\ldots$, let $u_j$ denote the monomial $u_j(x):=x^{2j+1}$.
The action of $\Tope_1$ on odd   power series  can be
analyzed in terms of the infinite matrix
$\mathbf{B}=\{b_{j,k}\}_{j,k=0}^{+\infty}$ with entries $b_{j,k}$ given by
\begin{equation}
\Tope_1 u_j(x)=\sum_{k=0}^{+\infty}b_{j,k}u_k,\qquad x\in I_1,
\label{eq-bcoeff1.1}
\end{equation}
since the transfer operator $\Tope_1$ preserves oddness.

We recall the notion of a \emph{totally positive matrix} \cite{Pinkus}.
An infinite matrix $\mathbf{A}=\{a_{j,k}\}_{j,k=0}^{+\infty}$ is said to
be \emph{totally positive} if all its minors are $\ge0$, and
\emph{strictly totally positive} if all its minors are $>0$.
Here, a minor
is the determinant of a square submatrix $\{a_{j_s,k_t}\}_{s,t=1}^{r}$ where
$j_1<\cdots<j_r$ and $k_1<\cdots<k_r$. This is a much stronger property
than the usual positive definiteness of a matrix, which would correspond to
considering only symmetric squares.

\begin{prop}
The matrix $\mathbf{B}=\{b_{j,k}\}_{j,k=0}^{+\infty}$ with coefficients given by
\eqref{eq-bcoeff1.1} is strictly totally positive.
\label{prop-tot.pos.1.1}
\end{prop}

\begin{proof}
We read off from the definition of $\Tope_1$ that
\[
\Tope_1 u_j(x)=-\sum_{n\in\Z^\times}(x+2n)^{-2j-3},
\]
and observe that the right-hand side may be written in the form
\[
\Tope_1u_j(x)=\frac{1}{2^{2j+2}(2j+2)!}\big\{\psi^{(2j+2)}(1+\tfrac{x}{2})
-\psi^{(2j+2)}(1-\tfrac{x}{2})\big\},
\]
where $\psi^{(m)}(x)$ is the poly-Gamma function (see \eqref{eq-polygamma1}).
From this, we immediately obtain
\[
b_{j,k}=\frac{\psi^{(2j+2k+3)}(1)}{2^{2j+2k+3}(2j+2)!(2k+1)!}.
\]
Since strict total positivity remains unchanged as we multiply a column or
a row by a positive number, the strict total positivity of the matrix
$\mathbf{B}$ is equivalent to the strict total positivity of the
infinite matrix with entries $\{c_{j+k}\}_{j,l=0}^{+\infty}$ where
$c_j:=\psi^{(2j+3)}(1)$. This is a Hankel matrix, and in view of Theorem 4.4
\cite{Pinkus}, its total positivity is equivalent to the strict positive
definiteness of all the finite square matrices $\{c_{j+l}\}_{j,l=0}^{N}$ and
$\{c_{j+l+1}\}_{j,l=0}^{N-1}$, for every $N=1,2,3,\ldots$. Following the digression
in Section 4.6 of
\cite{Pinkus}, we know that this is equivalent to having the $c_j$ be the
moments of a positive measure (the Stieltjes moment problem). However, it
is known that
\[
c_j=\psi^{(2j+3)}(1)=\int_0^{+\infty}t^{2j+3}\frac{\e^{-t}}{1-\e^{-t}}\diff t
=\int_0^{+\infty}t^{j}\frac{t\,\e^{-\sqrt{t}}}{1-\e^{-\sqrt{t}}}\diff t,
\]
which means that the $c_j$ are indeed the moments of a positive measure.
This completes the proof.
\end{proof}

We need to have a precise definition of the notion of counting
\emph{sign changes}, see \cite{Pinkus}.

\begin{defn}
%(Sign changes)
Let $\mathbf{a}=\{a_j\}_{j}$, $j=0,\ldots,N$, be a finite sequence
of real numbers.

\noindent (a) The number $S^-(\mathbf{a})$ counts the number of sign changes
in the sequence with zero terms discarded. This is the number of
\emph{strong sign changes}.

\noindent (b) The number $S^+(\mathbf{a})$ counts the maximal number of sign
changes in the sequence, where zero terms are arbitrarily replaced by $+1$ or
$-1$. This is the number of \emph{weak sign changes}.
\end{defn}

Obviously, the number of weak sign changes exceeds the number of strong
sign changes, i.e., $S^-(\mathbf{a})\le S^+(\mathbf{a})$.

%\begin{prop}
%For fixed $t$, $0<t<1$, we have  $ \Tope^n_1 \kfun_1^{II}(t,x)$
% is odd and  $ \Tope^n_1 \kfun_1^{II}(t,0)=  T^n_1 \kfun_1^{II}(t,1)=0$ 
%for every $n\in \Bbb Z_{+,0}$
%Moreover,   $\Tope^n_1 \kfun_1^{II}(t,x)$ is never the zero function.
%\label{kernel}
%\end{prop}
%\begin{proof} The first part of the statement follows by induction 
%from the fact that $\Tope_1$ preserves
% oddness and Proposition  \ref{prop-3.8.2}.
%As for the second statment, if  $\Tope ^n_1 \kfun_1^{II}(t,x)$ is the 
%zero function for some $n\in \Bbb Z_{+,0}$,
%then it follows that $\Tope^n\Kfun_1^{II}(t,x)$ is in the kernel of
%$\id -\Tope_1$, which is only possible for the zero function.
%The latter possibility cannot be true beacuse  
%$\Tope^n\Kfun_1^{II}(t,x)$ is always a strictly increasing by Proposition
%\ref{prop-increaspres1}
%\end{proof}

\begin{cor}
Fix $1\le\gamma<+\infty$.
If $f\in\PowerseriesclassRdown$ is odd, then
$\Tope_1 f$ is odd as well, and
$\Tope_1f\in\PowerseriesclassIIIRdown$.
\label{cor-classSpres1.1}
\end{cor}

\begin{proof}
Based on the explicit expression 
\eqref{eq-Tope.1} for $\Tope_1f$,
it is a straightforward exercise in the analysis of power series to check that
if $f\in\mathfrak{S}(\gamma)$, then
$\Tope_1 f\in\PowerseriesclassIII$. Moreover, it is clear that
the property of having real Taylor coefficients is preserved under $\Tope_1$.
To finish the proof, we pick an odd $f\in\PowerseriesclassRdown$,
and expand it in a Taylor series:
\[
f(x)=\sum_{j=0}^{+\infty}\hat f(2j+1)u_j(x),\qquad -\gamma<x<\gamma.
\]
where as before $u_j(x)=x^{2j+1}$.
Then, in view of \eqref{eq-bcoeff1.1},
\begin{equation}
\Tope_1 f(x)=\sum_{j=0}^{+\infty}\hat f(2j+1)\Tope_1 u_j (x)=
\sum_{k=0}^{+\infty}\Bigg\{\sum_{j}^{+\infty}b_{j,k}\hat f(2j+1)\Bigg\}
u_j(x),\qquad -1<x<1,
\label{eq-Topeform1}
\end{equation}
 and, as   noted before, the right-hand
side Taylor series converges in the  interval
$]-2+\frac1\gamma,2-\frac1{\gamma}[$. The assertion of the corollary is trivial
if $f(x)\equiv0$, so we may assume that $f$ does not vanish identically.
From the definition of the class $\PowerseriesclassRdown$, we read off that
for $N=1,2,3,\ldots$, the finite sequence $\{\hat f(2j+1)\}_{j=0}^{N}$ has
\emph{at most one strong sign change}.
Next, by Proposition \ref{prop-tot.pos.1.1},  we may apply the Variation 
Diminishing Theorem for strictly totally positive matrices (see Theorem 3.3 
in \cite{Pinkus}), which asserts that the sequence $\{F_{k,N}\}_{k=0}^N$,
where
\begin{equation}
F_{k,N}:=\sum_{j=0}^{N}b_{j,k}\hat f(2j+1),
\label{eq-Tseq1}
\end{equation}
has \emph{at most one weak sign change} in the index interval
$\{0,\ldots,N\}$.
Moreover, if there is a weak sign change in the sequence $\{F_{k,N}\}_{k=0}^N$,
then it is from $\ge0$ on the left to $\le0$ on the right. More precisely,
we have the following three possibilities:
\smallskip

\noindent(i) $F_{k;N}\ge0$ for all $k=0,\ldots,N$, or

\noindent(ii) $F_{k;N}\le0$ for all $k=0,\ldots,N$, or

\noindent(iii) there exists an index $k_0\in\{0,\ldots,N-1\}$ such that
$F_{k,N}\ge0$ for $k=0,\ldots,k_0$ while $F_{k,N}\le0$ for $k=k_0+1,\ldots,N$.
\smallskip

As we let $N\to+\infty$, the coefficients $F_{k,N}$ converge to
\[
F_{k}:=\sum_{j=0}^{+\infty}b_{j,k}\hat f(2j+1),
\]
where  the right-hand side is absolutely
convergent because all the coefficients (except possibly a finite number 
of them) has the same sign.
From the properties (i)--(iii), we see that the sequence
$\{F_k\}_k$ has one of the following three properties:
\smallskip

\noindent(i') $F_{k}\ge0$ for all $k=0,1,2,\ldots$, or

\noindent(ii') $F_{k}\le0$ for all $k=0,1,2,\ldots$, or

\noindent(iii') there exists an index $k_0\in\Z_{+,0}$ such that
$F_{k}\ge0$ for $k=0,\ldots,k_0$ while $F_{k,N}\le0$ for $k=k_0+1,k_0+2,\ldots$.
\smallskip

We remark that while it is clear that property (i) converges to (i'), and
that (ii) converges to (ii'), the case (iii) is less stable and might
degenerate into (i') or (ii'), as $N\to+\infty$. No matter which of these
cases (i')--(iii') we are in, the corresponding Taylor series
\begin{equation*}
\Tope_1 f(x)=\sum_{k=0}^{+\infty}F_k
x^{2k+1},\qquad -1<x<1,
%\label{eq-Topeform1}
\end{equation*}
is odd and belongs to $\PowerseriesclassIIIRdown$.
The proof is complete.
\end{proof}

We now turn to the proof of Theorem \ref{th-12.1.2}.

\begin{proof}[Proof of Theorem \ref{th-12.1.2}]
As the required estimate from above was obtained back in Proposition
\ref{prop-est.above.iterk1.1}, we may concentrate on the estimate from below.

The function $t\mapsto\kfun_1^{II}(t,x)$ is clearly even, and then the iterates
$\Tope_1\kfun_1^{II}(t,\cdot)$ are also even with respect to the parameter $t$.
So, by symmetry, it will be enough to treat the case $0<t<1$.
So, \emph{we assume} $0<t<1$, and observe that Proposition
\ref{prop-classS1.1} asserts that $\kfun_1^{II}(t,\cdot)$ is odd and belongs to
the class $\PowerseriesclassIIRdown$. Next, by applying Corollary
\ref{cor-classSpres1.1} once, we have that $\Tope_1\kfun_1^{II}(t,\cdot)$
is odd as well and belongs to $\PowerseriesclassIVRdown$. Here, we note
that $1<2-t<2$ and $1<2-\frac{1}{2-t}<\frac32$. By applying
Corollary \ref{cor-classSpres1.1} iteratively, we find more generally that
$\Tope_1^j\kfun_1^{II}(t,\cdot)\in\PowerseriesclassVRdownj$, for
some $\gamma_j(t)$ with $\gamma_j(t)>1$.
Now, since $\kfun_1^{II}(t,\cdot)$ is odd, we may apply repeatedly 
Proposition \ref{prop-3.8.2}, to see that
\begin{equation}
\Tope_1^j\kfun_1^{II}(t,\cdot)(1)=\Tope_1^{j-1}\kfun_1^{II}(t,\cdot)(1)=\cdots =
\Tope_1\kfun_1^{II}(t,\cdot)(1)=\kfun_1^{II}(t,\cdot)(1)=0.
\label{eq-zeroat1}
\end{equation}
Next, by \eqref{eq-zeroat1} and Lemma \ref{lem-zero.exist1.1},
we find that
\[
\Tope_1^j \kfun_1^{II}(t,\cdot)(x)>0,\qquad 0<x<1,\,\,\, j=1,2,3,\ldots.
\]
unless the function $\Tope_1^j\kfun_1^{II}(t,\cdot)\in\PowerseriesclassVRdownj$
vanishes identically. To rule out the latter possibility, we argue as follows.
If $\Tope_1^j\kfun_1^{II}(t,\cdot)=0$, then we would have that
\[
0=\Tope_1^j\kfun_1^{II}(t,\cdot)=\Tope_1^j(\id-\Tope_1)\Kfun_1^{II}(t,\cdot)
=(\id-\Tope_1)\Tope_1^j\Kfun_1^{II}(t,\cdot),
\]
that is, $\Tope_1^j\Kfun_1^{II}(t,\cdot)\in L^1(I_1)$ would be an
eigenfunction for the operator $\Tope_1$ corresponding to the eigenvalue 
$1$, which is only possible (in view of Proposition \ref{prop-5.6.5'}) when
$\Tope_1^j\Kfun_1^{II}(t,\cdot)=0$. This  is absurd, as $\Tope_1$ preserves
the class of odd strictly increasing functions, see Proposition
\ref{lem-symmetry2}(i). The proof is complete.
\end{proof}

\section{Asymptotic decay of the $\Tope_1$-orbit of an odd
distribution in $\LspaceIone$}
\label{sec-asymptotic.decay=1}

\subsection{An application of asymptotic decay for $\beta=1$}
\label{subsec-applasdecay=1}
%A distribution $u_0$ on the interval $I_1$ is said to be \emph{odd} if it is
%annihilated by the even test functions. It is then easy to see that
%$u_0\in\LspaceIone$ is odd if and only if the function $\pev[u_0]$ is odd.
%The following result is more sophisticated than Theorem \ref{thm-decay<1},
%and the proof requires the full strength of the machinery developed in
%Section \ref{sec-Hilbkernel}.

%\begin{thm}
%Fix $\beta=1$.
%Suppose $u_0\in\LspaceIone$ is an odd distribution.
%Then as $N\to+\infty$, we have that $1_{I_\eta}\pev[\Tope_1^N u_0]\to0$
%in $L^{1,\infty}(I_1)$, for each fixed $\eta$, $0<\eta<1$.
%\label{thm-decay=1}
%\end{thm}

We now 
%delay the presentation of the proof slightly, and 
explain how to obtain, in the critical parameter regime $\alpha\beta=1$,
Theorem \ref{thm-2.0} as a consequence of Theorem
\ref{thm-basic1.002}.

\begin{proof}[Proof of Theorem \ref{thm-2.0} for $\alpha\beta=1$]
As observed right after the formulation of Theorem \ref{thm-2.0},
a scaling argument allows us to reduce the redundancy and fix
$\alpha=1$, in which case the condition $0<\alpha\beta=1$ reads
$\beta=1$.
In view of Subsections \ref{subsec-dualform} and \ref{subsec-reformulation2}, 
it will be sufficient to show that for $u\in\Lspaces$,
\begin{equation}
\Perop_2 u=\Perop_2\Jop_1 u=0
\quad\Longrightarrow\quad u=0.
\label{eq-impli1.1'}
\end{equation}
Here, we recall the notation
\[
\Lspaces:=L^1_0(\R)+\Hop L^1_0(\R)\subset\Lspaceo.
\]
So, we assume that $u\in\Lspaces$ has $\Perop_2 u=\Perop_2\Jop_1 u=0$.
The distribution $u$ has a decomposition $u=f+\Hop g$, where $f,g\in L^1_0(\R)$.
We write
\begin{equation}
f^I(t)=\frac12(f(t)+f(-t)),\quad f^{II}(t)=\frac12(f(-t)-f(t)),
\label{eq-fIfII}
\end{equation}
and
\begin{equation}
g^I(t)=\frac12(g(t)+g(-t)),\quad g^{II}(t)=\frac12(g(-t)-g(t)),
\label{eq-gIgII}
\end{equation}
so that the functions $f^I,g^I\in L^1_0(\R)$ are even while
$f^{II},g^{II}\in L^1_0(\R)$ are odd. We then put
\[
u^I=f^I-\Hop g^{II},\quad u^{II}=f^{II}-\Hop g^{I},
\]
so that $u^I\in\Lspaces$ is an even distribution, while $u^{II}$ is odd.
This is so because the Hilbert transform is symmetry reversing, odd is mapped
to even, and even to odd.

\medskip

\noindent{\sc Step I}: \emph{We first prove that the implication
\eqref{eq-impli1.1'} holds for odd $u$, that is, when $u=-u^{II}$}:
\begin{equation}
\Perop_2 u^{II}=\Perop_2\Jop_1 u^{II}=0
\quad\Longleftrightarrow\quad u^{II}=0.
\label{eq-impli1.1.odd}
\end{equation}
The added arrow to the left is of course a trivial implication.
Let $u^{II}_0:=\Rop_1 u^{II}\in\LspaceIone$ and $u^{II}_1:=
\Rop_1^\dagger u^{II}\in\LspaceIonecompl$
denote the restrictions of the distribution $u^{II}$ to the symmetric interval
$I_1$ and to the complement $\R\setminus\bar I_1$, respectively. Clearly,
$u^{II}_0$ and $u^{II}_1$ are odd, because $u^{II}$ is.
\emph{We will be done with this step once we are able to show that
$u^{II}_0=0$}, because then $u_1^{II}$ vanishes as well, as a result of
Proposition \ref{prop-5.4.1}:
\[
u^{II}_1=-\Rop_1^\dagger\Jop_1\Tope_1 u^{II}_0=0.
\]
Indeed, we have Proposition \ref{prop-5.4.2}, which tells us that
$u^{II}_0=\Rop_1 u^{II}=0$ and $u^{II}_1=\Rop_1^\dagger u^{II}=0$ together imply
that $u^{II}=0$.
Finally, to obtain that $u^{II}_0=0$, we observe that in addition, Proposition
\ref{prop-5.4.1} says that the odd distribution $u^{II}_0\in\LspaceIone$
has the important property $u^{II}_0=\Tope_1^2 u^{II}_0$. By iteration, then,
we have $u^{II}_0=\Tope_1^{2n} u^{II}_0$ for $n=1,2,3,\ldots$, and by letting
$n\to+\infty$, we realize from Theorem \ref{thm-basic1.002} that $u_0=0$ is the
only possible solution in $\LspaceIone$.

\medskip

\noindent{\sc Step II}: \emph{We now prove, based on Step I, that the
implication \eqref{eq-impli1.1'} holds for an arbitrary distribution
$u\in\Lspaces$, regardless of symmetry}.
So, we take a distribution $u\in\Lspaces$ for which $\Perop_2 u=0$
and $\Perop_2\Jop_1 u=0$. We split $u=u^I-u^{II}$ as above, and note that since
the operators $\Perop_2$ and $\Jop_1$ both respect odd-even symmetry,
\[
0=\Perop_2 u=\Perop_2 u^I-\Perop_2 u^{II}\quad\text{and}\quad
0=\Perop_2\Jop_1 u=\Perop_2\Jop_1 u^I-\Perop_2\Jop_1 u^{II}
\]
correspond to the splitting of the $0$ distribution into odd-even parts inside
the space
\[
\Lspacesper:=L^1_0(\R/2\Z)+\Hop_2 L^1_0(\R/2\Z)\subset\Lspaceoper.
\]
This means that each part must vanish separately, that is,
\begin{equation}
\Perop_2 u^I=0,\quad \Perop_2\Jop_1 u^I=0,\quad \Perop_2 u^{II}=0,\quad
\Perop_2\Jop_1 u^{II}=0.
\label{eq-symm1.001.01}
\end{equation}
By Step I, we know that the implication \eqref{eq-symm1.001.01} holds for the
odd distribution $u^{II}$, so it is an immediate consequence of
\eqref{eq-symm1.001.01} that $u^{II}=0$. We need to understand the result
obtained in Step I better, and write the equivalence \eqref{eq-impli1.1.odd}
in terms of the functions $f^{II}$ and $g^{I}$:
\begin{equation}
f^{II}=\Hop g^{I}\,\,\Longleftrightarrow\,\,
\begin{cases}
\Perop_2 f^{II}=\Perop_2\Hop g^{I},
\\
\Perop_2\Jop_1 f^{II}=\Perop_2\Hop\Jop_1 g^{I}.
\end{cases}
\label{eq-symm1.001.02}
\end{equation}
Next, since we know that $\Perop_2\Hop=\Hop_2\Perop_2$ as operators on
$L^1_0(\R)$, we may rewrite \eqref{eq-symm1.001.02} as
\begin{equation}
f^{II}=\Hop g^{I}\Longleftrightarrow
\begin{cases}
\Perop_2 f^{II}=\Hop_2\Perop_2 g^{I},
\\
\Perop_2\Jop_1 f^{II}=\Hop_2\Perop_2\Jop_1 g^{I}.
\end{cases}
\label{eq-symm1.001.1}
\end{equation}
\emph{Since we already know that $u^{II}=0$, it remains to explain why
$u^{I}=0$ must hold as well}. The relation \eqref{eq-symm1.001.01} also
contains the conditions
$\Perop_2 u^{I}=\Perop_2\Jop_1 u^{I}=0$, which in terms of $f^I$ and $g^{II}$
amount to having
\[
\begin{cases}
\Perop_2 f^{I}=\Hop_2\Perop_2 g^{II},
\\
\Perop_2\Jop_1 f^{I}=\Hop_2\Perop_2\Jop_1 g^{II}.
\end{cases}
\]
Let us apply the periodic Hilbert transform $\Hop_2$ to the left-hand and
right-hand sides, which is an invertible transformation on $\Lspacesper$ with
$\Hop_2^2=-\id$. The result is
\[
\begin{cases}
\Hop_2\Perop_2 f^{I}=-\Perop_2 g^{II},
\\
\Hop_2\Perop_2\Jop_1 f^{I}=-\Perop_2\Jop_1 g^{II}.
\end{cases}
\]
But this places us in the setting of \eqref{eq-symm1.001.1}, only with $-g^{II}$
in place of $f^{II}$, and $f^{I}$ in place of $g^{I}$. So we get from
\eqref{eq-symm1.001.1}
that $-g^{II}=\Hop f^{I}$, which after application of $\Hop$ reads
$f^{I}=\Hop g^{II}$. This means that $u^{I}=f^{I}-\Hop g^{II}=0$, as desired.
Finally, since both $u^I$ and $u^{II}$ vanish, we obtain $u=u^{I}-u^{II}=0$.
This proves that the implication \eqref{eq-impli1.1'} holds for every
$u\in\Lspaces$, which completes the proof.
\end{proof}

\subsection{The proof of asymptotic decay for $\beta=1$}
\label{subsec-setup}

We now proceed with the proof of Theorem \ref{thm-basic1.002}. As in the proof
of Theorem \ref{thm-basic1.001}, we have to be particularly careful because the
operator $\Tope_1:\LspaceIone\to\LspaceIone$ has norm $>1$.
However, it clearly acts contractively on $L^1(I_1)$.

\begin{proof}[Proof of Theorem \ref{thm-basic1.002}]
Since $u_0\in\LspaceIone$, we know that there exist functions $f\in L^1(\R)$
and $g\in L^1_0(\R)$ such that $u_0=\Rop_1(f+\Hop g)$.

\medskip

\noindent{\sc Step I}: \emph{We find a suitable odd extension of $u_0$ to
all of $\R$.}
Let the functions
$f^I,f^{II},g^I,g^{II}$ be given by \eqref{eq-fIfII} and \eqref{eq-gIgII},
and put
\[
u^I=f^I-\Hop g^{II},\quad u^{II}=f^{II}-\Hop g^{I},
\]
so that $u^I\in\Lspaceo$ is an even distribution, while $u^{II}\in\Lspaceo$
is odd. The way things are set up, we have that $u_0=\Rop_1 u$, where
$u:=u^I-u^{II}$. Since it is given that $u_0$ is odd, we must have that
$\Rop_1u^I=0$, and that $u_0=-\Rop_1 u^{II}$. The distribution $-u^{II}$ is odd
on all of $\R$, and provides an extension of $u_0$ beyond the interval $I_1$.
\emph{We will focus our attention on the odd distribution
$u^{II}=f^{II}-\Hop g^{I}$, which has $\Rop_1 u^{II}=-u_0$}.

\medskip

\noindent{\sc Step II}: \emph{We now argue that without loss of generality,
we may require of the even function $g^I\in L^1_0(\R)$ that in addition }
\begin{equation}
\langle 1,g^I\rangle_{I_1}=\langle 1,g^I\rangle_{\R\setminus I_1}=0.
\label{eq-simplifyingass1}
\end{equation}
To this end, we consider the even function $h^I\in L^1_0(\R)$ given by
\[
h^I(x):=1_{I_1}(x)-x^{-2}1_{\R\setminus I_1}(x),\qquad x\in\R,
\]
with Hilbert transform
\[
\Hop h^I(x)=\frac{1}{\pi}\log\bigg|\frac{x+1}{x-1}\bigg|+\frac{1}{\pi x^2}
\log\bigg|\frac{x+1}{x-1}\bigg|-\frac{2}{\pi x},\qquad x\in\R,
\]
and note that $\Hop h^I\in L^1_0(\R)$ is odd. Now, if
\eqref{eq-simplifyingass1} is not fulfilled to begin with, then we
consider instead the functions
\[
F^{II}:=f^{II}+\frac12\langle g^I,1\rangle_{I_1}\Hop h^I,\quad
G^{I}:=g^{I}-\frac12\langle g^{I},1\rangle_{I_1}h^I.
\]
Indeed, we see that $F^{II},G^I\in L^1_0(\R)$ where $F^{II}$ is odd
and $G^I$ is even, that \eqref{eq-simplifyingass1} holds with
$G^I$ in place of $g^I$, and that $u^{II}=f^{II}-g^{I}=F^{II}-\Hop G^I$.
\emph{This allows us assume that $f^{II},g^{I}\in L^1_0(\R)$ are chosen
so that \eqref{eq-simplifyingass1} holds}, and completes the proof of
Step II.

\medskip

\noindent{\sc Step III}: \emph{Splitting of the functions $f^{II}$ and
$g^{I}$ according to intervals}.
We split $f=f_1+f_2$ and $g=g_1+g_2$, where
\[
f_1^{II}:=f^{II}1_{I_1}\in L^1_0(I_1),\quad f_2^{II}:=f^{II}1_{\R\setminus I_1}
\in L^1_0(\R\setminus I_1),
\]
and
\[
g_1:=g1_{I_1}\in L^1_0(I_1),\quad
g_2:=g1_{\R\setminus I_1}\in L^1_0(\R\setminus I_1).
\]
Here, we used in fact Step II. Note that the functions $f_1^{II},f_2^{II}$
are odd, while $g_1^I,g_2^I$ are even.
We write $u^{II}_0:=\Rop_1 u^{II}$, so that $u^{II}_0=-u_0$.
We note that
\begin{equation}
u_0^{II}=\Rop_1(f^{II}+\Hop g^{I})
=\Rop_1(f_1^{II}+f_2^{II}+\Hop g_1^I+\Hop g_2^I)=f_1^{II}+\Rop_1\Hop g_1^I
+\Rop_1\Hop g_2^I.
\label{eq-decomp.u0.001}
\end{equation}
By applying the operator $\Tope_1^N$ for $N=2,3,4,\ldots$ to the leftmost and
rightmost sides of \eqref{eq-decomp.u0.001}, we obtain
\begin{equation*}
\Tope_1^{N}u_0^{II}=\Tope_1^{N}f_1^{II}+\Tope_1^{N}\Rop_1\Hop g_1^I+
\Tope_1^{N}\Rop_1\Hop g_2^I.
%\label{eq-decomp.u0.002}
\end{equation*}
and after application of the ``valeur au point'' operation, this identity
reads, for $N=2,3,4,\ldots$,
\begin{equation}
\pev[\Tope_1^N u_0^{II}](x)=\Tope_1^{N}f_1^{II}(x)+
\pev[\Tope_1^{N}\Rop_1\Hop g_1^I](x)+
\pev[\Tope_1^{N}\Rop_1\Hop g_2^I](x),\qquad \almostev\,\, x\in I_1.
\label{eq-decomp.u0.002}
\end{equation}
Next, by Propositions \ref{prop-ergodicity1.01}(v) and \ref{prop-5.6.5'},
we have for each fixed $\eta$, $0<\eta<1$, the convergence
\begin{equation}
\Tope_1^{N}f_1^{II}\to0\quad\text{and}\quad
1_{I_\eta}
\pev[\Tope_1^{N}\Rop_1\Hop g_2^I]\to0,
\label{eq-decomp.u0.003}
\end{equation}
the first one in the norm of $L^1(I_1)$
as $N\to+\infty$. That is, two terms on the right-hand side
of \eqref{eq-decomp.u0.002} vanish in the limit on compact subintervals,
and we are left to analyze the remaining middle term.

By rearranging the terms in the finite expansion of Proposition
\ref{prop-5.7.3} with $n:=N$, applied to the even function $g_1^I\in L^1_0(I_1)$
in place of $f$, we obtain that
\begin{multline}
\Tope_1^{N}\Rop_1\Hop g_1^I=\Rop_1\Hop\Topep_1^{N}g_1^I
-\Tope_1^{N-1}\Rop_1\Hop\Jop_1 g_1^I
+\Tope_1\Rop_1\Hop\Jop_1\Topep_1^{N}g_1^I
\\
+\sum_{j=1}^{N-1}(\Tope_1^2-\id)\Tope_1^{N-j-1}\Rop_1\Hop\Jop_1\Topep_1^jg_1^I,
\label{eq-decomp.u0.004}
\end{multline}
Here, of course, $\Topep_1=\Tope_1$ as operators, but we keep writing
$\Topep_1^{m}g_1^I$ to emphasize that the function is extended to vanish off
the interval $I_1$, this is important because the Hilbert transform is
nonlocal.
Since we know that $g^I\in L^1_0(I_1)$,  Proposition 
\ref{prop-ergodicity1.01}(v)
tells us that $\Topep_1^N g_1^I=\Tope_1^Ng_1^I\to0$ in norm in $L^1(I_1)$ as
$N\to+\infty$,
so that
\begin{equation}
\pev[\Rop_1\Hop\Topep_1^{N}g_1^I]\to0\quad \text{and}\quad
\pev[\Tope_1\Rop_1\Hop\Jop_1\Topep_1^{N}g_1^I]\to0,
\label{eq-est.term1.9901}
\end{equation}
in $L^{1,\infty}(I_1)$ as $N\to+\infty$.
Moreover, by Proposition \ref{prop-5.6.5'},
$\Tope_1^{N-1}\Rop_1\Jop_1 g_1^I\to0$ as $n\to+\infty$, uniformly on compact
subsets of $I_1$, so that in particular,
\begin{equation}
1_{I_\eta}\pev\Tope_1^{N-1}\Rop_1\Jop_1 g_1^I\to0
\label{eq-est.term1.9902}
\end{equation}
in $L^1(I_1)$, for each fixed $\eta$, $0<\eta<1$.
We realize from \eqref{eq-est.term1.9901} and \eqref{eq-est.term1.9902} that
as $N\to+\infty$, the first three terms on the right-hand side of
fade away and we are left to analyze the expression with the summation sign.

\medskip

\noindent{\sc Step IV}: \emph{Application of kernel techniques}.
As Subsection \ref{subsec-smoothH}, we may write
\begin{equation}
\pev\,\Hop\Jop_1\Topep_1^jg_1^I(x)=\frac{1}{\pi}\int_{-1}^{1}\Kfun_1(t,x)\,
\Topep_1^jg_1^I(t)\diff t,\qquad x\in I_1,
\label{eq-decomp.u0.005}
\end{equation}
where $\Kfun_1(t,x)=t/(1+tx)$.
We recall the odd-even decomposition of $\Kfun_1(t,x)$:
%Moreover, since $g_1^I$ is odd, and $\Tope_1$ preserves oddness, and
\[
\Kfun_1(t,x)=\Kfun_1^I(t,x)-\Kfun_1^{II}(t,x),
\quad\text{where}\quad \Kfun_1^{I}(t,x)=\frac{t}{1-x^2t^2},\quad
\Kfun_1^{II}(t,x)=\frac{t^2x}{1-t^2x^2}.
\]
As, by inspection, the kernel $t\mapsto \Kfun_1^{I}(t,x)$ is odd,
and since the function $\Topep_1^jg_1^I$ is even, we may rewrite
\eqref{eq-decomp.u0.005} in the form
\begin{equation}
\Hop\Jop_1\Topep_1^jg_1^I(x)=-\frac{2}{\pi}\int_{0}^{1}\Kfun_1^{II}(t,x)\,
\Topep_1^jg_1^I(t)\diff t,\qquad x\in I_1,
\label{eq-decomp.u0.005.1}
\end{equation}
Using \eqref{eq-decomp.u0.005.1}, we may rewrite the expression with the
summation sign in \eqref{eq-decomp.u0.004}:
\begin{multline}
\sum_{j=1}^{N-1}(\Tope_1^2-\id)\Tope_1^{N-j-1}\Rop_1\Hop\Jop_1\Topep_1^jg_1^I(x)
%\\
=-\frac{2}{\pi}\int_{0}^{1}
\sum_{j=1}^{N-1}\Tope_1^{N-j-1}(\Tope_1^2-\id)\Kfun_1^{II}(t,\cdot)(x)\,
\Topep_1^jg_1^I(t)\diff t
\\
=\frac{2}{\pi}\sum_{j=1}^{N-1}\int_{0}^{1}
(\Tope_1^{N-j-1}+\Tope_1^{N-j})\kfun_1^{II}(t,\cdot)(x)\,
\Topep_1^jg_1^I(t)\diff t.
\label{eq-decomp.u0.006}
\end{multline}
The expression \eqref{eq-decomp.u0.006} is an odd function of $x$, so we
need only estimate it in the righthand interval $I_1^+=[0,1[$.
By appealing to the fundamental estimate of Theorem \ref{th-12.1.2},
we may obtain a pointwise estimate in \eqref{eq-decomp.u0.006}, for
$x\in I_1^+$, as follows:
\begin{multline}
\bigg|\pev\sum_{j=1}^{N-1}(\Tope_1^2-\id)\Tope_1^{N-j-1}\Rop_1\Hop
\Jop_1\Topep_1^jg_1^I(x)\bigg|
\\
\le\frac{2}{\pi}\sum_{j=1}^{N-1}\int_{0}^{1}
\big|(\Tope_1^{N-j-1}+\Tope_1^{N-j})\kfun_1^{II}(t,\cdot)(x)\,
\Topep_1^jg_1^I(t)\big|\diff t
\\
\le
\frac{1}{\pi}\sum_{j=1}^{N-1}
(\Tope_1^{N-j-1}+\Tope_1^{N-j})\kfun_1^{II}(1,\cdot)(x)\,
\|\Topep_1^j g_1^I\|_{L^1(I_1)}.
\label{eq-decomp.u0.006.01}
\end{multline}
As observed previously, since $g_1^I\in L^1_0(I_1)$,
Proposition \ref{prop-ergodicity1.01}(v) tells us that
$\Topep_1^jg_1^I=\Tope_1^jg_1^I\to0$ in norm in $L^1(I_1)$ as $j\to+\infty$.
It follows that we may, for a given positive real $\epsilon$, find a
positive integer $n_0=n_0(\epsilon)$ such that
$\|\Topep_1^jg_1^I\|_{L^1(I_1)}\le\epsilon$ for
$j\ge n_0(\epsilon)$. We split the summation in \eqref{eq-decomp.u0.006.01}
accordingly, for $N>n_0(\epsilon)$, and use that the transfer operator
$\Topep_1=\Tope_1$ is a contraction on $L^1(I_1)$:
\begin{multline}
\bigg|\pev\sum_{j=1}^{N-1}(\Tope_1^2-\id)\Tope_1^{N-j-1}\Rop_1\Hop
\Jop_1\Topep_1^jg_1^I(x)\bigg|
\\
\le \|g_1^I\|_{L^1(I_1)}
\frac{1}{\pi}\sum_{j=1}^{n_0(\epsilon)-1}
(\Tope_1^{N-j-1}+\Tope_1^{N-j})\kfun_1^{II}(1,\cdot)(x)
+\frac{\epsilon}{\pi} \sum_{j=n_0(\epsilon)}^{N-1}
(\Tope_1^{N-j-1}+\Tope_1^{N-j})\kfun_1^{II}(1,\cdot)(x),
\label{eq-decomp.u0.006.02}
\end{multline}
where again $x\in I_1^+$ is assumed. The odd part of the dynamically reduced
Hilbert kernel $x\mapsto \kfun_1^{II}(1,x)$ is odd and smooth on $\bar I_1$, so
Proposition \ref{prop-ergodicity1.01}(v)
tells us that for fixed $\epsilon$,
\[
\sum_{j=1}^{n_0(\epsilon)-1}
\big\|(\Tope_1^{N-j-1}+\Tope_1^{N-j})\kfun_1^{II}(1,\cdot)\big\|_{L^1(I_1)}\to0
\quad \text{as}\quad N\to+\infty.
\]
As for the second sum on the right-hand side of \eqref{eq-decomp.u0.006.02},
we use finite Neumann series summation \eqref{eqlem-Kt1.5} together
with Lemma \ref{lem-symmetry2} to obtain that
\[
\sum_{j=n_0(\epsilon)}^{N-1}
(\Tope_1^{N-j-1}+\Tope_1^{N-j})\kfun_1^{II}(1,\cdot)(x)
\le (\id+\Tope_1)\Kfun_1^{II}(1,\cdot)(x)\le\frac{2}{1-x^2},\qquad x\in I_1^+.
\]
Note that in the last step, we compared $\Kfun_1^{II}(1,x)$ with the invariant
density $\kappa_1(x)=(1-x^2)^{-1}$.
It now follows from the estimate \eqref{eq-decomp.u0.006.02} and symmetry
that for fixed $\eta$ with $0<\eta<1$,
\begin{equation}
\limsup_{N\to+\infty}
\bigg\|1_{I_\eta}\sum_{j=1}^{N-1}(\Tope_1^2-\id)\Tope_1^{N-j-1}\Rop_1\Hop
\Jop_1\Topep_1^jg_1^I(x)\bigg\|_{L^1(I_1)}
\le \frac{4\epsilon}{\pi}\log\frac{1+\eta}{1-\eta}.
\label{eq-decomp.u0.006.03}
\end{equation}
As we are free to let $\epsilon$ be as close to $0$ as we desire, it follows
that for fixed $\eta$ with $0<\eta<1$,
\begin{equation}
\limsup_{N\to+\infty}
\bigg\|1_{I_\eta}\sum_{j=1}^{N-1}(\Tope_1^2-\id)\Tope_1^{N-j-1}\Rop_1\Hop
\Jop_1\Topep_1^jg_1^I(x)\bigg\|_{L^1(I_1)}
=0.
\label{eq-decomp.u0.006.04}
\end{equation}
This means that also the last term on the right-hand side of
\eqref{eq-decomp.u0.004} tends to $0$ in the mean on all compact subintervals.
Putting things together in the context of the decomposition
\eqref{eq-decomp.u0.002}, we see from the convergences
\eqref{eq-decomp.u0.003} and
the further decomposition \eqref{eq-decomp.u0.004}, together with
the associated convergences \eqref{eq-est.term1.9901},
\eqref{eq-est.term1.9902}, and \eqref{eq-decomp.u0.006.03}, that
$1_{I_\eta}\pev[\Tope_1^N u_0^{II}]\to0$ in $L^{1,\infty}(I_1)$, which is the
claimed assertion, because $u_0=-u_0^{II}$.
The proof is complete.
\end{proof}

\end{document}